\documentclass[twoside,centertags]{amsart}

\usepackage[cmtip,color,all]{xy}
\usepackage{amsmath}
\usepackage{amssymb}
\usepackage{eucal}
\usepackage{graphicx}
\usepackage{mathrsfs}
\usepackage{color}

\usepackage{hyperref}

\newtheorem{thm}[equation]{Theorem}
\newtheorem{cor}[equation]{Corollary}
\newtheorem{lem}[equation]{Lemma}
\newtheorem{prop}[equation]{Proposition}
\theoremstyle{definition}
\newtheorem{defn}[equation]{Definition}
\newtheorem*{defnnonumber}{Definition}
\theoremstyle{remark}
\newtheorem{rem}[equation]{Remark}
\newtheorem{exm}[equation]{Example}

\def\op{\mathrm{op}}

\def\card{\operatorname{card}}

\def\r{\rightarrow} 
\def\rr{\Rightarrow} 

\def\into{\rightarrowtail}
\def\onto{\twoheadrightarrow}

\newcommand{\Mod}[2]{\mathrm{Mod}_{#1}(#2)}

\newcommand{\aro}[1]{$\operatorname{ARO}_{#1}$}
\newcommand{\arm}[1]{$\operatorname{ARM}_{#1}$}

\def\hom{\operatorname{Hom}}

\def\ho{\operatorname{Ho}}

\def\ext{\operatorname{Ext}}
\def\pext{\operatorname{PExt}}

\def\pd{\operatorname{pd}}
\def\Pgldim{\operatorname{pgd}}
\def\fgd{\operatorname{fgd}}

\def\Sp{\operatorname{Sp}}
\def\D{D}
\def\Hom{\operatorname{Hom}}
\def\Ppd{\operatorname{ppd}}
\def\K{K}

\def\Sh{\operatorname{Sh}}

\newcommand{\SH}{\operatorname{SH}}

\def\Hocolim{\mathop{\operatorname{Hocolim}}}

\newcommand{\C}[1]{\mathscr{#1}}  

\newcommand{\id}[1]{\operatorname{id}_{#1}} 

\def\st{\stackrel} 

\def\To{\longrightarrow}

\newcommand{\add}{\operatorname{Add}\,}

\newcommand{\Ab}{\operatorname{Ab}}

\def\colim{\mathop{\operatorname{colim}}}

\def\coker{\operatorname{Coker}}
\renewcommand{\ker}{\operatorname{Ker}}
\newcommand{\im}{\operatorname{Im}}

\numberwithin{equation}{section}

\newdir{ >}{{}*!/-7pt/@{>}}
\newdir{> }{{}*!/10pt/@{>}}
\newdir{>> }{{}*!/10pt/@{>>}}

\begin{document}

\title{Transfinite Adams representability}%
\author{Fernando Muro}%
\address{Universidad de Sevilla,
Facultad de Matem\'aticas,
Departamento de \'Algebra,
Avda. Reina Mercedes s/n,
41012 Sevilla, Spain}
\email{fmuro@us.es}
\author{Oriol Ravent\'os}%
\address{Fakult\"{a}t f\"{u}r Mathematik,
Universit\"{a}t Regensburg,
93040 Regensburg,
Germany}
\email{oriol.raventos-morera@ur.de}
\date{\today}

\thanks{Both authors were partially supported
by the Spanish Ministry
of Economy and Competitiveness under the grants  MTM2010-15831 and MTM2013-42178-P, and the Government of Catalonia under the grant SGR-119-2009. The first author also acknowledges the support of the Andalusian Ministry of Economy, Innovation and Science under the grant FQM-5713. The second author was also supported by the grant SFB 1085 ``Higher invariants'' founded by the German Research Foundation and the project CZ.1.07/2.3.00/20.0003
of the Operational Programme Education for Competitiveness of the Ministry
of Education, Youth and Sports of the Czech Republic.}

\begin{abstract}
We consider the following problems in a well generated triangulated category $\C T$.  Let $\alpha$ be a regular cardinal and $\C T^{\alpha}\subset \C T$ the full subcategory of $\alpha$\nobreakdash-compact objects. Is every functor $H\colon (\C T^{\alpha})^{\op}\r\operatorname{Ab}$ that preserves products of $<\alpha$ objects and takes exact triangles to exact sequences of the form $H\cong\C T(-,X)_{|_{\C T^{\alpha}}}$ for some $X$ in $\C T$? Is every natural transformation
$\tau\colon \C{T}(-,X)_{|_{\C{T}^{\alpha}}}\r \C{T}(-,Y)_{|_{\C{T}^{\alpha}}}$ of the form $\tau=\C{T}(-,f)_{|_{\C{T}^{\alpha}}}$
for some $f\colon X\rightarrow Y$ in $\C {T}$? If the answer to both questions is positive we say that $\C T$ satisfies $\alpha$\nobreakdash-Adams representability. A classical result going back to Brown and Adams shows that the stable homotopy category satisfies $\aleph_{0}$\nobreakdash-Adams representability. The case $\alpha=\aleph_{0}$ is well understood thanks to the work of Christensen, Keller, and Neeman. In this paper we develop an obstruction theory to decide whether $\C T$ satisfies $\alpha$\nobreakdash-Adams representability. We derive necessary and sufficient conditions of homological nature, and we compute several examples. In particular, we show that there are rings satisfying $\alpha$\nobreakdash-Adams representability  for all $\alpha\geq\aleph_{0}$ and rings which do not satisfy $\alpha$\nobreakdash-Adams representability  for any $\alpha\geq\aleph_{0}$. Moreover, we exhibit rings for which the answer to both questions is~no for all $\aleph_\omega>\alpha\geq\aleph_{2}$.
\end{abstract}

\maketitle
\tableofcontents

\section*{Introduction}

There are two classical representability theorems in the stable homotopy category~$\C T$. Any spectrum $X$ gives rise to a cohomology theory \mbox{$\C T(-,X)\colon \C T^{\op}\r\operatorname{Ab}$}. The \emph{Brown representability} theorem, \cite{Br62}, says that any cohomology theory \mbox{$H\colon \C T^{\op}\r\operatorname{Ab}$} is of the form $H\cong\C T(-,X)$ for some spectrum $X$. The \emph{Adams representability} theorem, \cite{Ad71}, is a kind of analogue for cohomology theories defined only on the full subcategory of compact spectra $\C T^{c}\subset \C T$. It asserts that any cohomology theory \mbox{$H\colon (\C T^{c})^{\op}\r\operatorname{Ab}$} is of the form $H=\C T(-,X)_{|_{\C T^{c}}}$ for some  $X$, and, moreover, any natural transformation
\begin{equation*}
\tau\colon\C T(-,X)_{|_{\C T^{c}}}\To \C T(-,Y)_{|_{\C T^{c}}}
\end{equation*}
is induced by a map $f\colon X\r Y$, $\tau=\C T(-,f)_{|_{\C T^{c}}}$. By Yoneda's lemma, the representing spectrum in Brown's theorem is unique and any natural transformation between cohomology theories on $\C T$ comes from a unique map between the representing
spectra. In Adams' theorem the spectrum $X$ is still unique, but there may be different maps $f$ representing a given natural transformation $\tau$. Maps representing the trivial natural transformation are called \emph{phantoms}. Brown proved Adams' theorem under the restrictive hypothesis that the cohomology theory $H$ takes values in countable abelian groups. Adams' theorem allows to extend cohomology theories which are, in principle, only defined for compact spectra like topological $K$\nobreakdash-theory  defined in terms of vector bundles. Adams' theorem is stronger than Brown's, cf.~\cite{Ad71}, and it also implies the representability of homology theories via the Spanier--Whitehead duality.

The analogue of Brown's representability theorem is satisfied by a wide class of triangulated categories $\C T$ including the \emph{well generated} ones, i.e.~if $\C T$ is well generated any  functor
\mbox{$H\colon\C{T}^\op\r \operatorname{Ab}$}
preserving products and taking exact triangles to exact sequences is of the form $H=\C{T}(-,X)$ for some $X$ in $\C T$ \cite[Theorem~8.3.3]{triang}.
The simplest examples of well generated categories are the compactly generated ones. An
object $C$ in $\C T$ is \emph{compact} if the functor $\C T(C,-)$ preserves direct sums, and  $\C T$ is \emph{compactly generated} if it has coproducts and the full subcategory of compact objects $\C T^{c}$ is essentially small and generates $\C T$, i.e.~an object $X$ in $\C T$ is trivial if and only if $\C T(C,X)=0$ for all $C$ in $\C T^{c}$. A compactly generated category $\C T$ satisfies \emph{Adams representability} if any additive functor
\mbox{$H\colon (\C T^{c})^{\op}\r\operatorname{Ab}$} taking exact triangles to exact sequences is of the form $H=\C T(-,X)_{|_{\C T^{c}}}$ for some  $X$ in $\C T$, and any natural transformation as $\tau$ above
is induced by a map $f\colon X\r Y$, $\tau=\C T(-,f)_{|_{\C T^{c}}}$. Despite the fact that the category of compact objects contains much information about the whole category, Adams representability is seldom satisfied.  It is satisfied, for instance,  when $\C{T}^{c}$ is essentially countable \cite{Ne97}. This covers the stable homotopy category, but not the derived category $D(R)$ of a ring $R$ unless $R$ is countable. Adams representability is thoroughly studied in \cite{Be00} and \cite{CKN01}, with emphasis on derived categories of rings. It turns out to be strongly related to the pure global dimension of the ring $R$, a homological invariant connected to set theory, e.g.~the first part of Adams representability for the derived category $D(\mathbb C \langle x,y\rangle)$ of a non-commutative polynomial ring on two variables over the complex numbers is equivalent to the continuum hypothesis.

Many well generated triangulated categories have not enough compact objects to generate,  e.g.~the homotopy category $K(\operatorname{Proj-}R)$ of complexes of projective right
$R$\nobreakdash-modules over some rings $R$ which are not right coherent \cite[Example 7.16]{Ne08}. There are even some well generated categories with no non-trivial compact objects at all, e.g.~the derived category $D(\Sh\!/M)$ of sheaves of abelian groups on a connected non-compact
paracompact manifold $M$ of $\dim M\geq 1$ \cite{Ne01b}. Therefore, in these contexts, Adams representability does not make much sense as considered above. In such cases, the role of compact objects is played by \emph{$\alpha$\nobreakdash-compact objects} for a regular cardinal $\alpha$. In a well generated category, for a large enough cardinal $\alpha$, the category $\C T^{\alpha}$ of $\alpha$\nobreakdash-compact objects is essentially small, closed under coproducts of $<\alpha$ objects, and generates $\C T$.
In this paper, we consider the following transfinite analogue of Adams representability in~$\C T$.

\begin{defnnonumber}
Let $\alpha$ be a regular cardinal and $\C{T}$ a well generated triangulated category. A functor $H\colon(\C{T}^{\alpha})^\op\r \operatorname{Ab}$ is \emph{cohomological} if it takes exact triangles to exact sequences. We say that $\C T$ satisfies \emph{$\alpha$\nobreakdash-Adams representability} if the following two properties are satisfied:
\begin{itemize}
 \item[\aro{\alpha}] Any cohomological functor \mbox{$H\colon(\C{T}^{\alpha})^\op\r \operatorname{Ab}$} that preserves products of $<\alpha$ objects is isomorphic to $\C{T}(-,X)_{|_{\C{T}^{\alpha}}}$ for some $X$ in $\C{T}$.
\item[\arm{\alpha}] Any natural transformation
\mbox{$\tau\colon \C{T}(-,X)_{|_{\C{T}^{\alpha}}}\r \C{T}(-,Y)_{|_{\C{T}^{\alpha}}}$}
is induced by a morphism \mbox{$f\colon X\rightarrow Y$} in $\C {T}$, $\tau=\C{T}(-,f)_{|_{\C{T}^{\alpha}}}.$
\end{itemize}
\end{defnnonumber}

The only case where these properties hold for obvious reasons for all $\alpha$ is the derived category $D(k)$ of a field $k$, since it is equivalent to the category of $\mathbb Z$-graded $k$-vector spaces. Observe that if $\C{T}$ is compactly generated $\aleph_0$\nobreakdash-Adams representability is the same as Adams representability as considered above. Since \aro{\aleph_{0}} and \arm{\aleph_{0}} fail so often, it is also natural to consider \aro{\alpha} and \arm{\alpha} for $\alpha>\aleph_{0}$ in compactly generated categories.

For $\C T$ a well generated triangulated category with models, Rosick\'y stated in \cite{Ro05}  that \aro{\alpha} and \arm{\alpha} were satisfied for a proper class of regular cardinals $\alpha$. Unfortunately, his proof contains a gap acknowledged in \cite{Ro07} and \cite{Ro}. Nevertheless, this statement is a fairly natural question. Heuristically, since any well generated category is an increasing union of the subcategories of $\alpha$\nobreakdash-compact objects $\C{T}=\cup_{\alpha}\C{T}^{\alpha}$ by \cite[Proposition 8.4.2]{triang}, Brown representability can be  regarded as the limit of \aro{\alpha} and \arm{\alpha} as $\alpha$ runs over all cardinals, and this question suggests that the limit statement is satisfied because it is satisfied in a `cofinal' sequence.

Neeman obtained in \cite{Ne09} striking consequences of Rosick\'y's statement. One of them is that any covariant functor on a well generated triangulated category $H\colon \C T\r\operatorname{Ab}$ preserving products and taking exact triangles to exact sequences would be representable $H\cong\C T(X,-)$. This is \emph{Brown representability for the dual} $\C T^{\op}$. This result cannot be deduced from the Brown representability theorem for well generated categories since the opposite of a well generated category is never well generated. It was known for compactly generated triangulated categories, cf.~\cite{Ne98} and \cite{Kr02}, and it is a major open problem in the field for well generated categories.

In this paper, we show that some well generated triangulated categories do not satisfy $\alpha$\nobreakdash-Adams representability. For instance, we prove that $D(\mathbb{Z})$ satisfies \arm{\alpha} if and only if $\alpha=\aleph_{0}$. This uses the fact that the $\alpha$\nobreakdash-pure global dimension of $\mathbb Z$ is $\Pgldim_{\alpha}(\mathbb Z)>1$ for $\alpha>\aleph_{0}$, cf.~\cite{BG}. The \emph{$\alpha$\nobreakdash-pure global dimension} of a ring $R$ is the smallest $n$ such that, for each right $R$\nobreakdash-module $M$, there is a sequence
$$
0\r P_{n}\r \cdots\r P_{1}\r M\r 0
$$
where each $P_{i}$ is a retract of a direct sum of $\alpha$\nobreakdash-presentable right $R$\nobreakdash-modules, i.e.~with $<\alpha$ generators and relations \cite[Chapter~7]{JL89}, and
$$
0\r \hom_{R}(Q,P_{n})\r \cdots\r \hom_{R}(Q,P_{1})\r \hom_{R}(Q,M)\r 0
$$
is exact for any $\alpha$\nobreakdash-presentable right $R$\nobreakdash-module $Q$.

A ring $R$ is \emph{$\alpha$\nobreakdash-coherent} if the kernel of any morphism between $\alpha$-presentable $R$-modules is $\alpha$\nobreakdash-presentable. Rings of $\card R<\alpha$ are $\alpha$\nobreakdash-coherent since in this case the $\alpha$-presentable $R$-modules are simply the $R$-modules of cardinality $<\alpha$. We prove that, if $R$ is $\alpha$\nobreakdash-coherent for some $\alpha>\aleph_{0}$ and $D(R)$ satisfies \arm{\alpha}, then $\Pgldim_{\alpha}(R)\leq 1$.

A ring $R$ is \emph{hereditary} if it has global dimension $\leq 1$, e.g.~$R=\mathbb Z$, discrete valuation rings (DVRs), and path algebras of quivers over a field. Hereditary rings are $\alpha$\nobreakdash-coherent for all $\alpha\geq\aleph_0$. For hereditary rings, we prove that  \aro{\alpha} is equivalent to  $\Pgldim_{\alpha}(R)\leq 2$ and that \arm{\alpha} is equivalent to  $\Pgldim_{\alpha}(R)\leq 1$, $\alpha>\aleph_0$. The case $\alpha=\aleph_0$  was shown in \cite{CKN01}. As we already mentioned,  $\Pgldim_{\alpha}(\mathbb{Z})>1$ for all $\alpha>\aleph_0$. This property is shared by all DVRs \cite{BS}. The first examples of rings with $\Pgldim_{\alpha}(R)>1$ for all $\alpha\geq \aleph_{0}$ have been obtained by Bazzoni and \v{S}{\v{t}}ov{\'i}\v{c}ek in \cite{BS}, e.g.~$R=k[[x,y]]$ for $k$ a field. \v{S}{\v{t}}ov{\'i}\v{c}ek recently informed us that, in work in progress \cite{Stovicek}, he has obtained the sharper lower bound $\Pgldim_{\aleph_n}(\widehat{\mathbb Z}_p)\geq n+1$ for $n$ a finite
cardinal and $\widehat{\mathbb Z}_p$ the $p$\nobreakdash-adic integers (he is actually extending the result to arbitrary discrete valuation domains, and then he plans to carry it over to the Kronecker algebra and many other examples, as in \cite{BS}). Therefore $D(\widehat{\mathbb Z}_p)$ satisfies neither \arm{\aleph_n} nor \aro{\aleph_n} for any finite $n\geq 2$. This is the first known example of a triangulated category exhibiting this behaviour.

Under the continuum hypothesis, we prove that $\Pgldim_{\aleph_{1}}(\mathbb Z)=\Pgldim_{\aleph_{1}}(\widehat{\mathbb Z}_p)=2$, which implies \aro{\aleph_{1}} for $D(\mathbb{Z})$ and $D(\widehat{\mathbb Z}_p)$, and more generally, if $2^{\aleph_{n-1}}=\aleph_{n}$ then $\Pgldim_{\aleph_{n}}(\mathbb Z)= \Pgldim_{\aleph_{n}}(\widehat{\mathbb Z}_p)=n+1$. In this sense \v{S}{\v{t}}ov{\'i}\v{c}ek lower bounds are optimal. It would be interesting to find out whether these equalities can be obtained without the (generalized) continuum hypothesis.



We would like to remark that the only cardinals $\alpha$ for which we know examples of triangulated categories satisfying \aro{\alpha} but not \arm{\alpha} are $\alpha=\aleph_0,\aleph_1$. For $\alpha=\aleph_0$, the derived category of a tame hereditary algebra over an uncountable field is an example by \cite[Theorem 3.4]{BBL82}. For $\alpha=\aleph_1$, all examples we know depend on the continuum hypothesis, e.g.~$D(\mathbb{Z})$ and $D(\widehat{\mathbb Z}_p)$. It would also be interesting to know if we can dispense with this hypothesis. Property \arm{\alpha} implies \aro{\alpha} for $\alpha=\aleph_0$ \cite[Theorem 11.8]{Be00}. We do not know if these properties are related at all for uncountable $\alpha$ (Beligiannis's proof does not generalize), see Remark \ref{auslander}.

Concerning positive results, we show that the derived category $D(R)$  of a hereditary right pure-semisimple ring, e.g.~the path algebra of a Dynkin quiver over a field, satisfies \aro{\alpha} and \arm{\alpha} for all $\alpha$. Under the continuum hypothesis, we prove \aro{\aleph_{1}} for the following categories, where $R$ denotes a ring of $\card R\leq \aleph_{1}$:
the stable homotopy category, the derived category $D(R)$ of right $R$\nobreakdash-modules,  the homotopy category  $K(\operatorname{Proj-}R)$  of complexes of projective right
$R$\nobreakdash-modules, the homotopy category $K(\operatorname{Inj-}R)$   of complexes of injective right
$R$\nobreakdash-modules if  $R$ is right Noetherian,
 the derived category $D(\Sh\!/M)$ of sheaves of abelian groups on a connected
paracompact manifold, and the stable motivic homotopy category over a Noetherian scheme of finite Krull dimension that can be covered by spectra of  rings of cardinal $\leq \aleph_{1}$. We believe that set-theoretical assumptions are really necessary in these examples, as they are in order for $D(\mathbb C \langle x,y\rangle)$ to satisfy \aro{\aleph_0}. These results obtained under the continuum hypothesis suggest  that for any \emph{specific} cohomological functor $H\colon (\C T^{\aleph_{1}})^{\op}\r\operatorname{Ab}$ preserving countable products there are many chances to find an object $X$ with $H=\C{T}(-,X)_{|_{\C{T}^{\aleph_{1}}}}$, for if such an object did not exist the continuum hypothesis would be false.

We tackle \aro{\alpha} and \arm{\alpha} by means of a fairly general obstruction theory for triangulated categories. 
We consider a well generated  triangulated category $\C T$ and a full subcategory $\C{C}\subset \C T^{\alpha}$ closed under (de)suspensions and coproducts of $<\alpha$ objects which generates $\C T$. We do not require $\C C$ to be triangulated, although in this paper the main example is $\C C=\C T^{\alpha}$.
We consider the  \emph{restricted Yoneda functor},
$$
S_\alpha\colon \C{T}\To\Mod{\alpha}{\C C},\quad S_\alpha(X)=\C{T}(-,X)_{|_{\C{C}}},
$$
 where $\Mod{\alpha}{\C C}$ is the abelian category of \emph{$\alpha$\nobreakdash-continuous (right) $\C{C}$\nobreakdash-modules}, i.e.~functors $\C{C}^\op\rightarrow \operatorname{Ab}$ preserving products of $<\alpha$ objects. Morphisms in the kernel of $S_{\alpha}$ are called \emph{$\C C$\nobreakdash-phantom maps}. We interpolate the functor $S_\alpha$ by an inverse sequence of categories
$$
\C{T}\r\cdots\r \mathbf{Post}_{n+1}^{\simeq}\st{t_{n}}\To \mathbf{Post}_{n}^{\simeq}\r\cdots\r \mathbf{Post}_{0}^{\simeq}\st{\sim}\To\Mod{\alpha}{\C C}.
$$
For each step $t_{n}\colon\mathbf{Post}_{n+1}^{\simeq}\r \mathbf{Post}_{n}^{\simeq}$, we define obstructions to the lifting of objects and morphisms along $t_{n}$. Obstructions take values in $\ext$ groups in $\Mod{\alpha}{\C C}$. The obstructions for the lifting of objects were first considered in \cite{rmtc} for $\alpha=\aleph_{0}$. In addition, we prove that the induced functor
$$
t\colon\C{T}\To \mathbf{Post}_{\infty}^{\simeq}= \lim_{n}\mathbf{Post}_{n}^{\simeq}
$$
is full and essentially surjective. We also analyze the kernel of $t_{n}$ and, moreover, we show that the kernel of the functor $t$ is the ideal of \emph{$\infty$\nobreakdash-$\C C$\nobreakdash-phantom maps}, i.e.~maps $f\colon X\r Y$ in $\C T$ which decompose as a product $f=f_{n}\cdots f_{1}$
of $n$ $\C C$\nobreakdash-phantom maps $f_{i}$, $1\leq i\leq n$, for all $n\geq 1$. Furthermore, we prove that $\infty$\nobreakdash-$\C C$\nobreakdash-phantom maps form a square-zero ideal, i.e.~the composition of two $\infty$\nobreakdash-$\C C$\nobreakdash-phantom maps is always zero. This is a new result even for a compactly generated triangulated category $\C T$ and $\C C=\C T^{c}$.

\subsection*{Organization of the paper}
In Section \ref{ryf}, we fix some terminology and give an equivalent definition of $\alpha$-Adams representability using the restricted Yoneda functor. We start Section \ref{sec-obs} by summarizing  the formal properties of the obstruction theory developed in Section \ref{obstructiontheory}, see Theorem \ref{lololo}. Using this result we first derive  necessary and sufficient conditions for \arm{\alpha} in Corollary \ref{necesariaysuficiente0}, then a sufficient condition for \aro{\alpha} in Corollary \ref{suficiente}, and finally we prove a necessary and sufficient condition for $\alpha$\nobreakdash-Adams representability in Corollary \ref{necesariaysuficiente}.

In Section \ref{sec-for-rings} we study the special case of the derived category of a ring $D(R)$. For an $\alpha$\nobreakdash-coherent ring $R$, Theorem \ref{mainderived} gives a necessary condition so that $D(R)$ satisfies \arm{\alpha}. Moreover, if $R$ is hereditary it gives necessary and sufficient conditions for \arm{\alpha} and \aro{\alpha} in $D(R)$. We use this result, together with some existing computations on the $\alpha$\nobreakdash-pure global dimension of rings, to give explicit examples of triangulated categories that satisfy \aro{\aleph_1} but not \arm{\aleph_1} (assuming the continuum hypothesis), examples that do not satisfy \arm{\alpha} for any $\alpha$, and examples that satisfy neither \arm{\aleph_n} nor  \aro{\aleph_n} for any finite $n\geq 2$. In Section \ref{sec-aleph_1}, we provide more examples of triangulated categories satisfying \aro{\aleph_1} (always under the continuum hypothesis). 

Section \ref{sec-Rosicky-functor} is devoted to Rosick{\' y} functors in well generated triangulated categories. The hypothetical existence of such functors was used by Neeman \cite{Ne09} to prove Brown representability for the dual. We prove (Corollary \ref{Cor-Rosicky-func}) that the existence of Rosick{\' y} functors is equivalent to the fact that the restricted Yoneda functor $S_\alpha$ is itself a Rosick{\' y} functor for some regular cardinal $\alpha$ and an appropriate full subcategory $\C C$ of $\C T^{\alpha}$. We exhibit a very simple example of triangulated category $\C T$ that has a Rosick{\' y} functor, but in which $S_\alpha$ is never a Rosick{\' y} functor when $\C C=\C T^{\alpha}$, for any $\alpha$, see Remark \ref{Rem-Rosicky-func}.

We develop the obstruction theory in Section \ref{obstructiontheory}. This is our main tool for the study of transfinite Adams representability. This section is divided in eight subsections. In the first five subsections, we introduce Adams and Postnikov resolutions, Postnikov systems, and relate them to $\C C$\nobreakdash-phantom maps. As an application, we prove in Corollary \ref{phantom3} that the ideal of $\infty$\nobreakdash-$\C C$\nobreakdash-phantom maps is a square zero ideal. In the sixth and seventh subsections we define the obstruction theory. In the final subsection we explain the connection of our obstruction theory with the Adams spectral sequence.

In Section \ref{sec-first-obs} we show how to compute the first obstruction to the realizability of certain objects in algebraic triangulated categories (Theorem \ref{calculillo}). This is used in the proof of Theorem \ref{mainderived} to obtain examples of triangulated categories that do not satisfy \aro{\alpha}.

Finally, in Section \ref{sec-alphacpt} we give a characterization of $\alpha$\nobreakdash-compact objects in terms of the size of morphism sets which slightly improves \cite[Theorem
C]{Kr02}. This is used in Section \ref{sec-aleph_1} to provide examples of triangulated categories that satisfy \aro{\aleph_1}.

\subsection*{Acknowledgements}
This piece of research started as a result of conversations with Amnon Neeman on Rosick\'{y}'s \cite{Ro05}, we thank him for driving us to these problems. We are grateful to Carles Casacuberta for many useful suggestions on preliminary versions of this article. We would also like to acknowledge fruitful conversations an exchange of ideas with Javier Guti\'errez, Henning Krause, Ji\v{r}i Rosick\'{y}, Markus Spitzweck and Jan \v{S}{\v{t}}ov{\'i}\v{c}ek.

\section{The restricted Yoneda functor}\label{ryf}

Our main references for the part of triangulated category theory which is relevant for this paper are \cite{triang}, \cite{Kr01}, and \cite{Kr}. For the reader's convenience, we start this section by recalling some basic notions. In the next definition we follow Krause's approach \cite{Kr01}.

\begin{defn}\label{defbasica}
Let  $\alpha$ be a regular cardinal and
 $\C{T}$ a triangulated category with coproducts. An object $S$ in $\C T$ is
\emph{$\alpha$\nobreakdash-small} if
any morphism $S\to \amalg_{i\in I} X_i$ in $\C T$ factors through
$\amalg_{i\in I'} X_i$, where $I'\subset I$ and $\card I'<\alpha$.

We say that $\C T$ is
\emph{$\alpha$\nobreakdash-compactly generated} if there is a set $\mathcal S$ of objects in $\C T$ such that:
\begin{itemize}
 \item[(a)] $\mathcal S$ \emph{generates} $\mathcal T$, i.e.~an object $X$ in $\C T$ is zero if and only if $\C T(S,X)=0$ for any $S\in\mathcal S$.
 \item[(b)] $\mathcal S$ is \emph{perfect}, i.e.~given a set of morphisms $\{f_i\colon X_i\r Y_i\}_{i\in I}$ in $\C T$, the map between morphism sets $\C T(S,\amalg_{i\in I}f_i)\colon\C T(S,\amalg_{i\in I}X_i)\r\C T(S,\amalg_{i\in I}Y_i)$ is surjective for all $S\in\mathcal S$ provided $\C T(S, f_i)\colon\C T(S, X_i)\r\C T(S, Y_i)$ is surjective for all $i\in I$ and all $S\in\mathcal S$.
 \item[(c)] $\mathcal S$ consists of $\alpha$\nobreakdash-small objects.
\end{itemize}
Moreover, we say that $\C T$ is \emph{well generated} if it is $\alpha$\nobreakdash-compactly generated for some regular cardinal $\alpha$.

The subcategory $\C T^\alpha$ of $\C T$
is defined as the unique maximal full subcategory formed by $\alpha$\nobreakdash-small objects in $\C T$ such that any morphism $S\to \amalg_{i\in I} X_i$ in $\C T$ with $S$ in $\C T^\alpha$ factors through a coproduct $\amalg_{i\in I}f_i\colon \amalg_{i\in I} S_i\to \amalg_{i\in I} X_i$ of morphisms $f_i\colon S_i\r X_i$ with $S_i$ in $\C T^\alpha$ for all $i\in I$. The existence of $\C T^\alpha$ is proved in \cite[Corollary 3.3.10]{triang}. Objects in $\C T^\alpha$ are called \emph{$\alpha$\nobreakdash-compact objects}.
\end{defn}

In the special case $\alpha=\aleph_0$ the $\aleph_0$\nobreakdash-compact objects are exactly the
$\aleph_0$\nobreakdash-small objects. If $\C T$ is $\alpha$\nobreakdash-compactly generated, the generating set $\mathcal S$ in Definition \ref{defbasica} is contained in $\C T^\alpha$ \cite[Lemma 5]{Kr01}, which is an essentially small triangulated subcategory of $\C T$. Well generated triangulated categories have products too by \cite[Proposition 8.4.6]{triang}. Notice that any $\alpha$\nobreakdash-compactly generated triangulated category is also $\beta$\nobreakdash-compactly generated for any regular cardinal $\beta\geq \alpha$. Moreover, $\alpha$\nobreakdash-compact objects are also $\beta$\nobreakdash-compact \cite[Lemma 4.2.3]{triang}.

Throughout this paper $\alpha$ denotes a regular cardinal,
 $\C{T}$ is a well generated triangulated category with suspension functor $\Sigma$, and
$\C{C}\subset\C{T}$ is an essentially small  full  subcategory such that:
\begin{enumerate}
\item it is closed under (de)suspensions,
\item it has coproducts of less than $\alpha$ objects,
\item it generates $\C{T}$,
\item it is perfect.
\end{enumerate}
In particular, $\C T$ is $\alpha$\nobreakdash-compactly generated and $\C{C}\subset\C{T}^{\alpha}$. We do not require $\C C$ to be triangulated. If it were, then necessarily $\C C=\C T^{\alpha}$ by \cite[Lemma 4.4.5]{triang}.
In order to avoid absurd situations, we assume that both $\C C$ and $\C T$ are non-trivial, i.e.~$\C C$ contains at least one object $X\neq 0$.

Let $\Mod{\alpha}{\C C}$ be the abelian category of functors $\C C^{\op}\r\Ab$ preserving products of less than $\alpha$ objects. Such functors are called \emph{$\alpha$\nobreakdash-continuous (right) $\C C$\nobreakdash-modules}. We will now summarize some properties of $\Mod{\alpha}{\C C}$ which are relevant in this paper, providing references when they are not obvious. For the definitions of $\alpha$\nobreakdash-filtered colimit, $\alpha$\nobreakdash-presentable object and locally $\alpha$\nobreakdash-presentable category we refer to \cite[Chapter~1]{AR94}.


The category $\Mod{\alpha}{\C C}$ is locally $\alpha$\nobreakdash-presentable (compare \cite[Proposition A.1.9]{triang} or \cite[Appendix B]{Kr}) and representable functors form a set of $\alpha$\nobreakdash-presentable projective generators (compare \cite[Lemma 6.4.1 and Lemma 6.4.2]{triang} or \cite[Lemma B.3]{Kr}). Moreover, $\Mod{\alpha}{\C C}$ is an abelian subcategory of $\Mod{\aleph_0}{\C C}$, the inclusion is an exact functor (\cite[Lemma 6.1.4]{triang}), and  $\alpha$\nobreakdash-filtered colimits are exact in $\Mod{\alpha}{\C C}$ and computed pointwise (\cite[Lemma A.1.3]{triang}), i.e.~if $\Lambda$ is an $\alpha$\nobreakdash-filtering category, $\Lambda\r \Mod{\alpha}{\C C}\colon \lambda\mapsto F_{\lambda}$ is a diagram of $\alpha$\nobreakdash-continuous $\C C$\nobreakdash-modules indexed by $\Lambda$, and $C$ is an object in $\C C$, then
$$
(\colim_{\lambda\in\Lambda}F_{\lambda})(C)=\colim_{\lambda\in\Lambda}(F_{\lambda}(C)).
$$
Here the first colimit is taken in $\Mod{\alpha}{\C C}$ and the second one is in the category $\operatorname{Ab}$ of abelian groups.

The \emph{restricted Yoneda functor},
$$
S_\alpha\colon \C{T}\To\Mod{\alpha}{\C C},\quad S_{\alpha}(X)=\C{T}(-,X)_{|_{\C{C}}},
$$
preserves products and coproducts (\cite[Proposition 6.8.1]{Kr}), takes exact triangles to exact sequences, and reflects isomorphisms since $\C C$ generates. If $\add(\C{C})\subset\C{T}$ denotes  the smallest subcategory closed under  coproducts and retracts containing $\C C$, then $S_{\alpha}$
induces an equivalence between $\add(\C{C})$ and the full subcategory of projective objects in $\Mod{\alpha}{\C C}$. Moreover, if $P$ is in $\add(\C{C})$ and $X$ is in $\C T$, then Yoneda's lemma implies that $S_{\alpha}$ induces an isomorphism
$$\C{T}(P,X)\cong\hom_{\alpha, \C C}(S_{\alpha}(P),S_{\alpha}(X)).$$
Here $\hom_{\alpha, \C C}$ denotes morphism sets in $\Mod{\alpha}{\C C}$.

Notice that properties \aro{\alpha} and \arm{\alpha}, defined in the introduction, translate as follows for $\C C=\C T^{\alpha}$:
\begin{itemize}
 \item[\aro{\alpha}] The essential image of $S_{\alpha}$ is the class of cohomological functors in $\Mod{\alpha}{\C T^\alpha}$.
 \item[\arm{\alpha}] The functor $S_{\alpha}$ is full.
\end{itemize}

Denote by $\pd(A)$ the projective dimension of an object $A$ in an abelian category~$\C A$.

\begin{prop}\label{full1}
If $S_{\alpha}$ is full, then $\pd(S_{\alpha}(X))\leq 1$ for all $X$ in~$\C T$.
\end{prop}

The proof of this proposition is essentially the same as the proof of \cite[Lemma 4.1]{Ne97}. We will use the following elementary lemma.

\begin{lem}\label{lem22}
If $X\st{f}\r Y\r Z\r\Sigma X$ is an exact triangle and $f$ decomposes as $f=\binom{f'}{0}\colon X\r Y'\oplus Y''=Y$, then this exact triangle is the direct sum of an exact triangle
$$
X\st{f'}\To Y'\To Z'\To\Sigma X$$  and $0\r Y''\st{1}\r Y''\r 0$.
In particular $Z\cong Z'\oplus Y''$.
\end{lem}

\begin{proof}[Proof of Proposition \ref{full1}]
Choose a projective presentation of $S_{\alpha}(X)$,
$$S_{\alpha}(P_{1})\To S_{\alpha}(P_{0})\onto S_{\alpha}(X).$$
It comes from  unique morphisms $P_{1}\st{p_1}\To P_{0}\st{p_0}\To X$
with $p_0p_1=0$, therefore $p_0$ factors through the mapping cone of $p_1$ in an exact triangle
$$
\xymatrix{P_{1}\ar[r]^-{p_1}&P_{0}\ar[r]^-{i}\ar[d]_-{p_0}&Y\ar[r]^-{q}\ar[ld]^-{p'}&\Sigma P_{1}\\
&X.&&}
$$
The universal property of a cokernel shows that $S_{\alpha}(i)$ factors  through $S_{\alpha}(p_0)$,
$$
\xymatrix{S_{\alpha}(P_{1})\ar[r]^-{S_{\alpha}(p_1)}&S_{\alpha}(P_{0})\ar[r]^{S_{\alpha}(i)}\ar@{->>}[d]_{S_{\alpha}(p_0)}&S_{\alpha}(Y)\ar[r]^{S_{\alpha}(q)}\ar@{<-_{)}}[ld]^{\phi}&S_{\alpha}(\Sigma P_{1})\\
&S_{\alpha}(X).&&}
$$
Since $S_{\alpha}(p_0)$ is an epimorphism and
$$
S_{\alpha}(p')\phi S_{\alpha}(p_0)=S_{\alpha}(p')S_{\alpha}(i)=S_{\alpha}(p'i)=S_{\alpha}(p_0),
$$
we deduce that $S_{\alpha}(p')\phi=1_{S_{\alpha}(X)}$. Using that the functor $S_{\alpha}$ is full, we can take a morphism $i'\colon X\r Y$ with $\phi=S_{\alpha}(i')$. Hence, $S_{\alpha}(p')\phi=S_{\alpha}(p'i')=1_{S_{\alpha}(X)}$ and, since $S_{\alpha}$ reflects isomorphisms, $p'i'$ is an automorphism of $X$, so $Y$ decomposes as $(i',i'')\colon X\oplus Z\cong Y$ for some $Z$ and $i''$. On the other hand, since the morphism $S_{\alpha}(i)$ factors as $S_{\alpha}(i')S_{\alpha}(p_0)$ and $S_{\alpha}(P_{0})$ is projective, $i$ itself factors as $i=i'p_0$, i.e.~$i$ decomposes as $i=\binom{p_0}{0}\colon P_{0}\r X\oplus Z\cong Y$. Now, Lemma \ref{lem22} shows that $P_{1}\cong P_{1}'\oplus \Sigma^{-1}Z$ and that there is an exact triangle
$$
P_{1}'\To P_{0}\st{p_0}\To X \To \Sigma P_{1}'.
$$
In particular, $S_{\alpha}(P_{1}')$ is projective.
Since $S_{\alpha}(p_0)$ is an epimorphism, the image under $S_{\alpha}$ of the previous exact triangle produces a length $1$ projective resolution of $S_{\alpha}(X)$,
$$S_{\alpha}(P_{1}')\hookrightarrow S_{\alpha}(P_{0})\onto S_{\alpha}(X).$$
\end{proof}

We derive the following necessary condition for \arm{\alpha}.

\begin{cor}\label{necesaria}
If $\C T$ satisfies \arm{\alpha}, then $\pd(S_{\alpha}(X))\leq 1$ for all $X$ in ${\C T}$ and $\C C=\C T^\alpha$.
\end{cor}

\section{An obstruction theory for the restricted Yoneda functor}\label{sec-obs}

In this section, we describe the formal properties of the obstruction theory developed in Section \ref{obstructiontheory}. We derive a sufficient condition for \aro{\alpha} (Corollary \ref{suficiente}) and necessary and sufficient conditions for \arm{\alpha} (Corollary \ref{necesariaysuficiente0}) and for the $\alpha$\nobreakdash-Adams representability theorem (Corollary \ref{necesariaysuficiente}).

The following notion of exact sequence of categories  generalizes \cite[Definition IV.4.10]{B89} by incorporating an obstruction $\kappa$ to the lifting of objects.

\begin{defn}\label{esc}
Given an additive category $\C B$, a \emph{$\C B$\nobreakdash-bimodule} $M$ is a biadditive functor $M\colon \C B^{\op}\times \C B\r\operatorname{Ab}$. The canonical example is the bimodule defined by morphism sets, that we denote by $\C B=\C B(-,-)$.
As usual, we can restrict scalars along additive functors $\C A\r\C B$, so $\C B$\nobreakdash-bimodules become $\C A$\nobreakdash-bimodules.

An \emph{exact sequence of categories}
$$
\xymatrix{&&M_0&\\M_2\ar[r]^-{\imath}&\C{A}\ar[r]^{t}&\C{B}\ar[r]^-{\theta}\ar[u]^-{\kappa}&M_1}
$$
consists of an additive functor $t$, three $\C{B}$\nobreakdash-bimodules $M_i$, $i=0,1,2$, an exact sequence
$$
\xymatrix{M_2(t(X),t(Y))\ar[r]^-{\imath_{X,Y}}&\C{A}(X,Y)\ar[r]^-{t}&\C{B}(t(X),t(Y))\ar[r]^-{\theta_{X,Y}}&M_1(t(X),t(Y))}
$$
for any two objects $X$ and $Y$ in $\C{A}$, and an element
$$
\kappa(B)\in M_0(B,B)
$$
for any object $B$ in $\C{B}$. The following conditions must be satisfied:
\begin{enumerate}
 \item \emph{Naturality}: for any morphism $f\colon B\r C$ in $\C{B}$,
$$f\cdot\kappa(B)=\kappa(C)\cdot f\in M_0(B,C).$$
\item \emph{Obstruction}: $\kappa(B)=0$ if and only if there exists an object $A$ in $\C{A}$ with $t(A)=B$.
\item \emph{Derivation}: \!given objects $X, Y, Z$ in $\C{A}$ and morphisms $t(X)\st{f}\r t(Y)\st{g}\r t(Z)$ in~$\C{B}$,
$$\theta_{X,Z}(gf)=\theta_{Y,Z}(g)\cdot f+g\cdot \theta_{X,Y}(f)\in M_1(X,Z).$$
\item \emph{Action}: for any object $X$ in $\C{A}$ and any $e\in M_1(t(X),t(X))$ there exists an object $X'=X+e$ in $\C{A}$ with $t(X)=t(X')$ and $\theta_{X,X'}(\id{t(X)})=e$.
\item $\imath$ is a morphism of $\C{A}$\nobreakdash-bimodules.
\end{enumerate}
We sometimes omit the subscripts from $\imath$ and $\theta$ so as not to overload notation.
\end{defn}

In an exact sequence of categories, $\kappa$ is a $0$\nobreakdash-dimensional element in Baues--Wirsching cohomology of categories $H^0(\C B,M_0)$, cf.~\cite{bw}. Moreover, the rest of the exact sequence is determined by a $1$\nobreakdash-dimensional and a $2$\nobreakdash-dimensional cohomology class, compare \cite[Chapter IV]{B89}. Condition (4) guarantees the existence of non-trivial obstructions to the realizability of morphisms as long as the receptacle $M_{1}$ is non-trivial.

A triangulated category $\C{T}$ is regarded as a graded category with graded morphism sets
$$
\C{T}^*(X,Y)=\bigoplus_{n\in\mathbb{Z}}\C{T}(X,\Sigma^nY).
$$
Since $\C{C}$ is closed under (de)suspensions,  $\Sigma$ admits an essentially unique exact extension to $\Mod{\alpha}{\C C}$
compatible with the restricted Yoneda functor, i.e.~the following diagram commutes up to natural isomorphism:
$$
\xymatrix{\C{T}\ar[r]^-\Sigma_\sim\ar[d]_{S_\alpha}&\C{T}\ar[d]^{S_\alpha}\\
\Mod{\alpha}{\C C}\ar[r]^-\Sigma_\sim&\Mod{\alpha}{\C C}.}
$$
The functor $\Sigma$ endows $\Mod{\alpha}{\C C}$ with the structure of a graded abelian category. Graded morphism sets in $\Mod{\alpha}{\C C}$  are defined as in $\C{T}$,
\begin{equation}\label{gradom}
\hom_{\alpha,\C{C}}^{*}(M,N)=\bigoplus_{n\in\mathbb{Z}}\hom_{\alpha,\C{C}}(M,\Sigma^nN).
\end{equation}
In a graded abelian category we also have graded $\ext$ functors that we denote by $\ext_{\alpha,\C{C}}^{p,q}$, where $p$ indicates the length of the extension, i.e.~$\ext_{\alpha,\C{C}}^{p,q}$ is a component of the $p^{\text{th}}$ derived functor of $\hom_{\alpha,\C{C}}^*$ and $q$ is the internal degree coming from the graded $\hom_{\alpha,\C{C}}^{*}$. Notice that $\ext_{\alpha,\C{C}}^{p,q}$ is a  $\Mod{\alpha}{\C C}$\nobreakdash-bimodule. We refer to \cite{street} for additive and abelian category theory in the graded setting.

The following theorem summarizes the main results of Section \ref{obstructiontheory}. 

\begin{thm}\label{lololo}
There is a sequence of exact sequences of categories, $n\geq 0$,
\begin{equation*}
\xymatrix{&&\ext_{\alpha,\C{C}}^{n+3,-1-n}&\\\ext_{\alpha,\C{C}}^{n+1,-1-n}\;\ar[r]^-{\imath_n}&\mathbf{Post}_{n+1}^{\simeq}\ar[r]^{t_{n   }}&\mathbf{Post}_{n}^{\simeq}\ar[r]^-{\theta_n}\ar[u]^-{\kappa_n}&\ext_{\alpha,\C{C}}^{n+2,-1-n}}
\end{equation*}
with $\mathbf{Post}_{0}^{\simeq}\simeq\Mod{\alpha}{\C C}$ and a full and essentially surjective functor
$$\C T\To \mathbf{Post}_{\infty}^{\simeq}=\lim_{n} \mathbf{Post}_{n}^{\simeq}$$
such that the composite $\C T\r \mathbf{Post}_{\infty}^{\simeq}\r \mathbf{Post}_{0}^{\simeq}\simeq\Mod{\alpha}{\C C}$ is naturally isomorphic to the restricted Yoneda functor $S_{\alpha}$.
\end{thm}


The properties of the functor $\C T\r \mathbf{Post}_{\infty}^{\simeq}$ are in Theorems \ref{cociente} and \ref{uneq}. We introduce $\kappa$ in Definition \ref{kappa}, Proposition \ref{naturality} shows that it is natural, and Proposition \ref{vanishing} proves the obstruction property. The construction of $\theta$ is in Definition \ref{theta}, and Propositions \ref{derivation} and \ref{derivation2} prove the derivation and action properties, respectively. Definition \ref{iota} and Proposition \ref{iprimera} yield $\imath$. The exactness properties are checked in Propositions \ref{dolorprofundo}  and \ref{isegunda}.

Under the hypotheses of the following corollary all obstructions in Theorem \ref{lololo} vanish since the recipient bimodules vanish.

\begin{cor}\label{dimensiones}
Under the standing assumptions:
\begin{enumerate}
\item If $F$ is an $\alpha$\nobreakdash-continuous $\C C$\nobreakdash-module with $\pd(F)\leq 2$, then $F\cong S_{\alpha}(X)$ for some $X$ in $\C T$.
\item If $\pd(S_{\alpha}(X))\leq 1$, then any morphism $\tau\colon S_{\alpha}(X)\r S_{\alpha}(Y)$ is $\tau=S_{\alpha}(f)$ for some $f\colon X\r Y$ in $\C T$.
\end{enumerate}
\end{cor}

Combining Corollary \ref{dimensiones} with Proposition \ref{full1} we obtain the following results.

\begin{cor}\label{dimensiones2}
The functor $S_{\alpha}$ is full if and only if its essential image consists of the $\alpha$\nobreakdash-continuous $\C C$\nobreakdash-modules $F$ with $\pd(F)\leq 1$.
\end{cor}

\begin{cor}\label{necesariaysuficiente0}
The category $\C T$ satisfies \arm{\alpha} if and only if $\pd(S_\alpha(X))\leq 1$ for all $X$ in $\C T$ and $\C C=\C T^\alpha$.
\end{cor}

A different approach to the lifting of morphisms along the restricted Yoneda functor for $\alpha=\aleph_0$ is developed in \cite{BK03}.

\begin{rem}\label{alhilo}
We now list some examples of compactly generated triangulated categories $\C T$ and $\C C$, different from $\C T^\alpha$ in general, were the obstruction theory summarized in Theorem \ref{lololo} is interesting. In each case, $\C C$ is the smallest full subcategory closed under (de)suspensions, coproducts of less than $\alpha$ objects, and retracts, containing a certain object that we call \emph{additive generator}. Moreover,  the obstruction theory is independent of the regular cardinal $\alpha$ and
$\Mod{\alpha}{\C C}$ is (equivalent to) a well known graded abelian category.
\begin{enumerate}
\item $\C T=D(R)$ the derived category of a ring $R$, the additive generator is $R$, regarded as a complex concentrated in degree $0$, $\Mod{\alpha}{\C C}=\Mod{}{R}^\mathbb{Z}$ is the category of graded $R$-modules, and the restricted Yoneda functor corresponds to the homology functor $M\mapsto H_{*}(M)$.
\item $\C T$ the stable module category of the group ring $kG$ of a finite $p$-group $G$ over a field $k$ of characteristic $p$,
the additive generator is the trivial representation $k$, $\Mod{\alpha}{\C C}$ is the category of $\widehat H^*(G,k)$-modules, where $\widehat H^*(G,k)$ is the Tate cohomology ring, and the restricted Yoneda functor identifies with the Tate cohomology functor with coefficients $M\mapsto \widehat H^*(G,M)$.
\item $\C T$ the homotopy category of modules over a ring spectrum $R$, the additive generator is $R$, $\Mod{\alpha}{\C C}$ is the category of $\pi_*(R)$-modules, and the restricted Yoneda functor corresponds to the stable homotopy functor $M\mapsto \pi_*(M)$.
\item $\C T$ the derived category of a differential graded algebra $A$, the additive generator is $A$,  $\Mod{\alpha}{\C C}$ is the category of $H_*(A)$-modules, and the restricted Yoneda functor identifies with the homology functor $M\mapsto H_*(M)$
\end{enumerate}
The first obstruction $\kappa_0$ to the realizability of an object has been considered in detail in the last three cases, see \cite{rmtc,Sag08,GH08} respectively. Indeed, \cite{rmtc} is where the obstructions $\kappa_n$ to the realizability of objects were first treated systematically.
\end{rem}

We now consider $\alpha$\nobreakdash-flat objects and their connection with $\alpha$\nobreakdash-Adams represen\-tability.

\begin{defn}
Let $\alpha$ be a regular cardinal and $\C A$ a locally $\alpha$\nobreakdash-presentable abelian category with exact $\alpha$\nobreakdash-filtered colimits and a set of $\alpha$\nobreakdash-presentable projective generators. An \emph{$\alpha$\nobreakdash-flat object} $A$ in $\C A$ is  an $\alpha$\nobreakdash-filtered colimit of $\alpha$\nobreakdash-presentable projective objects $A=\colim_{\lambda\in \Lambda}P_{\lambda}$. The \emph{$\alpha$\nobreakdash-flat global dimension} of $\C A$ is
$$\fgd_{\alpha}(\C A)=\sup\{\pd(A)\mid A\text{ is $\alpha$\nobreakdash-flat}\}.$$
\end{defn}

\begin{rem}\label{canonica}
An $\alpha$\nobreakdash-flat object  $A=\colim_{\lambda\in \Lambda}P_{\lambda}$ has a canonical projective resolution of the form
$$
\cdots\r \bigoplus_{\lambda\r \mu\r\nu\in\Lambda}P_{\lambda} \To \bigoplus_{\lambda\r \mu\in\Lambda}P_{\lambda} \To \bigoplus_{\lambda\in\Lambda} P_{\lambda} \onto A.
$$
These direct sums are indexed by the simplices of the nerve $N\Lambda$ of the  category $\Lambda$ indexing the colimit, i.e.~for each $n$, $N_{n}\Lambda=\{\text{chains of $n$ composable maps in }\Lambda\}$. In particular, for any other object $B$ in $\C A$ the higher $\ext$'s
$$\ext^{n}_{\C A}(A,B)={\lim_{\lambda\in\Lambda}}^{n}\hom_{\C A}(P_{\lambda},B)$$
are the derived functors of the inverse limit.
\end{rem}

\begin{rem}\label{sonlosplanos}
In $\Mod{\alpha}{\C T^\alpha}$, the $\alpha$\nobreakdash-flat objects coincide with the cohomological functors, cf.~\cite[Section~7.2]{triang}.
\end{rem}

The $\alpha$\nobreakdash-flat global dimension of $\C C$ can be bounded above if the cardinal of $\C C$ is not too large.

\begin{defn}
The \emph{cardinal} of a small category $\C C$ is
$$\card\C C=\card\coprod_{x,y\in \mathcal S}\C C(x,y),$$
where $\mathcal S$ is a set of representatives of isomorphism classes of objects in $\C C$.
\end{defn}

\begin{lem}\label{cotainferior}
If $\C C$ is a non-trivial additive category with coproducts of less than $\alpha$ objects, then $\card\C C\geq \alpha$.
\end{lem}

\begin{proof}
If $X\neq 0$ the identity in $X$ is non-trivial, so $\card \C C(X,X)\geq 2$. For $\beta<\alpha$,
$$
\C C(\coprod_{\beta}X,X)=\prod_{\beta}\C C(X,X),\qquad \card\prod_{\beta}\C C(X,X)\geq 2^{\beta}.
$$
Hence,
$\card\C C\geq \sup_{\beta<\alpha} 2^{\beta}$. We now distinguish two cases, if $\alpha=\gamma^{+}$ is a successor, then $\sup_{\beta<\alpha} 2^{\beta}=2^{\gamma}\geq \gamma^{+}=\alpha$, and, if $\alpha$ is a limit cardinal, then $\sup_{\beta<\alpha} 2^{\beta} \geq \sup_{\beta<\alpha}{\beta}=\alpha$.
\end{proof}

In Section \ref{sec-alphacpt} we show, under the generalized continuum hypothesis, that there is always a large enough cardinal $\alpha$ such that $\card\C T^{\alpha}=\alpha$ is as small as it can be.

By Lemma \ref{cotainferior}, the hypothesis of the following proposition can only be satisfied if $\alpha\leq\aleph_{n}$.

\begin{prop}\label{cota}
If $\card \C C\leq \aleph_{n}$, then $\fgd_{\alpha} (\Mod{\alpha}{\C C})\leq n+1$.
\end{prop}

\begin{proof}
The full inclusion $\Mod{\alpha}{\C C}\subset \Mod{\aleph_{0}}{\C C}$ preserves $\alpha$\nobreakdash-filtered colimits. The $\alpha$\nobreakdash-presentable projective objects in $\Mod{\alpha}{\C C}$ are the retracts of the representable functors $\C C(-,X)$, which also coincide with the $\aleph_{0}$\nobreakdash-presentable objects in $\Mod{\aleph_{0}}{\C C}$. Therefore $\alpha$\nobreakdash-flat objects in $\Mod{\alpha}{\C C}$ are also $\alpha$\nobreakdash-flat in $\Mod{\aleph_{0}}{\C C}$, in particular $\aleph_{0}$\nobreakdash-flat. Moreover, by Remark \ref{canonica},  if $F$ is an $\alpha$\nobreakdash-continuous $\C C$\nobreakdash-module and $H=\colim_{\lambda\in\Lambda}P_{\lambda}$ is an $\alpha$\nobreakdash-flat $\alpha$\nobreakdash-continuous $\C C$\nobreakdash-module, then
$$\ext^{n}_{\alpha,\C C}(H,F)={\lim_{\lambda\in\Lambda}}^{n}\hom_{\alpha,\C C}(P_{\lambda},F)
={\lim_{\lambda\in\Lambda}}^{n}\hom_{\aleph_{0},\C C}(P_{\lambda},F)=\ext^{n}_{\aleph_{0},\C C}(H,F).$$
This is proven in \cite[Proposition 7.5.5]{triang} assuming that $\C C$ is triangulated, but this hypothesis is not really used. If $\card \C C\leq \aleph_{n}$, then any $\aleph_{0}$\nobreakdash-flat $\aleph_{0}$\nobreakdash-continuous $\C C$\nobreakdash-module has projective dimension $\leq n+1$ in $\Mod{\aleph_{0}}{\C C}$, see \cite[Corollary 3.13]{Si77}. Hence the proposition follows from the previous equation.
\end{proof}

We now concentrate on the case $\C C=\C T^{\alpha}$. The following sufficient condition for \aro{\alpha} follows from  Corollary \ref{dimensiones} and Remark \ref{sonlosplanos}.

\begin{cor}\label{suficiente}
If $\fgd_{\alpha} (\Mod{\alpha}{\C T^{\alpha}})\leq 2$ , then $\C T$ satisfies \aro{\alpha}.
\end{cor}

For the following corollary we also use Proposition \ref{cota}. The restrictions on the cardinal $\alpha$ are imposed by Lemma \ref{cotainferior}.

\begin{cor}\label{cota2}
Let $\alpha$ be $\aleph_{0}$ or $\aleph_{1}$. If $\card \C T^{\alpha}\leq \aleph_1$, then $\C T$ satisfies \aro{\alpha}.
\end{cor}

The following homological characterization of $\alpha$\nobreakdash-Adams representability is a consequence of Corollaries \ref{dimensiones2} and \ref{suficiente} and Remark \ref{sonlosplanos}.

\begin{cor}\label{necesariaysuficiente}
A triangulated category $\C{T}$ satisfies $\alpha$\nobreakdash-Adams representability if and only if $\fgd_{\alpha} (\Mod{\alpha}{\C T^{\alpha}})\leq 1$.
\end{cor}

Using Proposition \ref{cota}, we obtain Neeman's sufficient condition for  $\aleph_{0}$\nobreakdash-Adams representability, cf.~\cite{Ne97}.

\begin{cor}
If $\card \C T^{\aleph_0}\leq \aleph_0$, then $\C T$ satisfies $\aleph_{0}$\nobreakdash-Adams representability.
\end{cor}

\begin{rem}\label{auslander}
In the case $\alpha=\aleph_0$, Beligiannis proves in \cite[Theorem~11.8]{Be00} that $\C{T}$ satisfies \arm{\aleph_0} if and only if $\fgd_{\aleph_0} (\Mod{\aleph_0}{\C T^{\aleph_0}})\leq 1$. Thus, by Corollary \ref{necesariaysuficiente}, \arm{\aleph_0} implies \aro{\aleph_0}.

A crucial step in his proof is that, since $\Mod{\aleph_0}{\C T^{\aleph_0}}$ is a Grothen\-dieck category, it follows from \cite[Theorem 2.7]{Si77} that $\fgd_{\aleph_0} (\Mod{\aleph_0}{\C T^{\aleph_0}})=\sup\{\pd(A)\mid A\text{ is $\alpha$\nobreakdash-flat and $\pd(A)<\infty$}\}$.

The fact that $\Mod{\aleph_0}{\C T^{\aleph_0}}$ is Grothendieck is used in order to apply (in each step of an inductive argument) the Auslander Lemma \cite[Lemma VI.2.6]{FS}. This lemma says that if an object $X$ is the union of a well-ordered continuous ascending sequence of subobjects $X_\alpha$ such that $\pd(X_{\alpha+1}/X_{\alpha})\leq k$, then $\pd(X)\leq k$. The second author has proved a generalization of the Auslander Lemma for $\Mod{\aleph_n}{{\C T}^{\aleph_n}}$, which is not Grothendieck. This generalization, under the same hypotheses, yields $\pd(X)\leq k+n$, and not $\leq k$, which hampers the inductive argument.
Using a completely different approach, we will extend Beligiannis' result for $\C T=D(R)$ with $R$ a hereditary ring and any $\alpha$, see Theorem \ref{mainderived} and Corollary \ref{cor-rings1}.
\end{rem}

\section{Transfinite Adams representability in the derived category of a ring}\label{sec-for-rings}

In this section we consider \aro{\alpha} and \arm{\alpha} for the derived category $D(R)$  of an $\alpha$\nobreakdash-coherent ring $R$. The main result is Theorem \ref{mainderived}, which gives a necessary condition for \arm{\alpha}, and also necessary and sufficient conditions for \aro{\alpha} and \arm{\alpha} if $R$ is hereditary. We also prove \aro{\aleph_1} for rings of cardinality $\leq\aleph_1$ under the continuum hypothesis (Proposition \ref{aleph_1-Adams-for-rings}). All modules considered in this section are right modules. We rely on the obstruction theory summarized in Section \ref{sec-obs} and, towards the end, also on the relation between the Adams spectral sequence and the first obstruction to the realization of an $\alpha$-continuous $\C C$-module, which is established independently later in Section \ref{sec-first-obs}.

\begin{defn}
Let $R$ be a ring and $\alpha$ a regular cardinal.
An $R$\nobreakdash-module is \emph{$\alpha$\nobreakdash-generated} if it has a set of generators of cardinal $<\alpha$, it is \emph{$\alpha$\nobreakdash-presentable}  if it is the quotient of two $\alpha$\nobreakdash-generated projective modules.
The ring $R$ is \emph{$\alpha$\nobreakdash-coherent} if all $\alpha$\nobreakdash-generated submodules of free $R$\nobreakdash-modules are $\alpha$\nobreakdash-presentable. It is enough to check this condition on ideals, cf.~\cite[Chapter~7]{JL89}.
\end{defn}

\begin{rem}
Alternatively, an $R$\nobreakdash-module $P$ is $\alpha$\nobreakdash-presentable if it admits a free presentation
$$\bigoplus_{J}R\To \bigoplus_{I}R\onto P$$
with $\card I,\card J<\alpha$. Any $\alpha$\nobreakdash-presentable $R$\nobreakdash-module is $\alpha$\nobreakdash-generated. The converse is true for projective $R$\nobreakdash-modules.

If $\card R<\alpha$, then $R$ is $\alpha$\nobreakdash-coherent since in this case $\alpha$-generated modules are the same as $\alpha$-presentable modules, and the same as modules of cardinality $<\alpha$. Moreover, hereditary rings are $\alpha$\nobreakdash-coherent for all $\alpha$ since ideals are projective.
\end{rem}

We now state the main result of this section. We make use of the $\alpha$\nobreakdash-pure global dimension of a ring $\Pgldim_{\alpha}(R)$ as it was defined in the introduction, cf.~\cite[Chapter~7]{JL89}, although below we give a more general definition for abelian categories.

\begin{thm}\label{mainderived}
Let $R$ be an $\alpha$\nobreakdash-coherent ring, $\alpha>\aleph_{0}$. If $D(R)$ satisfies \arm{\alpha}, then $\Pgldim_{\alpha}(R)\leq 1$. Moreover, if $R$ is hereditary, then
\begin{enumerate}
\item \aro{\alpha} for $D(R)$ $\Leftrightarrow$ $\Pgldim_{\alpha}(R)\leq 2$, and
\item \arm{\alpha} for $D(R)$ $\Leftrightarrow$ $\Pgldim_{\alpha}(R)\leq 1$.
\end{enumerate}
\end{thm}

We prove Theorem \ref{mainderived} at the end of this section. The version for $\alpha=\aleph_{0}$,  proved in \cite[Theorem 2.13]{CKN01}, also requires that finitely presented $R$\nobreakdash-modules have finite projective dimension, which is of course true for $R$ hereditary.

\begin{exm}\label{ejemplillo}
A consequence of Theorem \ref{mainderived} is that \arm{\alpha} is not satisfied for the derived category of $\alpha$\nobreakdash-coherent rings $R$ such that $\Pgldim_{\alpha}(R)> 1$. Hence we can use computations of lower bounds to $\alpha$\nobreakdash-pure projective dimensions in \cite{BL82}, \cite{BG},  and \cite{BS} to show that \arm{\alpha} is not satisfied for rings $R$ and regular cardinals $\alpha$ as indicated:
\begin{enumerate}
\item For $\alpha>\aleph_{0}$:
\begin{enumerate}
 \item $R=\mathbb Z$ .
 \item $R$ a DVR.
\end{enumerate}

\item Let $k$ be an uncountable field and $\alpha$ any regular cardinal or $k$ a countable field and $\alpha>\aleph_{0}$:
\begin{enumerate}
\item  $R=k[x,y]$.

\item $R$ the path algebra of a finite quiver without oriented cycles which is not a Dynkin quiver.

\end{enumerate}
\item $R=k[[x,y]]$ for any field $k$ and any regular cardinal $\alpha$.
\end{enumerate}
\end{exm}

\begin{exm}\label{ejemplillo2}
\v{S}{\v{t}}ov{\'i}\v{c}ek shows in recent work in progress \cite{Stovicek} that, for $p$ any prime and $\widehat{\mathbb Z}_p$ the ring of $p$-adic integers, $\Pgldim_{\aleph_n}(\widehat{\mathbb Z}_p)\geq n+1$ for any finite $n\geq 1$, therefore $D(\widehat{\mathbb Z}_p)$ does not satisfy either \arm{\aleph_n} or \aro{\aleph_n} for $n\geq 2$. This is the first triangulated category which is known not to satisfy any form of $\alpha$-Adams representability for an uncountable $\alpha$. \v{S}{\v{t}}ov{\'i}\v{c}ek is actually extending his result to arbitrary DVRs, the Kronecker algebra, etc. The final version will probably yield a variety of examples.
\end{exm}

\begin{rem}\label{rem-Dynkin}
It is well known that a ring $R$ has $\Pgldim_\alpha(R)=0$ for some $\alpha$ if and only if $\Pgldim_{\aleph_0}(R)=0$, see \cite[Theorem 8.4]{JL89}. These rings are called \emph{(right) pure-semisimple}. They are characterized by the fact that, for any $\alpha\geq\aleph_{0}$,  any $R$-module decomposes as a direct sum of $\alpha$-presentable $R$-modules. Rings of finite representation type are (two-sided) pure-semisimple \cite[Theorem 8.8]{JL89}. If $R$ is hereditary and pure-semisimple, e.g.~the path algebra of a Dynkin quiver over a field, then $D(R)$ satisfies $\alpha$-Adams representability for all $\alpha\geq\aleph_0$. So far we do not know of any ring $R$ with $\Pgldim_\alpha (R)=1$ for some $\alpha>\aleph_0$.
\end{rem}

\begin{rem}
Although Theorem \ref{mainderived} only characterizes $\alpha$-Adams representability for derived categories of hereditary rings, we can also compute some non-hereditary examples. Let $k$ be a field and $R=k[\varepsilon]/(\varepsilon^{2})$ its ring of dual numbers. Any object in $D(R)$ decomposes as a direct sum of complexes of the form
\begin{equation*}
\cdots\r 0\r \underbrace{R\st{\varepsilon}\r R\st{\varepsilon}\r R\r\cdots \r R}_{n\text{ copies of }R, \; n\geq 1,}\r 0\r\cdots,
\end{equation*}
which are $\aleph_{0}$-compact, and (de)suspended copies of $k$, compare \cite[\S3.3]{hcimdda}. The $R$-module $k=R/(\varepsilon)$ has infinite projective dimension, so it is not $\aleph_{0}$-compact in $D(R)$. However, it is  $\aleph_{1}$-compact since  it has a resolution by finitely generated projective $R$-modules
\[\cdots\r R\st{\varepsilon}\r R\st{\varepsilon}\r R\r 0\r\cdots,\]
see \cite[Theorem 20]{Mu}.

Therefore, $S_{\alpha}(X)$ is projective in $\Mod{\alpha}{D(R)^{\alpha}}$ for any $X$ in $D(R)$ and any $\alpha>\aleph_{0}$. For $\alpha=\aleph_{0}$, $S_{\aleph_{0}}(k)$ has projective dimension $1$. Indeed, the previous projective resolution of the $R$-module $k$ is the (homotopy) colimit of its naive truncations, and applying $S_{\aleph_{0}}$ to  the homotopy colimit exact triangle \eqref{hcet} we obtain a length $1$ projective resolution of $S_{\aleph_{0}}(k)$ in $\Mod{\aleph_{0}}{D(R)^{\aleph_{0}}}$. We deduce from Corollary \ref{necesariaysuficiente0} that $D(R)$ satisfies \arm{\alpha} for any $\alpha$. By \cite[Theorem~11.8]{Be00}, $D(R)$ also satisfies \aro{\aleph_{0}}.
For $\alpha>\aleph_{0}$, we lave  \aro{\alpha}  as an exercise for the reader.

This example extends to the finite-dimensional algebras $A_{n}$ in  \cite[\S3.1]{hcimdda}, $n\geq 1$.
\end{rem}


The following result proves \aro{\aleph_1} for rings of cardinality $\leq \aleph_1$ under the continuum hypothesis. The proof is given after some preliminary considerations.

\begin{prop}\label{aleph_1-Adams-for-rings}
Let $\alpha$ be an inaccessible cardinal or $\alpha=\beta^+=2^\beta$.
If  $R$ is a ring of $\card
R\leq\alpha$, then
$\card  D(R)^{\alpha}\leq \alpha$. In particular, if $\card
R\leq\aleph_n=2^{\aleph_{n-1}}$ then $\Pgldim_{\aleph_n}(R)\leq n+1$. Moreover, if $\card
R\leq\aleph_1$ and the continuum hypothesis holds then $\Pgldim_{\aleph_1}(R)\leq 2$ and $D(R)$ satisfies \aro{\aleph_1}.
\end{prop}

\begin{rem}
Recall from Example \ref{ejemplillo} that for the following rings $R$, $\Pgldim_{\aleph_1}(R)>1$.
\begin{enumerate}
\item $R=\mathbb Z$.
\item $R$ a DVR of $\card R\leq\aleph_1$.
\item Let $k$ be a field of $\card k\leq\aleph_1$:
\begin{enumerate}
\item  $R=k[x,y]$.

\item $R$ the path algebra of a finite quiver without oriented cycles which is not a Dynkin quiver.
\end{enumerate}
\item $R=k[[x,y]]$ for a field $k$ of $\card k\leq\aleph_1$.
\end{enumerate}
The last part of Proposition \ref{aleph_1-Adams-for-rings} applies to these rings, therefore, under the continuum hypothesis, $\Pgldim_{\aleph_1} (R)=2$ and $D(R)$ satisfies \aro{\aleph_1}.

Proposition \ref{aleph_1-Adams-for-rings} under the hypothesis $\aleph_n=2^{\aleph_{n-1}}$, $n\geq 1$ finite, combined with \cite{Stovicek}, shows that the $p$-adic integers satisfy $\Pgldim_{\aleph_n} (\widehat{\mathbb Z}_p)=n+1$. We wonder whether this can be proved without set theoretical assumptions.

\end{rem}

\begin{defn}\label{def_alpha_exact_seq}
Let $\alpha$ be a regular cardinal and $\C A$ a locally $\alpha$\nobreakdash-presentable abelian category with exact $\alpha$\nobreakdash-filtered colimits and a set of $\alpha$\nobreakdash-presentable projective generators.
A short exact sequence  $A\hookrightarrow B\onto C$ is  \emph{$\alpha$\nobreakdash-pure} if
$$
 \C A(P,A)\hookrightarrow \C A(P,B)\onto
\C A(P,C)
$$
is short exact for any $\alpha$\nobreakdash-presentable object $P$, or equivalently, if it is an $\alpha$\nobreakdash-filtered colimit of split short exact sequences.

A sequence $\cdots\r A_{n+1}\r A_{n}\st{d_{n}}\r A_{n-1}\r\cdots$ in $\C A$ is \emph{$\alpha$\nobreakdash-pure exact} if
it is exact and $\ker d_{n}\hookrightarrow A_{n}\onto \im d_{n}$ is $\alpha$\nobreakdash-pure for all $n\in\mathbb{Z}$.

An object $Q$ in $\C A$ is \emph{$\alpha$\nobreakdash-pure projective} if $\hom_{R}(Q,-)$ takes $\alpha$\nobreakdash-pure exact sequences to exact sequences, this is equivalent to say that $Q$ is a retract of a direct sum of $\alpha$\nobreakdash-presentables.

The notions of $\alpha$\nobreakdash-pure projective resolution, $\alpha$\nobreakdash-pure projective dimension $\Ppd_{\alpha}(A)$ of an object $A$ in $\C A$, etc., are defined in the obvious way. The \emph{$\alpha$\nobreakdash-pure global dimension} of $\C A$ is denoted by $$\Pgldim_{\alpha}(\C A)=\sup\{\Ppd_{\alpha}(A)\mid A\text{ in }\C A\}.$$
If $\C A=\Mod{}{R}$ is the category of modules over a ring $R$ we abbreviate $\Pgldim_{\alpha} (R)=\Pgldim_{\alpha}(\Mod{}{R})$.

Given $A$ and $B$ in $\C A$, the \emph{$\alpha$\nobreakdash-pure extension groups}
$$\pext^{n}_{\alpha,\C A}(A,B)$$
are defined as the cohomology of an $\alpha$\nobreakdash-pure projective resolution of $A$ with coefficients in $B$.
\end{defn}

\begin{rem}\label{lapuri}
Any object $A$ in $\C A$ is an $\alpha$\nobreakdash-filtered colimit of $\alpha$\nobreakdash-presentable objects $A=\colim_{\lambda\in\Lambda}P_{\lambda}$, hence the construction in Remark \ref{canonica} yields  an $\alpha$\nobreakdash-pure projective resolution of $A$, in particular
$$\pext^{n}_{\alpha,\C A}(A,B)={\lim_{\lambda\in\Lambda}}^{n}\hom_{\C A}(P_{\lambda},B).$$
If $A$ is $\alpha$\nobreakdash-flat we can take $P_{\lambda}$ projective for all $\lambda\in\Lambda$ and the projective resolution of $A$  in Remark \ref{canonica} is also $\alpha$\nobreakdash-pure, so $\pext^{n}_{\alpha,\C A}(A,B)=\ext^{n}_{\C A}(A,B)$ in this case. This proves that
$$\fgd_{\alpha}(\C A)\leq \Pgldim_{\alpha}(\C A).$$
For an arbitrary object $A$, the spectral sequence for the composition of functors $\hom_{\C A}(A,B)=\hom_{\C A}(\colim_{\lambda\in\Lambda}P_{\lambda},B)=\lim_{\lambda\in\Lambda}\hom_{\C A}(P_{\lambda},B)$ is of the form
$$E^{p,q}_{2}={\lim_{\lambda\in\Lambda}}^{p}\ext^{q}_{\C A}(P_{\lambda},B)\Longrightarrow \ext^{p+q}_{\C A}(A,B).$$
The comparison homomorphism between $\alpha$\nobreakdash-pure and ordinary extensions groups is part of this spectral sequence,
$$\pext^{n}_{\alpha,\C A}(A,B)=E^{n,0}_{2}\onto E^{n,0}_{\infty}\subset \ext^{n}_{\C A}(A,B).$$
\end{rem}

\begin{lem}\label{sonpuras}
Any short exact sequence $A\hookrightarrow B\onto C$ where $C$ is $\alpha$\nobreakdash-flat is $\alpha$\nobreakdash-pure.
\end{lem}

\begin{proof}
Since $C=\colim_{\lambda\in \Lambda}P_{\lambda}$ is an $\alpha$\nobreakdash-filtered colimit of $\alpha$\nobreakdash-presentable projective objects, taking pullback along the canonical morphisms
$P_{\lambda}\r C$
$$\xymatrix{A\ar@{=}[d]\ar@{^{(}->}[r]&Q_{\lambda}\ar@{->>}[r]\ar[d]\ar@{}[rd]|{\text{pull}}&P_{\lambda}\ar[d]\\
A\ar@{^{(}->}[r]&B\ar@{->>}[r]&C}$$
we can express the short exact sequence below as an $\alpha$\nobreakdash-filtered colimit
$$
\colim_{\lambda\in\Lambda}(A\hookrightarrow Q_{\lambda}\onto P_{\lambda})
$$
of short exact sequences which split since $P_\lambda$ is projective.
\end{proof}

The following lemma admits the same proof as \cite[Theorem 20]{Mu}. There, it is assumed that $R$ is right Noetherian or $\card R<\alpha$ but actually only $\alpha$\nobreakdash-coherence is used.

\begin{lem}\label{lem-alpha-compact}
Let $R$ be an $\alpha$\nobreakdash-coherent ring for some $\alpha>\aleph_{0}$.  A complex $X$ in $D(R)$ is $\alpha$\nobreakdash-compact if and only if $H_{n}(X)$ is an $\alpha$\nobreakdash-presentable $R$\nobreakdash-module for all $n\in\mathbb Z$.
\end{lem}

\begin{lem}\label{lem_purity_ring_D(ring)}
Let $R$ be an $\alpha$\nobreakdash-coherent ring, $\alpha>\aleph_0$.
\begin{enumerate}
\item The functor $H_{0}\colon \Mod{\alpha}{D(R)^{\alpha}}\r \Mod{}{R}$ defined as $H_0(F)=F(R)$ takes projective objects to $\alpha$\nobreakdash-pure projective $R$\nobreakdash-modules and preserves $\alpha$\nobreakdash-filtered colimits and $\alpha$\nobreakdash-pure exact sequences.

\item The functor $\mathbf{y}\colon \Mod{}{R}\subset D(R)\st{S_{\alpha}}\To \Mod{\alpha}{D(R)^{\alpha}}$ takes $\alpha$\nobreakdash-pure projective $R$\nobreakdash-modules to projective objects and preserves $\alpha$\nobreakdash-filtered colimits and $\alpha$\nobreakdash-pure exact sequences.
\end{enumerate}
\end{lem}

\begin{proof}
If $X$ is in $D(R)^{\alpha}$, then $H_{0}S_{\alpha}(X)=S_{\alpha}(X)(R)=D(X)(R,X)=H_{0}(R)$, which is $\alpha$\nobreakdash-presentable by the Lemma \ref{lem-alpha-compact}, hence $H_{0}$ takes projective objects to $\alpha$\nobreakdash-pure projective $R$\nobreakdash-modules. In $\Mod{\alpha}{D(R)^{\alpha}}$, $\alpha$\nobreakdash-filtered colimits are computed pointwise hence $H_{0}$ preserves these colimits. Since $H_{0}$ preserves split short exact sequences and $\alpha$\nobreakdash-filtered colimits, we deduce that it also preserves $\alpha$\nobreakdash-pure short exact sequences. This finishes the proof of (1).

If $M$ is an $\alpha$\nobreakdash-presentable $R$\nobreakdash-module, then $M$ is $\alpha$\nobreakdash-compact in $D(R)$ by Lemma \ref{lem-alpha-compact}, so  $S_{\alpha}(M)$ is projective in $\Mod{\alpha}{D(R)^{\alpha}}$. It follows that $\mathbf y$ takes $\alpha$\nobreakdash-pure projective $R$\nobreakdash-modules to projective objects.

Let $M=\colim_{\lambda\in\Lambda}M_{\lambda}$ be an $\alpha$\nobreakdash-filtered colimit of $R$\nobreakdash-modules. Denote by $\C S\subset D(R)^\alpha$ the full subcategory of objects $X$ such that the natural morphism
$$
(\colim_{\lambda\in\Lambda}S_{\alpha}(M_{\lambda}))(X)=\colim_{\lambda\in\Lambda}D(R)(X,M_{\lambda})\To D(R)(X,\colim_{\lambda\in\Lambda}M_{\lambda})=(S_{\alpha}(M))(X)
$$
is an isomorphism. The category $\C S$ contains $\Sigma^{n} R$, $n\in\mathbb Z$. Indeed, for $n\neq 0$ this morphism is $0\r 0$ and for $n=0$ it is the identity in  $\colim_{\lambda\in\Lambda}M_{\lambda}$. The category of abelian groups is locally finitely presentable, hence $\alpha$\nobreakdash-filtered colimits commute with products of less than $\alpha$ objects \cite[Proposition 1.59]{AR94}. This shows that $\C S$ is closed under coproducts of less than $\alpha$ objects. The category $\C S$ is also closed under exact triangles by the five lemma. Therefore $\C S=D(R)^{\alpha}$ and hence $\mathbf y$ preserves $\alpha$\nobreakdash-filtered colimits.

Any $\alpha$\nobreakdash-pure short exact sequence of $R$\nobreakdash-modules is an $\alpha$\nobreakdash-filtered colimit of split ones. Since $\mathbf y$  preserves split short exact sequences and $\alpha$\nobreakdash-filtered colimits we deduce that $\mathbf y$  preserves $\alpha$\nobreakdash-pure short exact sequences. This concludes the proof of (2).
\end{proof}

\begin{cor}\label{fomul}
Given an $\alpha$\nobreakdash-coherent ring $R$, $\alpha >\aleph_0$, and an $R$\nobreakdash-module $M=\colim_{\lambda}P_\lambda$ expressed as an $\alpha$\nobreakdash-filtered colimit of $\alpha$\nobreakdash-presentable $R$\nobreakdash-modules $P_\lambda$,
$$\ext^{n}_{\alpha, D(R)^\alpha}(S_{\alpha}(M),F)={\lim_{\lambda\in\Lambda}}^{n}F(P_{\lambda}).$$
In particular, for $F=S_{\alpha}(\Sigma^{j}N)$, $j\in\mathbb Z$,
$$\ext^{n}_{\alpha, D(R)^\alpha}(S_{\alpha}(M),S_{\alpha}(\Sigma^{j}N))={\lim_{\lambda\in\Lambda}}^{n}\ext^{j}_{R}(P_{\lambda},N).$$
\end{cor}

\begin{proof}
Take the $\alpha$\nobreakdash-pure projective resolution of $M$ in Remark \ref{lapuri}. Applying $S_\alpha$ we obtain a projective resolution of $S_\alpha(M)$ by Lemma \ref{lem_purity_ring_D(ring)} (2). Using this resolution to compute $\ext^{n}_{\alpha, D(R)^\alpha}(S_{\alpha}(M),F)$ we obtain the equation in the statement.
\end{proof}

\begin{prop} \label{Thmrings1}
Given an $\alpha$\nobreakdash-coherent ring  $R$,
$\alpha>\aleph_0$, if $H$ is a cohomological functor in $\Mod{\alpha}{D(R)^{\alpha}}$, then $\Ppd_{\alpha}(H(R))\leq \pd(H)$. Moreover, for any $R$\nobreakdash-module $M$,
$
\Ppd_\alpha(M)=\pd(S_\alpha(M)).
$
\end{prop}

\begin{proof}
By Remark \ref{sonlosplanos} and Lemma \ref{sonpuras}, any projective resolution of $H$ is also an $\alpha$\nobreakdash-pure projective resolution, hence Lemma \ref{lem_purity_ring_D(ring)} (1) proves the first part. Since $S_{\alpha}(M)(R)=M$, this also proves $\Ppd_\alpha(M)\leq\pd(S_\alpha(M))$. The other inequality follows from Lemma \ref{lem_purity_ring_D(ring)} (2).
\end{proof}

\begin{cor}\label{cor-rings1}
If $R$ is an $\alpha$\nobreakdash-coherent ring, $\alpha>\aleph_{0}$, then $$\Pgldim_{\alpha} (R)\leq \fgd_{\alpha} (\Mod{\alpha}{D(R)^{\alpha}}).$$ {Moreover, if $R$ is hereditary and $\fgd_{\alpha} (\Mod{\alpha}{D(R)^{\alpha}})\leq 2$, then the equality holds $\fgd_{\alpha} (\Mod{\alpha}{D(R)^{\alpha}})=\Pgldim_{\alpha} (R)$}.
\end{cor}

\begin{proof}
The first part follows directly from Proposition \ref{Thmrings1} and Remark \ref{sonlosplanos}. By Corollary \ref{suficiente}, if $\fgd_{\alpha} (\Mod{\alpha}{D(R)^{\alpha}})\leq 2$, then every $\alpha$\nobreakdash-flat object is representable and the result follows from Proposition \ref{Thmrings1} and the fact that, if $R$ is hereditary, then any complex $X$ splits as $X\cong\bigoplus_{n\in\mathbb{Z}}\Sigma^n H_n(X)$.
\end{proof}

We can now prove Proposition \ref{aleph_1-Adams-for-rings}.

\begin{proof}[Proof of Proposition \ref{aleph_1-Adams-for-rings}]
Let $S$ be the set of $K$\nobreakdash-projective complexes formed by free $R$\nobreakdash-modules of the form $\bigoplus_{i\in I} R$ with $\card I<\alpha$.
By \cite[Theorem 15]{Mu}, any $\alpha$\nobreakdash-compact complex in $D(R)$ is isomorphic to an object in $S$.
The morphism set between two of those free $R$\nobreakdash-modules
$$
\Hom_{R}(\bigoplus_{i\in I} R, \bigoplus_{j\in J}
R)\cong\prod_{i\in I}\Hom_{R}(R, \bigoplus_{j\in J}
R)\cong \prod_{i\in I}\bigoplus_{j\in J}\hom_{R}(R,R)\cong \prod_{i\in I}\bigoplus_{j\in J}R
$$
has cardinal $\leq \alpha^{\card I}$, and, under our assumptions, $\alpha^{\card I}\leq\alpha$, compare \cite[Theorem 5.20]{Je03}. This shows that $\card S\leq \alpha$, and moreover that the set of chain maps between two objects $X$ and $Y$ in $S$ has cardinal $\leq \alpha$. Since $D(R)(X,Y)$ is the quotient of the set of chain maps by the homotopy relation, we deduce that $\card D(R)^{\alpha}\leq \alpha$.

For the last part of the statement we use Proposition \ref{cota} and Corollaries \ref{cota2} and \ref{cor-rings1}.
\end{proof}

The following result gives a necessary condition for the representability of cohomological functors in $\Mod{\alpha}{D(R)^{\alpha}}$ which fit into an extension of restricted representables.

\begin{lem}\label{loca}
Let $M$ and $N$ be $R$\nobreakdash-modules and $S_{\alpha}(\Sigma^{j}N)\st{a}\hookrightarrow F\st{b}\onto S_{\alpha}(M)$ an extension in $\Mod{\alpha}{D(R)^{\alpha}}$, $j>0$, $\alpha>\aleph_0$, classified by an element
$$
e_{F}\in\ext^{1}_{\alpha, D(R)^{\alpha}}(S_{\alpha}(M),S_{\alpha}(\Sigma^{j}N))={\lim_{\lambda}}^{1}\ext^{j}_{R}(P_{\lambda},N)=E^{1,j}_{2}.
$$
Here $M=\colim_{\lambda\in\Lambda}P_{\lambda}$ is an $\alpha$\nobreakdash-filtered colimit of $\alpha$\nobreakdash-presentable $R$\nobreakdash-modules.
If $F=S_{\alpha}(X)$ for some $X$ in $D(R)$, then the second differential of the  spectral sequence in Remark \ref{lapuri} maps $e_{F}$ to zero,
$$
d_{2}\colon E^{1,j}_{2}\longrightarrow E^{3,j-1}_{2},\quad d_{2}(e_{F})=0.
$$
\end{lem}

\begin{proof}
The spectral sequence in Remark \ref{lapuri} identifies with the Adams spectral sequence in Section \ref{SS} below abutting to $D(R)(M,\Sigma^{j}N)=\ext^{j}_{R}(M,N)$ via the second equation in Corollary \ref{fomul}. Hence, the statement follows from Theorem \ref{calculillo} and the fact that
the following morphism is injective for $p=3$ and $q=-1$,
\begin{align}
\label{inyectiva} \ext^{p,q}_{\alpha,D(R)^{\alpha}}(S_\alpha(M),S_\alpha(\Sigma^{j}N))&\To \ext^{p,q}_{\alpha,D(R)^{\alpha}}(F,F),\\
\nonumber x&\;\mapsto\; a\cdot x\cdot b.
\end{align}
We show that it is injective for $p\geq 0$ and $q<0$.
Indeed, this morphism decomposes as
$$\ext^{p,q}_{\alpha,D(R)^{\alpha}}(S_\alpha(M),S_\alpha(\Sigma^{j}N))\st{a\cdot-}\To  \ext^{p,q}_{\alpha,D(R)^{\alpha}}(S_\alpha(M),F)\st{-\cdot b}\To \ext^{p,q}_{\alpha,D(R)^{\alpha}}(F,F).$$
The kernel of the first arrow is the image of a morphism from
$$\ext^{p-1,q}_{\alpha,D(R)^{\alpha}}(S_\alpha(M),S_\alpha(M))={\lim_{\lambda}}^{p-1}\ext_{R}^{q}(P_{\lambda},M)=0,$$
which vanishes since $q<0$. The kernel of the second arrow is the image of a morphism from the middle term of the following  exact sequence
$$\begin{array}{c}
  \ext^{p-1,q}_{\alpha,D(R)^{\alpha}}(S_\alpha(\Sigma^{j}N),S_\alpha(\Sigma^{j}N))\\
\downarrow\\
\ext^{p-1,q}_{\alpha,D(R)^{\alpha}}(S_\alpha(\Sigma^{j}N),F)\\
\downarrow\\
 \ext^{p-1,q}_{\alpha,D(R)^{\alpha}}(S_\alpha(\Sigma^{j}N),S_\alpha(M))
\end{array}$$
which vanishes since
\begin{align*}
\ext^{p-1,q}_{\alpha,D(R)^{\alpha}}(S_\alpha(\Sigma^{j}N),S_\alpha(\Sigma^{j}N))&
=
\ext^{p-1,q}_{\alpha,D(R)^{\alpha}}(S_\alpha(N),S_\alpha(N))\\
&={\lim_{\lambda}}^{p-1}\ext_{R}^{q}(N_{\lambda},N)=0,\\
\ext^{p-1,q}_{\alpha,D(R)^{\alpha}}(S_\alpha(\Sigma^{j}N),S_\alpha(M))&
=
\ext^{p-1,q}_{\alpha,D(R)^{\alpha}}(S_\alpha(N),S_\alpha(\Sigma^{-j} M))\\
&=
{\lim_{\lambda}}^{p-1}\ext_{R}^{q-j}(N_{\lambda},M)=0.
\end{align*}
Here we use that $q<0<j$.
\end{proof}

As a consequence, we obtain a sufficient condition for the existence of non-representable cohomological functors in $\Mod{\alpha}{D(R)^{\alpha}}$.

\begin{prop}\label{Prop-macho}
Let $R$ be an $\alpha$\nobreakdash-coherent ring, $\alpha>\aleph_0$. If there is an $R$\nobreakdash-module $N$ with injective dimension $\leq 1$ but $\pext_{\alpha,R}^{n}(M,N)\neq 0$ for some $R$\nobreakdash-module $M$ and some $n\geq 3$, then \aro{\alpha} fails for $D(R)$.
\end{prop}

\begin{proof}
If $n>3$ we can take an $\alpha$\nobreakdash-pure short exact sequence
 $M'\hookrightarrow P\onto M$ with $\alpha$\nobreakdash-pure projective $P$, so
$\pext_{\alpha,R}^{n}( M, N)\cong \pext_{\alpha,R}^{n-1}( M', N)$,
hence we may assume that $n=3$.

By Lemma \ref{loca} it is enough to show that $d_{2}\colon E^{1,1}_{2}\r E^{3,0}_{2}$ is non-trivial. The target is non-trivial $E^{3,0}_{2}=\pext_{\alpha,R}^{3}(M,N)\neq 0$. By degree reasons, there are no non-trivial differentials out of $E^{3,0}_{n}$, hence $E^{3,0}_{2}$ surjects onto $E_{\infty}^{3,0}\subset \ext^{3}_{R}(M,N)=0$. Therefore, all elements in $E^{3,0}_{2}$ must be in the image of an incoming differential. Since $E^{0,2}_{2}=\lim_{\lambda}\ext^{2}_R(P_{\lambda},N)=0$, then $E^{0,2}_{3}=0$ and the only possibly non-trivial incoming differential is $d_{2}\colon E^{1,1}_{2}\r E^{3,0}_{2}$, which must be surjective.
\end{proof}

We finally prove Theorem \ref{mainderived}.

\begin{proof}[Proof of Theorem \ref{mainderived}]
The first part of the statement follows from Corollary \ref{necesaria} and Proposition \ref{Thmrings1}. If $R$ is hereditary, any complex $X$ splits as a direct sum of its shifted homologies  $X\cong\bigoplus_{n\in\mathbb{Z}}\Sigma^nH_n(X)$. Therefore, on the one hand, (2) follows from Corollary \ref{necesariaysuficiente0} and Proposition \ref{Thmrings1}, and on the other hand (1) is an immediate consequence of Corollaries \ref{suficiente} and \ref{cor-rings1} and Proposition \ref{Prop-macho}.
\end{proof}

\section{On $\aleph_{1}$-Adams representability for objects and the continuum hypothesis}\label{sec-aleph_1}\numberwithin{equation}{subsection}

We already know by Corollary \ref{cota2} that if $\C T$ is an $\aleph_{1}$\nobreakdash-compactly generated triangulated category with $\card\C T^{\aleph_{1}}= \aleph_{1}$, then $\C T$ satisfies \aro{\aleph_{1}}. We have applied this result to derived categories of rings (Proposition \ref{aleph_1-Adams-for-rings}). In this section, we give further examples \emph{assuming the continuum hypothesis}. Some of these examples are based on Corollary  \ref{saltito}, the last result of this paper.

\subsection{Stable homotopy category of spectra}
The stable homotopy category of spectra $\C T=\ho(\Sp)$  is $\aleph_{0}$\nobreakdash-compactly generated and $\card \ho(\Sp)^{\aleph_{0}}\leq \aleph_{0}< \aleph_{1}$. Then $\card \ho(\Sp)^{\aleph_{1}}=\aleph_{1}$ by Corollary \ref{saltito}.

\subsection{Homotopy category of projectives modules} Let $\C T= K(\operatorname{Proj-}R)$ be the homotopy category of complexes of projective (right) modules over a ring $R$ of $\card R\leq \aleph_{1}$. This category is often not $\aleph_{0}$\nobreakdash-compactly generated, but it is always $\aleph_{1}$\nobreakdash-compactly generated, cf.~\cite{Ne08}.

\begin{prop}
Under the continuum hypothesis $\card K(\operatorname{Proj-}R)^{\aleph_{1}}\leq \aleph_{1}$.
\end{prop}

\begin{proof}
By \cite[Theorem 5.9]{Ne08},  a complex of projective
$R$\nobreakdash-modules is $\aleph_1$\nobreakdash-compact in $K(\operatorname{Proj-}R)$ if and only if it is isomorphic
in $\K(R$\nobreakdash-{\rm Proj}$)$ to a complex of free $R$\nobreakdash-modules with  $<\aleph_1$ generators. Since we are assuming the continuum hypothesis and $\card
R\leq \aleph_1$, we can proceed exactly as in the proof of
Proposition \ref{aleph_1-Adams-for-rings}.
\end{proof}

\subsection{Homotopy category of injectives modules} Let $R$ be a right Noetherian ring of $\card R\leq \aleph_{1}$. The homotopy category $\C T=K(\operatorname{Inj-}R)$  of injective (right) $R$\nobreakdash-modules is $\aleph_{0}$\nobreakdash-compactly generated  \cite{Kr05}.

\begin{prop}
Under the continuum hypothesis $\card K(\operatorname{Inj-}R)^{\aleph_{0}}\leq \aleph_{1}$.
\end{prop}

\begin{proof}
By \cite{Kr05},  $\K(\text{Inj-}R)^{\aleph_0}$ is equivalent to  the derived
category $\D^b(\operatorname{mod}(R))$ of bounded complexes of finitely presentable
$R$\nobreakdash-mod\-ules. Since $R$ is right Noetherian, $\D^b(\operatorname{mod}(R))$ is equivalent to the full subcategory of $\K(\text{Proj-}R)^{\aleph_0}$ spanned by bounded below complexes of finitely presentable projective $R$\nobreakdash-modules with bounded cohomology. Now proceed as in the proof of Proposition \ref{aleph_1-Adams-for-rings}.
\end{proof}

Finally, $\card K(\operatorname{Inj-}R)^{\aleph_{1}}\leq \aleph_{1}$ by Corollary \ref{saltito}.

\subsection{Derived category of sheaves on a non-compact manifold}
Let $M$ be a connected paracompact manifold and
$\D(\Sh\!/M)$ the derived category of the abelian category $\Sh\!/M$ of sheaves of abelian groups over $M$.
Neeman \cite{Ne01b} proved that if $M$ is  non-compact, connected and $\dim M\geq 1$, then
$\D(\Sh\!/M)$ has no non-zero
compact object, so it cannot be  $\aleph_{0}$\nobreakdash-compactly generated.

\begin{prop}
The category $\D(\Sh\!/M)$ is $\aleph_1$\nobreakdash-compactly generated and, under the continuum hypothesis, $\card \D(\Sh\!/M)^{\aleph_{1}}\leq \aleph_{1}$.
\end{prop}

\begin{proof}
Since $M$ is paracompact, we can take  a countable basis  $\{U_i\}_{i\in I}$ of the topology formed by  connected open sets of $M$.
By \cite[Section
1.9]{Gr57}, a set of generators of $\Sh\!/M$ is given by
$\{\mathbb{Z}_{U_i}\}_{i\in I}$, where
$\mathbb{Z}_{U_{i}}$ is the extension by zero of the constant sheaf $\mathbb{Z}$ on $U_{i}$.
The full small
subcategory $\C{R}$  of $\Sh\!/M$ spanned by these sheaves is the $\mathbb Z$-linear category of the quiver with vertex set   $\{U_i\}_{i\in I}$ and an arrow $U_{j}\rightarrow U_{i}$ whenever $U_{j}\subset U_{i}$. This category is clearly countable. We regard $\C R$  as a ring with several objects. The derived category $\D(\Sh\!/M)$ is a Bousfield localization $\D(\Sh\!/M)=D(\C R)/\C L_{\Sh\!/M}$ \cite[Proposition 5.1]{AJS00}. Since $\card\C R<\aleph_{1}$, the many object version of \cite[Theorem 20]{Mu} proves that the generators of the localizing subcategory $\mathcal L_{\Sh\!/M}$ described in the proof of \cite[Proposition 5.1]{AJS00} are $\aleph_{1}$\nobreakdash-compact. Hence
$\D(\Sh\!/M)$ is $\aleph_{1}$\nobreakdash-compactly generated by \cite[Theorem 4.4.9]{triang}, and the subcategory of $\aleph_{1}$\nobreakdash-compact objects is $\D(\Sh\!/M)^{\aleph_{1}}=D(\C R)^{\aleph_{1}}/\C L_{\Sh\!/M}^{\aleph_{1}}$.

Now, let us assume the continuum hypothesis. The many objects version of Proposition \ref{aleph_1-Adams-for-rings} shows that $\card D(\C R)^{\aleph_{1}}\leq \aleph_{1}$, and the explicit description of  the Verdier quotient $\D(\Sh\!/M)^{\aleph_{1}}=D(\C R)^{\aleph_{1}}/\C L_{\Sh\!/M}^{\aleph_{1}}$ proves that \mbox{$\card D(\Sh\!/M)^{\aleph_{1}}\leq \aleph_{1}$} too.
\end{proof}

\subsection{Stable motivic homotopy category}
Let $S$ be a Noetherian scheme of finite Krull
dimension. The  stable  motivic  homotopy  category  $\SH(S)$  of Morel and Voevodsky is a compactly generated triangulated category which intuitively models a homotopy theory of schemes over $S$ where the affine line $\mathbb A^{1}$ plays the role of
the unit interval in classical homotopy theory. In practice, we start with the category $Sm/S$  of smooth schemes of
finite type over $S$ endowed with the Nisnevich topology. We perform two left Bousfield localizations on the category of simplicial presheaves on $Sm/S$, one to turn maps inducing isomorphisms on homotopy sheaves into weak equivalences and another one to contract the affine line $\mathbb{A}^1$. Then we consider spectra  with respect to the suspension functor defined by smashing with the projective line $\mathbb P^{1}\simeq\mathbb{S}^1\wedge(\mathbb{A}^1-0)$ pointed at $\infty$. This yields a stable model category whose homotopy category is $\SH(S)$.

It was stated in \cite[Proposition 5.5]{Vo98} and proved in \cite[Theorem 13]{NS} that if $S$ can be covered by spectra of countable rings, then $\card\SH(S)^{\aleph_{0}}\leq \aleph_{0}<\aleph_1$, hence under the continuum hypothesis $\card\SH(S)^{\aleph_{1}}\leq \aleph_{1}$, see Corollary \ref{saltito}. The results in \cite{NS} extend straightforwardly to show that, if $S$ can be covered by spectra of rings of cardinal $\leq \aleph_{1}$, then $\card\SH(S)^{\aleph_{0}}\leq \aleph_{1}$. Therefore $\card\SH(S)^{\aleph_{1}}\leq \aleph_{1}$ under the continuum hypothesis, again by Corollary \ref{saltito}.

\section{Neeman's conjecture on Rosick\'y functors}\label{sec-Rosicky-functor}\numberwithin{equation}{section}

The following definition is due to Neeman \cite[Definition 1.19]{Ne09}.

\begin{defn}\label{def_Rosicky-functor}
Let $\C T$ be a triangulated category with (co)products.
A \emph{Rosick\'y functor} is a functor $H\colon \C T\r\C A$ to an abelian category with (co)products which takes exact triangles to exact sequences, is full, reflects isomorphisms, preserves (co)products, and there is  a small full subcategory $\C P\subset \C T$ closed under (de)susp\-ensions, formed by $\alpha$\nobreakdash-small objects in $\C T$ for some regular cardinal $\alpha$, and such that $\{H(P)\mid P\in\operatorname{Ob}\C P\}$ is a set of projective generators of $\C A$ and $H$ induces a bijection $\C T(P,X)\cong\C A(H(P),H(X))$ whenever $P$ is in $\C P$.
\end{defn}

Under the standing assumptions of Section \ref{ryf}, the restricted Yoneda functor $S_{\alpha}$ satisfies all properties of a Rosick\'y functor, taking $\C P= \C C$,  except for being full. Moreover, if $\C P=\C C=\C T^\alpha$ then $S_\alpha$ is a Rosick\'y functor if and only if \arm{\alpha} holds.

Neeman conjectured that a triangulated category has a Rosick\'y functor if and only if it is well generated in \cite[Conjecture 1.27]{Ne09}. It is easy to see that, if $\C T$ has a Rosick\'y functor, then it is well generated. We give a proof, first discovered by Rosick\'y, in this section. Neeman's conjecture is still open in the other direction. A consequence of Corollary \ref{Cor-Rosicky-func} is that it is enough to look for Rosick\'y functors of the form $S_\alpha$ for an appropriate $\C C$. Example \ref{ejemplillo} shows that we cannot always take $\C C=\C T^\alpha$ for some $\alpha$, which was the experts' first guess. Nevertheless, it is still an open question whether categories such as $D(k[[x,y]])$ possess a Rosick\'y functor.

\begin{prop}
Let $\C T$ be a triangulated category with coproducts. If there exists a Rosick\'y functor $H\colon \C T\r\C A$ then the category $\C T$ is well generated. Moreover, if  $\C C$ is the completion of $\C P$ by  coproducts of $<\alpha$ objects, then $\C C$ satisfies assumptions (1--3) in Section \ref{ryf} and $S_{\alpha}$ factors as
$$S_{\alpha}\colon \C T\st{H}\To \C A\st{i}\To \Mod{\alpha}{\C C},$$
where $i$ is fully faithful and exact.
\end{prop}

\begin{proof}
Let us first show that $\C T$ is well generated. This fact was first discovered by Rosick\'y (unpublished). We follow Definition \ref{defbasica}. The set of objects of $\C  P$ clearly satisfies (a) and (c). We now check (b). Assume that $\{f_i\colon X_i\r Y_i\}_{i\in I}$ is a set of morphisms in $\C T$ such that $\C T(P,f_i)$ is an epimorphism for all $i\in I$ and $P$ in $\C P$. Since $\C T(P,f_i)=\C A(H(P),H(f_i))$ and the objects $H(P)$ form a set of projective generators, $H(f_i)$ is an epimorphism in $\C A$ for all $i\in I$. In an abelian category, a coproduct of epimorphisms is an epimorphism. Since $H$ preserves coproducts we deduce that $H(\amalg_{
i\in I}f_i)$ is an epimorphism, hence, $\C T(P,\amalg_{i\in I}f_i)=\C A(H(P),H(\amalg_{i\in I}f_i))$ is
surjective for all $P$ in $\C P$. This proves (b). Moreover, by \cite[Lemma 5]{Kr01}, $\C P\subset\C T^\alpha$, therefore $\C C$ satisfies (1--3) in Section \ref{ryf}.

The functor $i$ is defined by $i(A)=\C A(H(-),A)$. This $\C C$\nobreakdash-module is $\alpha$\nobreakdash-continuous since $H$ preserves coproducts.
The properties of Rosick\'y functors show that  $H$ induces an equivalence between $\C C$ and its full image in $\C A$. Hence $\{H(C)\mid C\in\operatorname{Ob}\C C\}$ is also a set of projective generators of $\C A$ and $i$ is fully faithful. The composite $iH$ is naturally isomorphic to $S_{\alpha}$ since for any $X$ in $\C T$ and any coproduct $\amalg_{i\in I}P_{i}$ with $P_{i}$ in $\C P$ and $\card I<\alpha$,
\begin{align*}
S_{\alpha}(X)(\coprod_{i\in I}P_{i})&=\C T(\coprod_{i\in I}P_{i},X)
=\prod_{i\in I}\C T(P_{i},X)
\st{H}\cong
\prod_{i\in I}\C A(H(P_{i}),H(X))\\&=
\C A(\bigoplus_{i\in I}H(P_{i}),H(X))=
\C A(H(\coprod_{i\in I}P_{i}),H(X))=iH(X)(\coprod_{i\in I}P_{i}).
\end{align*}
\end{proof}

\begin{cor}\label{Cor-Rosicky-func}
A triangulated category  $\C T$ admits a Rosick\'y functor if and only if it is well generated and $S_{\alpha}$ is full for some $\C C\subset \C T$ satisfying (1--3) in Section \ref{ryf}.
\end{cor}

Recall that Corollary \ref{dimensiones2} gives us a criterion for the restricted Yoneda functor $S_\alpha$ to be full.

\begin{rem}\label{Rem-Rosicky-func}
Let $Q$ be a finite quiver without oriented cycles which is not a Dynkin quiver, $k$ an uncountable field, $kQ$ its path algebra over $k$, which is hereditary, and $\alpha$ any regular cardinal. As we showed in Example \ref{ejemplillo}, for $\C T=D(kQ)$ and $\C C=\C T^\alpha$ the functor $S_\alpha$ is never a Rosick\'y functor. Nevertheless, if $R$ is any hereditary ring, the homology functor $H_*\colon D(R)\r\Mod{}{R}^\mathbb{Z}$ to the category of $\mathbb Z$\nobreakdash-graded $R$\nobreakdash-modules is a Rosick\'y functor for $\C P$ the full subcategory spanned by $\{\Sigma^n R\}_{n\in \mathbb Z}$, here $\alpha=\aleph_0$.

These are the only known Rosick\'y functors different from the restricted Yoneda functor $S_\alpha$ with $\C C=\C T^\alpha$ for $\C T$ a category satisfying \arm{\alpha}. Triangulated categories possessing a Rosick\'y functor satisfy further properties of interest, e.g.~the Brown representability theorem for the dual, see \cite{Ne09}. Hence it would be interesting to know if there are more kinds of Rosick\'y functors.
\end{rem}

\section{Obstruction theory in triangulated categories}\label{obstructiontheory}\numberwithin{equation}{subsection}

Recall that we are under the standing assumptions of Section \ref{ryf}. In diagrams, the degree of a homogeneous morphism in $\C T$ or $\Mod{\alpha}{\C C}$ is indicated by a label in the arrow, e.g.
$$\xymatrix{X\ar[r]^f_{+n}&Y}$$
is a morphism  $f\colon X\to \Sigma^{n} Y$. We mostly consider homogeneous morphisms. We do not explicitly indicate the degree when it is $0$, when it is understood, or when it is irrelevant. Hence an exact triangle $X\r Y\r Z\r\Sigma X$ in $\C{T}$ looks like
$$\xymatrix@!=5pt{X\ar[rr]^{f}&&Y\ar[ld]^{i}\\
&Z.\ar[lu]^{q}_{\scriptscriptstyle +1}&}$$

\subsection{Phantom maps and cellular objects}

\begin{defn}
A morphism $f\colon X\r Y$ in $\C T$ is a \emph{$\C C$\nobreakdash-phantom map} if $S_\alpha(f)=0$. Moreover,  $f$ is an \emph{$n$\nobreakdash-$\C C$\nobreakdash-phantom map} if it decomposes as a product of $n$ ordinary $\C C$\nobreakdash-phantom maps, i.e.~$f=f_1\cdots f_n$ with $f_i$ a $\C C$\nobreakdash-phantom map, $1\leq i\leq n$. An \emph{$\infty$\nobreakdash-$\C C$\nobreakdash-phantom map} is a morphism $f$ which is an $n$\nobreakdash-$\C C$\nobreakdash-phantom map for all $n>0$.
\end{defn}

\begin{rem}\label{classical}
Classical phantom maps in a compactly generated triangulated category $\C T$ are precisely the $\C T^{\aleph_{0}}$-phantom maps \cite[Definition 5.1]{Ch98}. Hence,  $\C T^{\alpha}$-phantom maps deserve to be called \emph{$\alpha$-phantom maps}.

Classical phantom maps have been much studied in the literature in different contexts, see for instance \cite{Ne97, CS98, Ch98, BG99, Be00, BG01, Ben02, FH13, Ben14}. The (non)vanishing of $n$-phantom maps for $n$ high enough has attracted special attention. However, to the best of our knowledge, examples of non-trivial $\infty$-phantom maps have not been previously produced in the literature. The reader can find some in Remark \ref{oo0} below.

In the four examples of Remark \ref{alhilo}, $\C C$-phantom maps are called \emph{ghosts} \cite{Ch98,ccm}, and $\alpha$-phantom maps are ghosts for any $\alpha$. The (non)vanishing of $n$-ghosts has also drawn attention, see in addition \cite{HL09, HL11}. Remark \ref{oo0}  also contains examples of non-trivial $\infty$-ghosts, which we believe to be new.
\end{rem}

The following result is a consequence of the fact that $S_\alpha$ takes exact triangles to exact sequences.

\begin{lem}\label{util}
In an exact triangle
$$\xymatrix@!=5pt{X\ar[rr]^{f}&&Y\ar[ld]^{i}\\
&Z\ar[lu]^{q}&}$$
where we deliberately do not specify which morphism is of degree $+1$,
the following statements are equivalent:
\begin{itemize}
 \item $f$ is a $\C C$\nobreakdash-phantom map.
\item $S_\alpha(i)$ is a monomorphism.
\item $S_\alpha(q)$ is an epimorphism.
\item $S_{\alpha}(Y)\st{S_{\alpha}(i)}\hookrightarrow S_{\alpha}(Z)\st{S_{\alpha}(q)}\onto S_{\alpha}(X)$ is a short exact sequence.
\end{itemize}
\end{lem}

\begin{rem}
The class of $\C C$\nobreakdash-phantom maps forms an ideal $\C I\subset\C T$ and $n$\nobreakdash-$\C C$\nobreakdash-phantom maps form its $n^{\text{th}}$ power ideal, $\C I^{n}=\C{I}\st{n}\cdots\C{I}\subset\C T$. Moreover, $\infty$\nobreakdash-$\C C$\nobreakdash-phantom maps are the intersection ideal
$$\C I^{\infty}=\bigcap_{n>0}\C I^{n}\subset\C T.$$
\end{rem}

\begin{defn}
A \emph{$0$\nobreakdash-$\C C$-cellular object} is a trivial object in $\C T$. Moreover, $X$ is \emph{$n$\nobreakdash-$\C C$-cellular} for $n>0$ if it is a retract of an object $X'$ fitting into an exact triangle
$$
\xymatrix@!=5pt{P\ar[rr]&&Y\ar[ld]\\
&X'\ar[lu]^{\scriptscriptstyle +1}&}
$$
where $Y$ is $(n-1)$\nobreakdash-$\C C$-cellular and $P$ is in $\add(\C C)$. A \emph{$\C C$-cellular} object is an object which is $n$\nobreakdash-$\C C$-cellular for some $n\geq 0$.
\end{defn}

\begin{prop}\label{phancell}
Let $1\leq n\leq\infty$. A morphism $f\colon X\r Y$ in $\C T$ is an $n$\nobreakdash-$\C C$\nobreakdash-phantom map if and only if for any morphism $g\colon Z\r X$ from an $n$\nobreakdash-$\C C$-cellular object $Z$ we have $fg=0$. Moreover, $Z$ is an $n$\nobreakdash-$\C C$-cellular object if and only if for any morphism $g\colon Z\r X$ and any $n$\nobreakdash-$\C C$\nobreakdash-phantom map $f\colon X\r Y$ we have $fg=0$.
\end{prop}

Since $\C{C}$ is essentially small, $(\add({\C{C}}),\C{I})$ is a projective class by \cite[Lemma 3.2]{Ch98}, hence Proposition \ref{phancell} follows from \cite[Theorem 3.5]{Ch98}.

\subsection{Adams and Postnikov resolutions}

Adams resolutions go back to Adams' construction of the spectral sequence that bears his name. The definition below is due to Christensen, cf.~\cite{Ch98}.

\begin{defn}\label{ar}
An \emph{Adams resolution} $(X,W_*,P_*)$ of an object $X$ in $\C{T}$ is a countable sequence of exact triangles
\begin{equation*}
\xymatrix@!R@!C=8pt{X\ar[rr]^{j_{0}}&&W_0\ar[ld]^{\scriptscriptstyle +1}_{r_{0}}\ar[rr]^{j_{1}}&&W_1\ar[ld]^{\scriptscriptstyle +1}_{r_{1}}\ar[rr]^{j_{2}}&&W_2\ar[ld]^{\scriptscriptstyle +1}_{r_{2}}\ar[rr]^{j_{3}}&&W_3\ar[ld]^{\scriptscriptstyle +1}_{r_{3}}\ar@{}[rd]|{\displaystyle\cdots}\\
&P_0\ar[lu]^{g_{0}}&&P_1\ar[lu]^{g_{1}}&&P_2\ar[lu]^{g_{2}}&&P_3\ar[lu]^{g_{3}}&&}
\end{equation*}
such $j_{n}$ is a $\C C$\nobreakdash-phantom map and $P_{n}$ is in $\add({\C{C}})$, $n\geq 0$.
\end{defn}

\begin{rem}\label{rem-ar}
An Adams resolution of $X$ can be easily constructed by induction, as usual projective resolutions. We start with an epimorphism from a projective object $S_\alpha(P_0)\onto S_\alpha(X)$, i.e.~$P_0$ is in $\add({\C{C}})$. This morphism is represented by a unique $g_0\colon P_0\r X$. If we extend $g_0$ to an exact triangle we obtain $r_0$ and $j_0$, which is a $\C C$\nobreakdash-phantom map by Lemma \ref{util}. If we have constructed the first $n$ triangles, $n\geq 1$, we take an epimorphism from a projective object $S_\alpha(P_{n})\onto S_\alpha(W_{n-1})$ and proceed in the same way. Observe that, since $j_{n}$ is a $\C C$-phantom map for all $n\geq 0$, $S_\alpha(W_{n-1})$ is the kernel of $S_{\alpha}(r_{n-2}g_{n-1})$ for $n\geq 2$ and the kernel of $S_{\alpha}(g_{0})$ for $n=1$.

By Lemma \ref{util}, for any Adams resolution $(X,W_*,P_*)$ the restricted Yoneda functor $S_{\alpha}$ maps
$$\xymatrix@C=15pt{
0&X\ar[l]&&\ar[ll]^-{g_{0}}
P_{0}
\ar@{<-}[rr]^-{\scriptscriptstyle +1}_-{ r_{0}g_{1} }&&
 P_{1}
\ar@{<-}[rr]^-{\scriptscriptstyle +1}_-{ r_{1}g_{2} }&&
 P_{2}
\ar@{<-}[rr]^-{\scriptscriptstyle +1}_-{ r_{2}g_{3} }&&
 P_{3} \ar@{<-}[r]&\cdots}
$$
to a projective resolution of $S_\alpha(X)$ in $\Mod{\alpha}{\C C}$.
\end{rem}

Postnikov resolutions are an enrichment of Postnikov systems, whose definition we recall below, see Definition \ref{postnikov0}.

\begin{defn}\label{postrel}
A \emph{Postnikov resolution} $(X,X_{*},P_{*})$ of an object $X$ in $\C{T}$ is a diagram
$$\xymatrix@!R@!C=8pt{
X&&&&&&&&&
\\
0\ar[rr]^{i_{0}}
\ar[u]
&&X_0\ar[ld]^{\scriptscriptstyle +1}_{q_{0}}\ar[rr]^{i_{1}}\ar@(u,l)@{-}@<-.25ex>[llu]^<(.15){\scriptscriptstyle +1}_<(.05){p_{0}}&&
X_1\ar[ld]^{\scriptscriptstyle +1}_{q_{1}}\ar[rr]^{i_{2}}\ar@(u,l)@{-}[llllu]^<(.15){\scriptscriptstyle +1}_<(.05){p_{1}}&&
X_2\ar[ld]^{\scriptscriptstyle +1}_{q_{2}}\ar[rr]^{i_{3}}\ar@(u,l)@{-}[llllllu]^<(.15){\scriptscriptstyle +1}_<(.05){p_{2}}&&
X_3\ar[ld]^{\scriptscriptstyle +1}_{q_{3}}\ar@{}[rd]|{\displaystyle\cdots}
\ar@{}[ru]|{\displaystyle\cdots}
\ar@(u,l)@{->}[llllllllu]^<(.15){\scriptscriptstyle +1}_<(.05){p_{3}}\\
&P_0\ar[lu]^{f_{0}}&&P_1\ar[lu]^{f_{1}}&&P_2\ar[lu]^{f_{2}}&&P_3\ar[lu]^{f_{3}}&&}$$
consisting of a countable sequence of exact triangles and commutative triangles, $p_{n}=p_{n+1}i_{n+1}$, $n\geq 0$,
such that $S_{\alpha}$ maps
\begin{equation}\label{postrel-eq}
\xymatrix@C=15pt{
0&X\ar[l]&&\ar[ll]^-{p_{0}q_{0}^{-1}}
P_{0}
\ar@{<-}[rr]^-{\scriptscriptstyle +1}_-{ q_{0}f_{1} }&&
 P_{1}
\ar@{<-}[rr]^-{\scriptscriptstyle +1}_-{ q_{1}f_{2} }&&
 P_{2}
\ar@{<-}[rr]^-{\scriptscriptstyle +1}_-{ q_{2}f_{3} }&&
 P_{3} \ar@{<-}[r]&\cdots}
\end{equation}
to a projective resolution of $S_\alpha(X)$. In particular $X_{n}$ is $(n+1)$\nobreakdash-$\C C$-cellular.

We will denote the structure morphisms by
$f_{n}^{X}$, $i_{n}^{X}$, $q_{n}^{X}$, and $p_{n}^{X}$ when we need to distinguish between different Postnikov resolutions.
\end{defn}

\begin{lem}\label{postnikovaadams}
Given an object $X$ in $\C{T}$ and an Adams  resolution $(X,W_*,P_*)$, there exists a Postnikov resolution $(X,X_*,P_*)$ fitting in octahedra as follows, $n\geq 0$,
$$\xymatrix{
W_{n-1}
\ar[rrr]^{j_n}
\ar[ddd]_{\phi_{n-1}}
&&&
W_{n}
\ar[ddd]^{\phi_{n}}
\ar[ldd]^{r_{n}}_<(.5){\scriptscriptstyle+1\!\!}
\\
&&
X
\ar[llu]_{\;j_{n-1}\cdots j_{0}}
\ar[ru]^{j_{n}\cdots j_{0}\!\!\!}
\ar@{<-}[lldd]|!{"3,3";"1,1"}{\hole}_{p_{n-1}}^<(.4){\!\!\!\scriptscriptstyle+1}
\ar@{<-}[rdd]|!{"1,4";"3,3"}{\hole}_<(.4){\scriptscriptstyle+1\!\!}^{p_n}
&\\
&&P_{n}
\ar[lluu]^{g_{n}}
\ar[lld]^{f_n}
&\\
X_{n-1}\ar[rrr]_{i_n}
&&&
X_{n}\ar[lu]_<(.5){\scriptscriptstyle\!\!\!\!+1}^{q_n}
}$$
Here, for $n=0$ we use the convention $X_{-1}=0$, $W_{-1}=X$, and $X\r W_{-1}$ the identity  morphism. Conversely, if a Postnikov  resolution $(X,X_*,P_*)$ is given, then there exists an Adams resolution $(X,W_*,P_*)$ fitting in octahedra as above.
\end{lem}

\begin{proof}
The Postnikov resolution together with the octahedra are constructed inductively. The step $n=0$ is essentially given in the statement. We just need to choose a degree $+1$ isomorphism $q_{0}$, e.g.~$X_{0}=\Sigma P_{0}$ and $q_{0}$ the identity. In the $n^{\text{th}}$ step, we first complete $f_{n}=\phi_{n-1}g_{n}$ to an exact triangle, this yields $i_n$ and $q_n$. Then we obtain $\phi_{n}$ and $p_{n}$ by applying the octahedral axiom.

Let us tackle the converse. The Adams resolution together with the octahedra are also defined by induction. For the step $n=0$, we just need to complete $g_0=p_0q_{0}^{-1}$ to an exact triangle. This yields $j_0$, $r_0$ and $\phi_0=q_{0}^{-1}r_{0}$. Notice that $j_0$ is a $\C C$\nobreakdash-phantom map since $S_\alpha(p_0)$ is an epimorphism in $\Mod{\alpha}{\C C}$.

In the $n^{\text{th}}$ step, we first complete $p_n$ to an exact triangle, this yields $\phi_n$ and the morphism $X\r W_n$, which a fortiori will be $j_n\cdots j_0$ (so far we do not have a $j_n$). We also obtain $r_n=q_n\phi_n$. We then apply the octahedral axiom. This produces $g_n$ and $j_n$. We must check that $j_{n}$ is a $\C C$\nobreakdash-phantom, or equivalently that $S_{\alpha}(g_{n})$ is an epimorphism.

For $n=1$, we have an exact sequence
$$
\xymatrix{0&S_\alpha(X)\ar[l]&&\ar[ll]_-{S_{\alpha}(p_{0}q_{0}^{-1})}
S_\alpha(P_{0})
\ar@{<-}[rr]_-{\scriptscriptstyle +1}^-{S_\alpha(q_{0}f_{1}) }&&
 S_\alpha(P_{1}).}
$$
Since $q_0$ is an isomorphism, $\im S_\alpha(f_1)=\ker S_\alpha(p_0)$ and the exact triangle
$$
\xymatrix@!=5pt{X \ar[rd]_{j_0}&&X_0\ar[ll]^{p_0}_{\scriptscriptstyle +1} \\
&W_0\ar[ru]_{\phi_0}&}
$$
shows that $\ker S_\alpha(p_0)=\im S_\alpha(\phi_0)$. Since $j_0$ is $\C C$\nobreakdash-phantom $S_\alpha(\phi_0)$ is a monomorphism. Hence, $S_{\alpha}(g_{1})$ must an epimorphism since $f_{1}=\phi_{0}g_{1}$ and $\im S_\alpha(f_1)=\im S_\alpha(\phi_0)$.

Let $n>1$. By induction hypothesis, for $0\leq k<n$, $j_{k}$ is a $\C C$\nobreakdash-phantom and the sequences
$$
\xymatrix{0&S_\alpha(W_{k-1})\ar[l]&\ar@{->>}[l]_-{S_{\alpha}(g_k)}
S_\alpha(P_{k})
&
 S_\alpha(W_{k}) \ar@{_{(}->}[l]^-{\scriptscriptstyle +1}_-{S_\alpha(r_k) } & 0 \ar[l]}
$$
are short exact. Moreover, in the following diagram

$$
\xymatrix{
S_\alpha(P_{n-2})&&\ar[ll]^-{\scriptscriptstyle +1}_-{S_{\alpha}(q_{n-2}f_{n-1})}
S_\alpha(P_{n-1}) \ar@{->>}[ld]^-{S_{\alpha}(g_n-1)}
&&
 S_\alpha(P_{n}) \ar[ll]^-{\scriptscriptstyle +1}_-{S_\alpha(q_{n-1}f_{n}) } \ar[ld]^-{S_{\alpha}(g_{n})} \\
 & S_\alpha(W_{n-2}) \ar@{_{(}->}[lu]_-{\scriptscriptstyle +1}^-{S_\alpha(r_{n-2})} && S_\alpha(W_{n-1}) \ar@{_{(}->}[lu]_-{\scriptscriptstyle +1}^-{S_\alpha(r_{n-1})} & }
$$
the horizontal row is also an exact sequence. Hence, $
\im S_\alpha(r_{n-1})
=\im S_{\alpha}(q_{n-1}f_{n})$ and therefore $S_{\alpha}(g_{n})$ must be an epimorphism.
\end{proof}

\begin{cor}\label{postexis}
Any object $X$ in $\C T$ has a Postnikov resolution.
\end{cor}

This follows from Lemma \ref{postnikovaadams} and the fact that any object in $\C T$ has an Adams resolution, see Remark \ref{rem-ar}.

\subsection{Postnikov resolutions and $\infty$-phantom maps}\label{hcpr}

In this section we define a homotopy category of Postnikov resolutions. This is one of the key ingredients of our obstruction theory.

\begin{defn}\label{morpostres}
A \emph{morphism of Postnikov resolutions}
\begin{equation}\label{mpr}
(h,\psi_{*},\varphi_{*})\colon (X,X_{*},P_{*})\To (Y,Y_{*},Q_{*})
\end{equation}
is given by morphisms $h\colon X\r Y$, $\psi_{n}\colon X_{n}\r Y_{n}$, $\varphi_{n}\colon P_{n}\r Q_{n}$ in $\C T$, $n\geq 0$,
such that the obvious triangles and squares commute,
$$
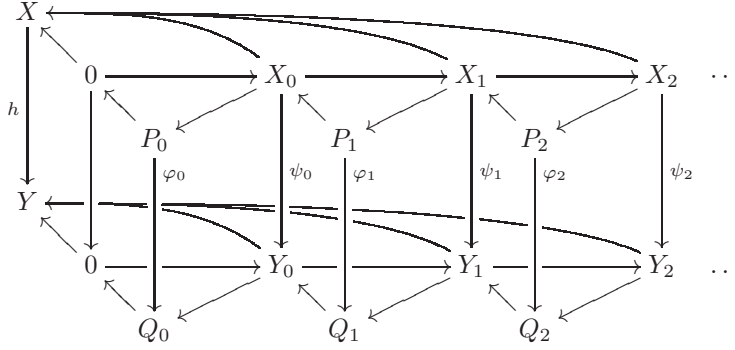
\xymatrix@!C=0pt@!R=0pt{
X\ar[ddd]_h &\\
&\ar[lu]0\ar[rrr]\ar[ddd]&&&
\ar@(lu,r)[llllu]X_0\ar[lld]^{
}\ar[rrr]\ar[ddd]^{\psi_{0}}&&&
\ar@(lu,r)@{-}[lllllllu]X_1\ar[lld]^{
}\ar[rrr]\ar[ddd]^{\psi_{1}}&&&
\ar@(lu,r)@{-}[llllllllllu]X_2\ar[lld]^{
}\ar[ddd]^{\psi_{2}} &\cdots\\
&&P_0\ar[lu]\ar[ddd]^<(.1){\varphi_{0}}&&&P_1\ar[lu]\ar[ddd]^<(.1){\varphi_{1}}&&&P_2\ar[lu]\ar[ddd]^<(.1){\varphi_{2}}&&&&\\
Y&&&&\\
&\ar[lu]0\ar[rrr]|!{"2,3";"5,3"}{\hole} &&&
\ar@(lu,r)[llllu]|!{"2,3";"5,3"}{\hole}|!{"2,2";"5,2"}{\hole}
Y_0\ar[lld]^{
}\ar[rrr]|!{"2,6";"5,6"}{\hole}
&&&
\ar@(lu,r)@{-}[lllllllu]|!{"2,6";"5,6"}{\hole}|!{"2,5";"5,5"}{\hole}|!{"2,3";"5,3"}{\hole}|!{"2,2";"5,2"}{\hole}
Y_1\ar[lld]^{
}\ar[rrr]|!{"2,9";"5,9"}{\hole}&&&
\ar@(lu,r)@{-}[llllllllllu]|!{"2,9";"5,9"}{\hole}|!{"2,8";"5,8"}{\hole}|!{"2,6";"5,6"}{\hole}|!{"2,5";"5,5"}{\hole}|!{"2,3";"5,3"}{\hole}|!{"2,2";"5,2"}{\hole}
Y_2\ar[lld]^{
} &\dots\\
&&Q_0\ar[lu]&&&Q_1\ar[lu]&&&Q_2\ar[lu]&&&}$$
A pair of morphisms of Postnikov resolutions
\begin{equation*}
(h,\psi_{*},\varphi_{*}),(\bar h,\bar \psi_{*},\bar \varphi_{*})\colon (X,X_{*},P_{*})\To (Y,Y_{*},Q_{*})
\end{equation*}
are \emph{homotopic} $(h,\psi_{*},\varphi_{*})\simeq(\bar h,\bar \psi_{*},\bar \varphi_{*})$ if, for all $n>0$, the following equivalent conditions hold:
\begin{enumerate}
\item $\psi_{n}i_{n}^X=\bar{\psi}_{n}i_{n}^X$,
 \item $i_{n}^Y\psi_{n-1}=i_{n}^Y\bar \psi_{n-1}$,
\item $\psi_{n}- \bar\psi_{n}$ factors through $q_{n}^X\colon X_{n}\r P_{n}$,
\item $\psi_{n-1}- \bar\psi_{n-1}$ factors through $f_{n}^Y\colon Q_{n}\r Y_{n-1}$.
\end{enumerate}
This natural equivalence relation is additive: two morphisms are homotopic iff their difference $(h-\bar h,\psi_{*}-\bar \psi_{*},\varphi_{*}-\bar \varphi_{*})$ is nullhomotopic, i.e.~homotopic to the trivial map. We denote by $\mathbf{Pres}_\infty$ the category of Postnikov resolutions and $\mathbf{Pres}_\infty^\simeq$ its homotopy category. Both of them are additive.
\end{defn}

The following theorem is the main result of this section. It establishes the existence of a functor with a certain property. Usually, when defining a functor, the complicated part is to show that composition is preserved. In this case the complicated part is the definition of the functor on morphisms, once this is achieved compatibility with composition is obvious.

\begin{thm}\label{cociente}
There exists an essentially unique functor
$$\Psi\colon\C T\To \mathbf{Pres}_\infty^\simeq$$
sending an object $X$ to a Postnikov resolution $\Psi(X)$ of $X$ and a map $h\colon X\r Y$ to the homotopy class
$\Psi(h)$ of a morphism with first coordinate $h$. This functor is additive, full and essentially surjective. Moreover, the kernel of $\Psi$ is the ideal $\C I^{\infty}$ of $\infty$\nobreakdash-$\C C$\nobreakdash-phantom maps, hence $\Psi$ induces an equivalence of categories
$\C T/\C I^{\infty}\simeq  \mathbf{Pres}_\infty^\simeq$.
\end{thm}

We prove Theorem \ref{cociente} at the end of this section. See Remark \ref{oo0} below for explicit examples where the ideal $\C I^{\infty}$ of $\infty$\nobreakdash-$\C C$\nobreakdash-phantom maps is non-trivial and hence $\Psi$ is not an equivalence of categories.

\begin{lem}\label{unsplit}
Given a Postnikov resolution $(X,X_{*},P_{*})$, the following sequence is exact for $n\geq 0$,
$$\xymatrix{S_{\alpha}(P_{n+1})\ar[rr]^{S_{\alpha}(f_{n+1})}&& S_{\alpha}(X_{n})\ar@{->>}[rr]^{S_{\alpha}(p_{n})}_-{\scriptscriptstyle +1}&&  S_{\alpha}(X).}$$
Moreover, $S_{\alpha}(p_{n})$ splits for $n>0$.
\end{lem}

\begin{proof}
For $n=0$ it holds by definition since $q_{0}$ is an isomorphism. For $n>0$, consider an associated Adams resolution via Lemma \ref{postnikovaadams}. Since $j_{n}\cdots j_{0}$ and $j_{n+1}$ are $\C C$\nobreakdash-phantoms
$$\xymatrix@C=40pt@R=5pt{S_{\alpha}(W_{n})\ar@{^{(}->}[r]^{S_{\alpha}(\phi_{n})}& S_{\alpha}(X_{n})\ar@{->>}[r]^{S_{\alpha}(p_{n})}_-{\scriptscriptstyle +1}&  S_{\alpha}(X)\\
S_{\alpha}(W_{n+1})\ar@{^{(}->}[r]^{S_{\alpha}(r_{n+1})}_-{\scriptscriptstyle +1}& S_{\alpha}(P_{n+1})\ar@{->>}[r]^{S_{\alpha}(g_{n+1})}&  S_{\alpha}(W_{n})}$$
are short exact by Lemma \ref{util}, and $f_{n+1}=\phi_{n}g_{n+1}$, hence the sequence in the statement is exact.

Now let $n>0$. Recall that the sequence
$$
\xymatrix{0&S_\alpha(X)\ar[l]&&\ar[ll]_-{S_{\alpha}(p_{0}q_{0}^{-1})}
S_\alpha(P_{0})
\ar@{<-}[rr]_-{\scriptscriptstyle +1}^-{S_\alpha(q_{0}f_{1}) }&&
 S_\alpha(P_{1}).}
$$
is exact. The map $S_{\alpha}(i_{n}\cdots i_{1}q_{0}^{-1})\colon S_{\alpha}(P_{0})\st{\scriptscriptstyle -1}\r S_{\alpha}(X_{n})$ factors uniquely through $S_{\alpha}(p_{0}q_{0}^{-1})\colon S_{\alpha}(P_{0})\onto S_{\alpha}(X)$ since $(i_{n}\cdots i_{1}q_{0}^{-1})(q_{0}f_{1})=i_{n}\cdots i_{1}f_{1}=0$. The factorization $S_{\alpha}(X)\st{\scriptscriptstyle -1}\rightarrow S_{\alpha}(X_{n})$ composed with $S_{\alpha}(p_{n})$ is the identity in $S_{\alpha}(X)$ since $p_{n}(i_{n}\cdots i_{1}q_{0}^{-1})=p_{0}q_{0}^{-1}$, hence we are done.
\end{proof}

\begin{cor}
If $X$ is an object in $\C T$ such that $S_{\alpha}(X)$ has projective dimension $\leq n$ in $\Mod{\alpha}{\C C}$ then $X$ is $(n+1)$-$\C C$-cellular.
\end{cor}

\begin{proof}
Let $S_{\alpha}(P_{*})$ be a projective resolution of $S_{\alpha}(X)$ of length $\leq n$. Using Remark \ref{rem-ar}, we can construct an Adams resolution $(X,W_{*},P_{*})$, and using Lemma \ref{postnikovaadams}, a Postnikov resolution $(X,X_{*},P_{*})$. By the condition on the length and Lemma \ref{unsplit}, $S_{\alpha}(p_{n})$ is an isomorphism. Hence, $p_{n}\colon X_{n}\cong X$  is also an isomorphism, and $X_{n}$ is $(n+1)$-$\C C$-cellular by definition.
\end{proof}

\begin{prop}\label{sextiende}
Given a morphism $h\colon X\r Y$ in $\C T$ and Postnikov resolutions $(X,X_*,P_*)$ and $(Y,Y_*,Q_*)$ there exists a morphism of Postnikov resolutions as in \eqref{mpr} extending $h$.
\end{prop}

\begin{proof}
We proceed by induction. The morphisms $\varphi_{0}$ and $\varphi_{1}$ can be constructed by completing the following diagram of exact rows
$$\xymatrix@C=40pt{
S_{\alpha}(P_{1})\ar[r]_-{\scriptscriptstyle +1}^-{S_{\alpha}(q_{0}f_{1})}&S_{\alpha}(P_{0})\ar@{->>}[r]^-{S_{\alpha}(p_{0}q_{0}^{-1})}&S_{\alpha}(X)\ar[d]^{S_{\alpha}(h)}\\
S_{\alpha}(Q_{1})\ar[r]_-{\scriptscriptstyle +1}^-{S_{\alpha}(q_{0}f_{1})}&S_{\alpha}(Q_{0})\ar@{->>}[r]^-{S_{\alpha}(p_{0}q_{0}^{-1})}&S_{\alpha}(Y)
}$$
to commutative squares, and $\psi_0=(q_0^Y)^{-1}\varphi_0q_0^X$.

Assume we have constructed up to the following diagram of solid arrows
$$
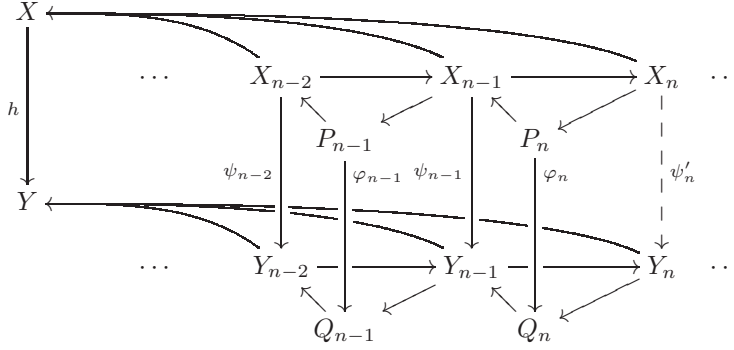
\xymatrix@!C=0pt@!R=0pt{
X\ar[ddd]_h\\
&&\cdots&&
\ar@(lu,r)[llllu]X_{n-2}\ar[rrr]\ar[ddd]_{\psi_{n-2}}&&&
\ar@(lu,r)@{-}[lllllllu]X_{n-1}\ar[lld]^{
}\ar[rrr]\ar[ddd]_{\psi_{n-1}}&&&
\ar@(lu,r)@{-}[llllllllllu]X_{n}\ar@{-->}[ddd]^{\psi'_n}\ar[lld]^{
}&\cdots\\
&&&&&P_{n-1}\ar[lu]\ar[ddd]^<(.1){\varphi_{n-1}}&&&P_{n}\ar[lu]\ar[ddd]^<(.1){\varphi_{n}}&&&\\
Y&&&\\
&&\cdots&&
\ar@(lu,r)[llllu]
Y_{n-2}\ar[rrr]|!{"2,6";"5,6"}{\hole}
&&&
\ar@(lu,r)@{-}[lllllllu]|!{"2,6";"5,6"}{\hole}|!{"2,5";"5,5"}{\hole}
Y_{n-1}\ar[lld]^{
}\ar[rrr]|!{"2,9";"5,9"}{\hole}&&&
\ar@(lu,r)@{-}[llllllllllu]|!{"2,9";"5,9"}{\hole}|!{"2,8";"5,8"}{\hole}|!{"2,6";"5,6"}{\hole}|!{"2,5";"5,5"}{\hole}
Y_{n}\ar[lld]^{
}&\cdots\\
&&&&&Q_{n-1}\ar[lu]&&&Q_{n}\ar[lu]&&&}$$
for some $n>0$. We choose $\psi_n'$ (the dashed arrow) extending $\psi_{n-1}$ and $\varphi_{n}$ to a morphism of exact triangles. In general,
\begin{equation}\label{cuadradotonto}
\xymatrix{
X_{n}\ar[r]^-{p_{n}^X}_-{\scriptscriptstyle +1}\ar[d]_{\psi_{n}'}&X\ar[d]^{h}\\
Y_{n}\ar[r]^-{p_{n}^Y}_-{\scriptscriptstyle +1}&Y
}
\end{equation}
does not commute, but precomposing with $i^X_{n}\colon X_{n-1}\r X_{n}$,
\begin{align*}
p_{n}^{Y}\psi_{n}'i_{n}^{X}&=p_{n}^{Y}i_{n}^{Y}\psi_{n-1}=p_{n-1}^{Y}\psi_{n-1}=hp_{n-1}^{X}=hp_{n}^{X}i_{n}^{X}.
\end{align*}
Hence, $hp_{n}^{X}-p_{n}^{Y}\psi_{n}'$ factors as
$$
\xymatrix{X_{n}\ar[r]_{\scriptscriptstyle +1}^{q_{n}^X} & P_{n}\ar[r]^{\beta} & Y.}
$$
The composite
$$\xymatrix{
S_\alpha(P_{n})\ar[r]^-{S_\alpha(\beta)} & S_\alpha(Y)\ar@{^{(}->}[rr]_-{\scriptscriptstyle-1}^-{
\begin{array}{c}
\scriptstyle\text{splitting in}\\[-5pt]\scriptstyle\text{the proof of}\\[-5pt]\scriptstyle\text{Lemma \ref{unsplit}}
\end{array}}&&  S_\alpha(Y_{n}),
}$$
is the image by $S_\alpha$ of a unique $\gamma\colon P_{n}\st{\scriptscriptstyle -1}\r Y_{n}$ since $S_\alpha(P_n)$ is projective. This morphism satisfies $p_n^Y\gamma=\beta$ and $q_{n}^{Y}\gamma =0$. The first equation holds by the splitting condition. For the second equation it is enough to check that $S_{\alpha}(q_{n}^{Y}\gamma) =0$, and this holds since the splitting of $S_{\alpha}(p_{n})$ in Lemma \ref{unsplit} is induced by $S_{\alpha}(i_{n}\cdots i_{1}q_{0}^{-1})$ and $q_{n}^Y i_{n}^Y=0$. Hence, the morphism $\psi_{n}=\psi_{n}'+\gamma q_{n}^{X}$ still extends $\psi_{n-1}$ and $\varphi_{n}$ to a morphism of exact triangles since
\begin{align*}
q_n^Y\psi_n&=q_n^Y\psi_n'+q_n^Y\gamma q_n^X=\varphi_nq_n^X+0q_n^X=\varphi_nq_n^X,\\
\psi_ni_n^X&=\psi_n'i_n^X+\gamma q_n^X i_n^X=i_n^Y\psi_{n-1}+\gamma0=i_n^Y\psi_{n-1}.
\end{align*}
Moreover, the square \eqref{cuadradotonto} commutes if we replace $\psi_{n}'$ with $\psi_{n}$ since
\begin{align*}
p_n^Y\psi_n&=p_n^Y\psi_n'+p_n^Y\gamma q_n^X=p_n^Y\psi_n'+\beta q_n^X=p_n^Y\psi_n'+hp_n^X-p_n^Y\psi_n'=hp_n^X.
\end{align*}

In order to conclude the induction step we must take $\varphi_{n+1}\colon P_{n+1}\r Q_{n+1}$ completing
$$\xymatrix{P_{n+1}\ar[r]^-{f_{n+1}^X}&X_{n}\ar[d]^{\psi_{n}}\ar[r]^-{p_{n}^X}_-{\scriptscriptstyle +1}&X\ar[d]^{h}\\
Q_{n+1}\ar[r]^-{f_{n+1}^Y}&Y_{n}\ar[r]^-{p_{n}^Y}_-{\scriptscriptstyle +1}&Y}$$
to a commutative diagram. This can be done. Actually, by Lemma \ref{unsplit}, it is enough to notice that $S_\alpha(P_{n+1})$ is projective and that $p_{n}^{Y}\psi_{n}f_{n+1}^{X}=hp_{n+1}^{X}f_{n+1}^{X}=h0=0$.
\end{proof}

\begin{prop}\label{phantom1}
If $(X,X_*,P_*)$ is a Postnikov resolution, then $h\colon X\r Y$ is an $n$\nobreakdash-$\C C$\nobreakdash-phantom map, $n>0$, if and only if $hp_{n-1}=0$. In particular, $h$ is an $\infty$\nobreakdash-$\C C$\nobreakdash-phantom map if and only if $hp_{n}=0$ for all $n\geq 0$.
\end{prop}

\begin{proof}
Since $X_{n-1}$ is $n$\nobreakdash-$\C C$-cellular, if $h$ is an $n$\nobreakdash-$\C C$\nobreakdash-phantom map, then $hp_{n-1}=0$, see Proposition \ref{phancell}. Conversely, by Lemma \ref{postnikovaadams} the morphism $p_{n-1}$ fits in an exact triangle
$$\xymatrix{X_{n-1}\ar[rr]^{p_{n-1}}_{\scriptscriptstyle +1}&& X\ar[ld]^{j_{n-1}\cdots j_{0}}\\
&W_{n-1}\ar[lu]^{\phi_{n-1}}&}$$
with $j_{n-1}\cdots j_{0}$ an $n$\nobreakdash-$\C C$\nobreakdash-phantom map. Therefore, if $hp_{n-1}=0$ then $h$ factors through $j_{n-1}\cdots j_{0}$, so $h$ is an $n$\nobreakdash-$\C C$\nobreakdash-phantom map too.
\end{proof}

\begin{prop}\label{nulinf}
A morphism of Postnikov resolutions as in \eqref{mpr} is nullhomotopic if and only if $h$ is an $\infty$\nobreakdash-$\C C$\nobreakdash-phantom map.
\end{prop}

\begin{proof}
If we assume that $(h,\psi_{*},\varphi_{*})$ is  nullhomotopic, then $\psi_n$ factors through $f_{n+1}^Y$  for all $n\geq 0$. By Lemma \ref{unsplit}, $p_n^Yf_{n+1}^Y=0$ and then, $hp_n^X=p_n^Y\psi_n=0$. Hence, $h$ is an $\infty$\nobreakdash-$\C C$\nobreakdash-phantom map by Corollary \ref{phantom1}.

Assume now that $h$ is an $\infty$\nobreakdash-$\C C$\nobreakdash-phantom map. We construct by induction on $n\geq 0$ a map $\beta_n\colon P_n\st{\scriptscriptstyle-1}\r Q_{n+1}$ such that the following square commutes
$$
\xymatrix{X_n \ar[r]^{q_n^X}_{\scriptscriptstyle +1} \ar[d]_{\psi_n} & P_n \ar[d]^{\beta_n}_{\scriptscriptstyle -1} \\ Y_n & Q_{n+1}.\ar[l]^{f^Y_{n+1}}}
$$
For $n=0$, the following diagram with exact rows
$$\xymatrix@C=40pt{
S_{\alpha}(P_{1})\ar[r]_{\scriptscriptstyle +1}^-{S_{\alpha}(q_{0}f_{1})}\ar[d]^{S_{\alpha}(\varphi_{1})}&S_{\alpha}(P_{0})\ar@{->>}[r]^-{S_{\alpha}(p_{0}q_{0}^{-1})}\ar[d]^{S_{\alpha}(\varphi_{0})}&S_{\alpha}(X)\ar[d]^{S_{\alpha}(h)=0}\\
S_{\alpha}(Q_{1})\ar[r]_{\scriptscriptstyle +1}^-{S_{\alpha}(q_{0}f_{1})}&S_{\alpha}(Q_{0})\ar@{->>}[r]^-{S_{\alpha}(p_{0}q_{0}^{-1})}&S_{\alpha}(Y)
}$$
shows that we can take $\beta_{0}\colon P_{0}\r Q_{1}$ with $\varphi_{0}=q_{0}^{Y}f_{1}^{Y}\beta_{0}$. This choice of $\beta_0$ works since $\psi_{0}=(q_{0}^{Y})^{-1}\varphi_{0}q_{0}^{X}$.

Assume we have checked our claim up to $n-1$.
Choose an Adams resolution $(Y,W_*,Q_*)$ associated to the Postnikov resolution $(Y,Y_*,Q_*)$ in the sense of Lemma \ref{postnikovaadams}. We use the notation therein, exchanging $X$ and $P$  with $Y$ and $Q$, respectively. Since $h$ is an $\infty$\nobreakdash-$\C C$\nobreakdash-phantom map, by Corollary \ref{phantom1},
\begin{align*}
p_n^Y\psi_n&
=hp_n^X=0,
\end{align*}
so $\psi_{n}$ factors as $X_n\st{\gamma_n}\r W_n\st{\phi_{n}}\r Y_{n}$. By induction hypothesis,
\begin{align*}
\phi_n\gamma_n i_{n}^X &=\psi_{n}i_n^X=i_n^Y \psi_{n-1}=i_n^Yf_n^Y\beta_{n-1}q_{n-1}^X=0\beta_{n-1}q_{n-1}^X=0.
\end{align*}
Since $j_n\cdots j_0$ is an $(n+1)$\nobreakdash-$\C C$\nobreakdash-phantom map and $X_{n-1}$, the source of $i_n^X$, is $n$\nobreakdash-$\C C$-cellular, the homomorphism
$$\C{T}(X_{n-1},\phi_n)\colon \C{T}(X_{n-1},W_n)\To \C{T}(X_{n-1},Y_n)$$
is injective, so the previous equation yields $\gamma_n i_{n}^X=0$.
Hence, $\gamma_{n}$ factors as
$$\xymatrix{X_{n}\ar[r]^-{q_{n}^X}_-{\scriptscriptstyle +1}& P_{n}\ar[r]^-{\varepsilon_{n}}_-{\scriptscriptstyle -1}& W_{n}.}$$
Furthermore, since $j_{n+1}$ is a $\C C$\nobreakdash-phantom map, $S_\alpha(g_{n+1})\colon S_\alpha(Q_{n+1})\r S_\alpha(W_{n})$ is an epimorphism and we can factor $\varepsilon_{n}$ as
$$\xymatrix{P_{n}\ar[r]^-{\beta_{n}}_-{\scriptscriptstyle -1}& Q_{n+1}\ar[r]^-{g_{n+1}}& W_{n}.}$$
Finally, $f_{n+1}^Y\beta_nq_n^X=\phi_ng_{n+1}\beta_nq_n^X=\phi_n\varepsilon_nq_n^X=\phi_n\gamma_n=\psi_n$.
\end{proof}

\begin{proof}[Proof of Theorem \ref{cociente}]
Any object $X$ in $\C T$ has a Postnikov resolution $\Psi(X)$ by Corollary \ref{postexis}. We choose one. Proposition \ref{sextiende} proves that there are choices for $\Psi(h)$ as in the statement. Moreover, the choice is unique in the homotopy category by Proposition \ref{nulinf}. By uniqueness, $\Psi$ must be an additive functor. Propositions \ref{sextiende} and \ref{nulinf} prove that any two Postnikov resolutions of $X$ are isomorphic in $\mathbf{Pres}_\infty^\simeq$ via a unique homotopy class extending the identity in $X$, hence $\Psi$ is essentially unique. Moreover,  $\Psi$ is full since the homotopy class of an arbitrary morphism $(h,\psi_*,\varphi_*)\colon \Psi(X)\r\Psi(Y)$ is $\Psi(h)$. Finally, the kernel of $\Psi$ is $\C I^{\infty}$ by Proposition~\ref{nulinf}.
\end{proof}

\subsection{Homotopy colimits and Postnikov resolutions}

Recall that a \emph{homotopy colimit} \cite[Definition 1.6.4]{triang} of a sequence in a triangulated category with countable coproducts $\C{T}$
\begin{equation*}
\xymatrix@C=20pt{X_{0}\ar[r]^-{i_{1}}&X_{1}\ar[r]^-{i_{2}}&X_{2}\ar[r]^-{i_{3}}&X_{3}\ar[r]&\cdots}
\end{equation*}
is an exact triangle
\begin{equation}\label{hcet}
\xymatrix{\displaystyle\coprod\limits_{n> 0}X_{n}\ar[rr]^{\eqref{hcet2}}&&\displaystyle\coprod\limits_{n> 0}X_{n}\ar[ld]^{\scriptscriptstyle +1}_{(p_{n}')_{n> 0}}\\
&\Hocolim\limits_{n}X_{n}\ar[lu]^{\delta'}&}
\end{equation}
where the upper arrow is given by the following matrix
\begin{equation}\label{hcet2}
\left(
\begin{array}{ccccc}
1&0&0&0&\cdots\\\
-i_{2}&1&0&0\\
0&-i_{3}&1&0\\
0&0&-i_{4}&1\\
\vdots&&&&\ddots
\end{array}
\right).
\end{equation}
Usually, $\delta'$ is taken to be the degree $+1$ map, but the previous convention is more convenient for our purposes. Moreover, $X_{0}$ and $i_{1}$ are usually not neglected in \eqref{hcet} and \eqref{hcet2}, but the construction turns out to be equivalent, see \cite[Lemma 1.7.1]{triang}.

\begin{prop}\label{hocolimpostres}
Given a Postnikov resolution $(X,X_{*},P_{*})$, there is a homotopy colimit given by an exact triangle of the form
\begin{equation*}
\xymatrix{\displaystyle\coprod\limits_{n> 0}X_{n}\ar[rr]^{\eqref{hcet2}}&&\displaystyle\coprod\limits_{n> 0}X_{n}\ar[ld]^{\scriptscriptstyle +1}_{(p_{n})_{n> 0}}\\
&X.\ar[lu]^{\delta}&}
\end{equation*}
\end{prop}

In the proof of Proposition \ref{hocolimpostres} we use the following lemma.

\begin{lem}\label{replace}
Given morphisms $$\xymatrix{X\ar[r]^{f} & Y\ar[r]^{i}_{\scriptscriptstyle +1} & Z}$$ such that $if=0$ and
$$\xymatrix{S_{\alpha}(X)\ar@{^{(}->}[r]^{S_{\alpha}(f)}& S_{\alpha}(Y)\ar@{->>}[r]^{S_{\alpha}(i)}_{\scriptscriptstyle +1}& S_{\alpha}(Z)}$$
is a short exact sequence, there is an exact triangle
$$\xymatrix{X\ar[rr]^{f}&&Y\ar[ld]^{\scriptscriptstyle +1}_{i}\\
&Z.\ar[lu]^{q}&}$$
\end{lem}

\begin{proof}
Complete $f$ to an exact triangle
$$\xymatrix{X\ar[rr]^{f}&&Y\ar[ld]^{\scriptscriptstyle +1}_{i'}\\
&Z'.\ar[lu]^{q'}&}$$
Since $if=0$, $i$ factors as $Y\st{i'}\r Z'\st{\phi}\r Z$. Since $S_{\alpha}(f)$ is a monomorphism, the sequence
$$\xymatrix{S_{\alpha}(X)\ar@{^{(}->}[r]^{S_{\alpha}(f)}& S_{\alpha}(Y)\ar@{->>}[r]^{S_{\alpha}(i')}_{\scriptscriptstyle +1}& S_{\alpha}(Z')}$$
is also short exact by Lemma \ref{util}. Therefore, $S_{\alpha}(\phi)$ is an isomorphism. Finally, since $S_{\alpha}$ reflects isomorphisms, $\phi$ is an isomorphism and we can take $q=q'\phi^{-1}$.
\end{proof}

\begin{proof}[Proof of Proposition \ref{hocolimpostres}]
Clearly, $(p_{n})_{n>0}\eqref{hcet2}=0$ since $p_{n}=p_{n+1}i_{n+1}$, $n>0$. Using the splitting $S_{\alpha}(X_{n})\cong S_{\alpha}(X)\oplus\im S_{\alpha}(f_{n+1})$ given by Lemma \ref{unsplit}, $n>0$, and the fact that $S_{\alpha}$ preserves coproducts, we can identify   $S_{\alpha} \eqref{hcet2}$ with  the endomorphism of
\begin{equation}\label{split1}
\displaystyle\left(\bigoplus\limits_{n> 0}S_{\alpha}(X)\right)\oplus \left(\bigoplus\limits_{n> 0}\im S_{\alpha}(f_{n+1})\right)
\end{equation}
which decomposes as the identity on the second factor, since $i_{n}f_{n}=0$, and the endomorphism defined by the matrix
\begin{equation}\label{split2}
\left(
\begin{array}{rrrrr}
1&0&0&0&\cdots\\
-1&1&0&0\\
0&-1&1&0\\
0&0&-1&\;\;1\\
\vdots&&&&\ddots
\end{array}
\right)
\end{equation}
on the first factor, since the splitting $S_\alpha(X)\hookrightarrow S_\alpha(X_n)$ is induced by $S_\alpha(i_n\cdots i_1q_0^{-1})$, see the proof of Lemma \ref{unsplit}.

The endomorphism \eqref{split2}, and hence $S_{\alpha}\eqref{hcet2}$, is a split monomorphism. The matrix
$$\left(
\begin{array}{rrrrr}
0&-1&-1&-1&\cdots\\
0&0&-1&-1\\
0&0&0&-1\\
0&0&0&0\\
\vdots&&&&\ddots
\end{array}
\right)$$
defines a retraction of \eqref{split2}. The cokernel of $\eqref{split2}\coprod \id{}$ is $S_{\alpha}(X)$. The natural projection is $0$ on the second factor of \eqref{split1} and
$$\left(\begin{array}{ccccc}
1&1&1&1&\cdots
\end{array}\right)
$$
on the first factor. This morphism identifies with $S_{\alpha}(p_{n})_{n>0}$ via the direct sum decomposition, since $p_{n}f_{n+1}=0$ by Lemma \ref{unsplit}. Therefore, Lemma \ref{replace} applies.
\end{proof}

\begin{cor}\label{phantom2.5}
In the conditions of the statement of  Proposition \ref{hocolimpostres}, $h\colon X\r Y$ is an $\infty$-$\C C$-phantom  if and only if it factors as $h=h'\delta$,
$$X\st{\delta}\To \coprod_{n>0}X_{n}\st{h'}\To Y.$$
In particular, $\delta$ is an $\infty$-$\C C$-phantom map.
\end{cor}

\begin{proof}
It is enough to notice that, by Proposition \ref{phantom1}, $h$ is an $\infty$\nobreakdash-$\C C$\nobreakdash-phantom if and only if $0=(hp_{n})_{n>0}=h(p_{n})_{n>0}$. The rest follows from elementary properties of the homotopy colimit exact triangle in  Proposition \ref{hocolimpostres}.
\end{proof}

The following corollary is a new result. It should be compared to the fact that if $\aleph_0$\nobreakdash-Adams representability holds then the ideal of $\C C$\nobreakdash-phantom maps is a square zero ideal, cf.~\cite{Ne97}. Actually, one can check along the same lines that this is also true under $\alpha$\nobreakdash-Adams representability. For an arbitrary well generated triangulated category we have the following result.

\begin{cor}\label{phantom3}
The ideal $\C I^{\infty}$ of $\infty$-$\C C$\nobreakdash-phantom maps is a square zero ideal $(\C I^{\infty})^{2}=0$, i.e.~if $h\colon X\r Y$ and $k\colon Y\r Z$ are $\infty$\nobreakdash-$\C C$\nobreakdash-phantom maps, then $kh=0$.
\end{cor}

\begin{proof}
Factor $h$ as in Corollary \ref{phantom2.5}. Since $k$ is an $\infty$\nobreakdash-$\C C$\nobreakdash-phantom map and each $X_n$ is $(n+1)$\nobreakdash-$\C C$-cellular, $n>0$, $kh'=0$ by Proposition \ref{phancell}. Hence, $kh=kh'\delta=0\delta=0$.
\end{proof}

\begin{rem}\label{wlinext}
Theorem \ref{cociente} and Corollary \ref{phantom3} show that
$$\C I^\infty\into\C T\st{\Psi}\onto \mathbf{Pres}_\infty^\simeq$$
is a weak linear extension \cite[Definition II.1.7]{B91}, therefore the $\C T$\nobreakdash-bimodule $\C I^\infty$ is actually a $\mathbf{Pres}_\infty^\simeq$\nobreakdash-bimodule and the weak linear extension is classified up to equivalence by a class in cohomology of categories
$$\{\C T\}\in H^2(\mathbf{Pres}_\infty^\simeq,\C I^\infty).$$
This can be compared to the fact that, under $\aleph_0$\nobreakdash-Adams representability (and also under $\alpha$\nobreakdash-Adams representability replacing $\aleph_0$ with $\alpha$, as one can easily deduce from the results of this paper) $\C T$ is a linear extension of the full subcategory of $\aleph_0$\nobreakdash-flat objects in $\Mod{\aleph_0}{\C T^{\aleph_0}}$ by $\C I$, cf.~\cite[\S5]{CS98}.
\end{rem}

\begin{rem}\label{oo0}
Here we construct examples of non-trivial $\infty$-$\C C$-phantom maps in the derived category $D(R)$ of appropriate rings $R$.
 This shows that the functor $\Psi$ in Theorem \ref{cociente} is not always an equivalence of categories. Our $\C C$ is such that $\C C$-phantom maps are the same as ghosts, i.e.~morphisms in the derived category inducing trivial morphisms in homology. For some of these rings, $\C C=D(R)^{\aleph_{0}}$, so we obtain non-trivial $\infty$-phantom maps in the classical sense, see Remark \ref{classical}.

Let us place ourselves in the context of Remark \ref{alhilo} (1), i.e.~$\C T=D(R)$ is the derived category of a ring $R$, $\alpha$ is any cardinal and $\C C\subset\C T$ is the smallest full subcategory closed under (de)suspensions, coproducts of less than $\alpha$ objects, and retracts, containing $R$. Up to isomorphism, an object in $\C C$ is a complex of $\alpha$\nobreakdash-presentable projective $R$-modules with trivial differential, in addition bounded if $\alpha=\aleph_{0}$. We have an equivalence with the category graded $R$-modules $\Mod{\alpha}{\C C}\simeq\Mod{}{R}^{\mathbb Z}$ for all $\alpha$, and the restricted Yoneda functor identifies with the homology functor $X\mapsto H_{*}(X)$. Hence, $\C C$-phantom maps are precisely ghosts.

Assuming the existence of a simple $R$-module $X$ of infinite projective dimension, we now construct a non-trivial $\infty$-$\C C$-phantom map out of $X$ in $D(R)$. For instance, $R$ can be a commutative Noetherian local ring of infinite global dimension and $X$ the residue field. Maybe a more important example is $R$ the free Boolean algebra on a set $S$ of cardinality $\geq\aleph_{\omega}$ and $X=R/(S)=\mathbb F_{2}$, see \cite{tgdbr}. In this case, since $R$ is von Neumann regular, $\C C=\C T^{\aleph_{0}}$ for $\alpha=\aleph_{0}$ and we obtain an example of a non-trivial $\infty$-phantom map.

Out of any projective resolution of $X$ in the category of $R$-modules,
\[\xymatrix@C=10pt{
0&X\ar[l]&\ar[l]_{\varepsilon}
P_{0}
\ar@{<-}[r]&
 P_{1}
\ar@{<-}[r]&
 P_{2}
\ar@{<-}[r]&
 P_{3} \ar@{<-}[r]&\cdots,}\]
we can form the following Postnikov resolution 
of $X$ in $D(R)$
\[\xymatrix@!R@!C=8pt{
X&&&&&&&&&
\\
0\ar[rr]^{i_{0}}
\ar[u]
&&X_0\ar[ld]^{\scriptscriptstyle +1}_{q_{0}}\ar[rr]^{i_{1}}\ar@(u,l)@{-}@<-.25ex>[llu]^<(.15){\scriptscriptstyle +1}_<(.05){p_{0}}&&
X_1\ar[ld]^{\scriptscriptstyle +1}_{q_{1}}\ar[rr]^{i_{2}}\ar@(u,l)@{-}[llllu]^<(.15){\scriptscriptstyle +1}_<(.05){p_{1}}&&
X_2\ar[ld]^{\scriptscriptstyle +1}_{q_{2}}\ar[rr]^{i_{3}}\ar@(u,l)@{-}[llllllu]^<(.15){\scriptscriptstyle +1}_<(.05){p_{2}}&&
X_3\ar[ld]^{\scriptscriptstyle +1}_{q_{3}}\ar@{}[rd]|{\displaystyle\cdots}
\ar@{}[ru]|{\displaystyle\cdots}
\ar@(u,l)@{->}[llllllllu]^<(.15){\scriptscriptstyle +1}_<(.05){p_{3}}\\
&P_0\ar[lu]^{f_{0}}&&\Sigma P_1\ar[lu]^{f_{1}}&&\Sigma^{2} P_2\ar[lu]^{f_{2}}&&\Sigma^{3} P_3\ar[lu]^{f_{3}}&&}\]
where $\Sigma^{-1}X_{n}$ is the naive truncation,
\[\xymatrix@C=10pt{
\cdots\ar@{<-}[r]&0\ar@{<-}[r]&
P_{0}
\ar@{<-}[r]&
 P_{1}
\ar@{<-}[r]&\cdots\ar@{<-}[r]&P_{n}\ar@{<-}[r]&0\ar@{<-}[r]&\cdots,}\]
all morphisms $p_{n}$ are defined by $\varepsilon$ in the projective resolution, the morphisms $i_{n}$ are the obvious inclusions, $f_{n}$ is given up to sign by the differential $P_{n}\r P_{n-1}$, and $q_{n}$ is represented by the projection onto the quotient $X_{n}/X_{n-1}=\Sigma^{n+1}P_{n}$.

By Corollary \ref{phantom2.5},  $\delta$ in the homotopy colimit exact triangle of Proposition \ref{hocolimpostres}  is an $\infty$-$\C C$-phantom map. Let us check that it is not zero, or equivalently that $(p_{n})_{n>0}$ does not have a section.

Assume to the contrary that a section exists, i.e.~a morphism
\[s\colon X\st{\scriptscriptstyle -1}\To \coprod_{n>0}X_{n}\]
such that
\[(p_{n})_{n>0}s=\id{X}.\]
By functoriality, $H_{0}(s)$ cannot be a trivial morphism since $X$ is a non-trivial $R$-module.

The homology of $X_{n}$ is $H_{1}(X_{n})=X$, $H_{n+1}(X_{n})=Z_{n}$, the kernel of $P_{n}\r P_{n-1}$ (the $n$-cycles of the projective resolution of $X$), and zero elsewhere. Consider the exact triangle
\[\Sigma^{n}Z_{n}\To \Sigma^{-1}X_{n}\st{p_{n}}\To X\st{e_{n}}\To \Sigma^{n+1}Z_{n}\]
coming from the standard $t$-structure on $D(R)$. Here,
\[0\neq e_{n}\in\C T(X,\Sigma^{n+1}Z_{n})=\ext_{R}^{n+1}(X,Z_{n})\]
is non-trivial since $X$ has infinite projective dimension. In the following commutative diagram, where the bottom row is the coproduct of the previous exact triangles, the diagonal morphism is $H_{0}(s)$,
\[\xymatrix@u@C=30pt{
\amalg_{n> 0}\Sigma^{n}Z_{n}\ar[d]&\\
\amalg_{n> 0}\Sigma^{-1}X_{n}
\ar[d]^-{\amalg_{n> 0}p_{n}}
&
X\ar
[l]_{s}
\ar
[ld]^{H_{0}(s)}
\\
\amalg_{n> 0}X\ar[d]^-{\amalg_{n> 0}e_{n}}&\\
\amalg_{n> 0}\Sigma^{n+1}Z_{n}.&}\]
In particular, the composite $(\amalg_{n> 0}e_{n})H_{0}(s)$ should be zero. The target coproduct is also a product in $D(R)$ since its factors are $R$-modules concentrated in different degrees,
\[\coprod_{n> 0}\Sigma^{n+1}Z_{n}=\prod_{n> 0}\Sigma^{n+1}Z_{n}.\]
Hence,
\[
(\amalg_{n> 0}e_{n})H_{0}(s)\in \C T(X,\prod_{n> 0}\Sigma^{n+1}Z_{n})=\prod_{n>0}\ext_{R}^{n+1}(X,Z_{n}).
\]
We now compute this element.

Since $X$ is finitely generated, the $R$-module morphism
\[H_{0}(s)\colon X\st{\scriptscriptstyle -1}\To \coprod_{n>0}X\]
factors through a finite subcoproduct, i.e.~it is defined by a sequence of endomorphisms $(t_{n}\colon X\r X)_{n>0}$ which are almost all zero. Moreover, if
\[t_{n}^{*}\colon \ext_{R}^{n+1}(X,Z_{n})\To \ext_{R}^{n+1}(X,Z_{n})\]
denotes the homomorphism defined by functoriality of $\ext_{R}^{n+1}$ in the first variable, then
\[
(\amalg_{n> 0}e_{n})H_{0}(s)=(t_{n}^{*}(e_{i}))_{n>0}\in \prod_{n>0}\ext_{R}^{n+1}(X,Z_{n}).
\]
Since $H_{0}(s)\neq 0$, at least one of the $t_{n}$'s should be non-trivial, i.e.~there exists some $i>0$ such that $0\neq t_{i}\colon X\r X$. The $R$-module $X$ is simple, so $t_{i}$ is actually an automorphism, therefore $t_{i}^{*}$ is injective and $t_{i}^{*}(e_{i})\neq 0$ since $e_{i}\neq 0$. This contradicts the fact that $(\amalg_{n> 0}e_{n})H_{0}(s)$ should be zero.
\end{rem}

\subsection{Postnikov systems}

Postnikov systems were introduced in \cite{Ka88}. 
In this section, we make them the objects of a certain category where we define a natural homotopy relation. The main result of this section establishes an equivalence between the homotopy category of Postnikov resolutions, defined in Section \ref{hcpr}, and the homotopy category of Postnikov systems.

\begin{defn}\label{postnikov0}
A \emph{Postnikov system} $(X_{*},P_{*})$ is a countable sequence of exact triangles
\begin{equation*}
\xymatrix@!R@!C=8pt{0\ar[rr]^{i_{0}}&&X_0\ar[ld]^{\scriptscriptstyle +1}_{q_{0}}\ar[rr]^{i_{1}}&&X_1\ar[ld]^{\scriptscriptstyle +1}_{q_{1}}\ar[rr]^{i_{2}}&&X_2\ar[ld]^{\scriptscriptstyle +1}_{q_{2}}\ar[rr]^{i_{3}}&&X_3\ar[ld]^{\scriptscriptstyle +1}_{q_{3}}\ar@{}[rd]|{\displaystyle\cdots}\\
&P_0\ar[lu]^{f_{0}}&&P_1\ar[lu]^{f_{1}}&&P_2\ar[lu]^{f_{2}}&&P_3\ar[lu]^{f_{3}}&&}
\end{equation*}
such that $S_{\alpha}$ maps
$$\xymatrix@C=15pt{ P_{0}
\ar@{<-}[rr]^-{\scriptscriptstyle +1}_-{ q_{0}f_{1} }&&
 P_{1}
\ar@{<-}[rr]^-{\scriptscriptstyle +1}_-{ q_{1}f_{2} }&&
 P_{2}
\ar@{<-}[rr]^-{\scriptscriptstyle +1}_-{ q_{2}f_{3} }&&
 P_{3} \ar@{<-}[r]&\cdots}.
$$
to an exact sequence of projective objects in $\Mod{\alpha}{\C C}$. In particular, $X_n$ is $(n+1)$\nobreakdash-$\C C$-cellular. We will denote the structure morphisms by $f_{n}^{X}$, $i_{n}^{X}$ and $q_{n}^{X}$
when we need to distinguish between different Postnikov systems.

A \emph{morphism of Postnikov systems}
$$(\psi_{*},\varphi_{*}) \colon(X_{*},P_{*}) \To (Y_{*},Q_{*}) $$
is a sequence of exact triangle morphisms as follows
$$\xymatrix@!C=0pt@!R=0pt{0\ar[rrr]\ar[dd]&&&X_0\ar[lld]^{\scriptscriptstyle +1}\ar[rrr]\ar[dd]^{\psi_{0}}&&&X_1\ar[lld]^{\scriptscriptstyle +1}\ar[rrr]\ar[dd]^{\psi_{1}}&&&X_2\ar[lld]^{\scriptscriptstyle +1}\ar[rrr]\ar[dd]^{\psi_{2}}&&&X_3\ar[lld]^{\scriptscriptstyle +1}\ar[dd]^{\psi_{3}}\\
&P_0\ar[lu]\ar[dd]^<(.2){\varphi_{0}}&&&P_1\ar[lu]\ar[dd]^<(.2){\varphi_{1}}&&&P_2\ar[lu]\ar[dd]^<(.2){\varphi_{2}}&&&P_3\ar[lu]\ar[dd]^<(.2){\varphi_{3}}&&&\cdots\\
0\ar[rrr]|!{"2,2";"4,2"}{\hole} &&&Y_0\ar[lld]^{\scriptscriptstyle +1}\ar[rrr]|!{"2,5";"4,5"}{\hole}
&&&Y_1\ar[lld]^{\scriptscriptstyle +1}\ar[rrr]|!{"2,8";"4,8"}{\hole}&&&Y_2\ar[lld]^{\scriptscriptstyle +1}\ar[rrr]|!{"2,11";"4,11"}{\hole}&&&Y_3\ar[lld]^{\scriptscriptstyle +1}\\
&Q_0\ar[lu]&&&Q_1\ar[lu]&&&Q_2\ar[lu]&&&Q_3\ar[lu]&&&}$$
Composition of morphisms of Postnikov systems is defined in the obvious way.

A pair of morphisms
$$
(\psi_*,\varphi_*),(\bar\psi_*,\bar\varphi_*)\colon(X_{*},P_{*}) \To (Y_{*},Q_{*})
$$
are \emph{homotopic} $(\psi_*,\varphi_*)\simeq(\bar\psi_*,\bar\varphi_*)$ if the four equivalent conditions (1--4) in Definition \ref{morpostres} are satisfied. This natural equivalence relation is additive: two morphisms are homotopic iff their difference $(\psi_{*}-\bar \psi_{*},\varphi_{*}-\bar \varphi_{*})$ is nullhomotopic. We denote by $\mathbf{Post}_\infty$ the category of Postnikov systems and $\mathbf{Post}_\infty^{\simeq}$ its homotopy category. Both of them are additive.
\end{defn}

\begin{thm}\label{uneq}
The forgetful functor
$$
\Phi\colon \mathbf{Pres}_{\infty}^{\simeq}\To \mathbf{Post}_{\infty}^{\simeq},\quad
\Phi (X,X_{*},P_{*})= (X_{*},P_{*}),
$$
is an equivalence of categories surjective on objects.
\end{thm}

Notice that every equivalence of categories is essentially surjective on objects, but we will actually prove that this one is strictly surjective on objects. This theorem is proved after the following lemma.

\begin{lem}\label{iso}
In a Postnikov system $(X_{*},P_{*})$,  $S_{\alpha}(i_{1}q_0^{-1})$ induces a degree $-1$ isomorphism $H_{0}S_{\alpha}(P_{*})\cong \im S_{\alpha}(i_{1})$, and $S_{\alpha}(i_{n+1})$ induces a degree $0$ isomorphism $\im S_{\alpha}(i_{n})\cong \im S_{\alpha}(i_{n+1})$, $n>0$. In particular, $S_{\alpha}(X_{n})\cong H_{0}S_{\alpha}(P_{*})\oplus \ker S_{\alpha}(i_{n+1})$ for $n>0$.
\end{lem}

\begin{proof}
The functor $S_{\alpha}$ takes exact triangles to exact sequences, therefore
$$\xymatrix{S_\alpha(X_*)\ar[rr]^{S_\alpha(i_*)}_{\scriptscriptstyle(+1,0)}&&S_\alpha(X_*)\ar[ld]_{S_\alpha(q_*)}^{\scriptscriptstyle(0,+1)}\\
&S_\alpha(P_*)\ar[lu]^{S_\alpha(f_*)}_{\scriptscriptstyle(-1,0)}&}$$
is an exact couple. Here the first degree corresponds to the subscript $*$, and the second degree is the internal degree in the graded abelian category $\Mod{\alpha}{\C{C}}$.

Since $S_\alpha(P_*)$ is exact in degrees $\neq 0$, the derived exact couple is
$$\xymatrix{\im S_\alpha(i_*)\ar[rr]_{\scriptscriptstyle(+1,0)}&&\im S_\alpha(i_*)\ar[ld]^{\scriptscriptstyle(-1,+1)}\\
&H_{0}S_{\alpha}(P_{*})\ar[lu]_{\scriptscriptstyle(0,0)}^0&}$$
with $H_{0}S_{\alpha}(P_{*})$ concentrated in degree $0$. Indeed, since $\im S_\alpha(i_*)$ is concentrated in degrees $>0$, the map $H_0 S_\alpha(P_*)\r\im S_\alpha(i_*)$ is the trivial morphism, hence the lemma follows.
\end{proof}

\begin{proof}[Proof of Theorem \ref{uneq}]
Let $(X_{*},P_{*})$ be a Postnikov system. Take a homotopy colimit as in \eqref{hcet}. We claim that $(\Hocolim_{n}X_{n},X_{*},P_{*})$ is a Postnikov resolution. Actually, it is only left to check that $S_{\alpha}(\Hocolim_{n}X_{n})=H_{0}S_{\alpha}(P_{*})$. By Lemma \ref{iso}, $S_{\alpha}\eqref{hcet2}$ can be identified with the endomorphism of
\begin{equation}\label{split1.1}
\displaystyle\left(\bigoplus\limits_{n> 0}H_{0}S_{\alpha}(P_{*})\right)\oplus \left(\bigoplus\limits_{n> 0}\ker S_{\alpha}(i_{n+1})\right)
\end{equation}
which decomposes as the identity on the second factor and \eqref{split2} on the first factor, compare the proof of Proposition \ref{hocolimpostres}. There we also check that this endomorphism is a split monomorphism. Proceeding as in that proof, we deduce that the cokernel of the monomorphism $S_{\alpha}\eqref{hcet2}$ is $H_{0}S_{\alpha}(P_{*})$. This cokernel can also be identified with  $S_{\alpha}(\Hocolim_{n}X_{n})$ by Lemma \ref{util}. This proves the claim and that $\Phi$ is surjective on objects.

Let $(X,X_{*},P_{*})$ and $(Y,Y_{*},Q_{*})$  be Postnikov resolutions and $(\psi_{*},\varphi_{*})\colon(X_{*},P_{*})\r (Y_{*},Q_{*})$ a morphism of Postnikov systems. We choose exact triangles defining homotopy colimits as in Proposition~\ref{hocolimpostres}. The following commutative square of solid arrows can be extended to a triangle morphism
$$\xymatrix{\displaystyle\coprod\limits_{n> 0}X_{n}\ar[rrr]^<(.2){\eqref{hcet2}}
\ar[dd]_{(\psi_{n})_{n> 0}}
&&&\displaystyle\coprod\limits_{n> 0}X_{n}\ar[ld]^{\scriptscriptstyle +1}_{(p_{n})_{n> 0}}\ar[dd]^{(\psi_{n})_{n> 0}}\\
&&X\ar[llu]\ar@{-->}[dd]^<(.2){h}&\\
\displaystyle\coprod\limits_{n> 0}Y_{n}\ar[rrr]^<(.2){\eqref{hcet2}}
&&&\displaystyle\coprod\limits_{n> 0}Y_{n}\ar[ld]_{\scriptscriptstyle +1}^{(p_{n})_{n> 0}}\\
&&Y\ar[llu]&}$$
Hence,  $(h,\psi_{*},\varphi_{*})\colon(X,X_{*},P_{*})\r (Y,Y_{*},Q_{*})$ is a morphism of Postnikov resolutions. This shows that $\Phi$ is full.

The functor $\Phi$ is faithful since, by definition, two morphisms of Postnikov resolutions are homotopic if and only if the underlying morphisms of Postnikov systems are.
\end{proof}

\begin{rem}\label{wlinext2}
By Theorem \ref{uneq} and Remark \ref{wlinext},
$$\C I^\infty\into\C T\st{\Phi\Psi}\onto \mathbf{Post}_\infty^\simeq$$
is a weak linear extension, the $\C T$\nobreakdash-bimodule $\C I^\infty$ is actually a $\mathbf{Post}_\infty^\simeq$\nobreakdash-bimodule and the weak linear extension is classified up to equivalence by a class in cohomology of categories
$$\{\C T\}\in H^2(\mathbf{Post}_\infty^\simeq,\C I^\infty).$$
It is interesting to notice that $\mathbf{Post}_\infty^\simeq$ only depends of the full subcategory of $\C C$-cellular objects in $\C T$, and that there are no non-trivial $\infty$\nobreakdash-$\C C$\nobreakdash-phantom maps between two $\C C$-cellular objects. Hence, the previous linear extension is a way of breaking $\C T$ into an $\infty$\nobreakdash-$\C C$\nobreakdash-phantom part and an $\infty$\nobreakdash-$\C C$\nobreakdash-phantomless part.
\end{rem}

\subsection{Truncated Postnikov systems and obstructions}\label{sec-truncated-post}

Our notion of truncated Postnikov system enriches that considered in \cite{rmtc} in a way which is suitable to develop an obstruction theory. We also define homotopy categories of truncated Postnikov systems.

\begin{defn}\label{truncated}
An \emph{$n$\nobreakdash-truncated Postnikov system} $(X_{\leq n},P_*)$, $n\geq 0$, is a diagram in $\C{T}$
\begin{equation*}
\xymatrix@!R@!C=4.9pt{0\ar[rr]^{i_{0}} &&X_0\ar[ld]^{q_{0}}_{\scriptscriptstyle+1}\ar[rr]^{i_{1}} &&X_1\ar[ld]^{q_{1}}_{\scriptscriptstyle+1}&\ar@{}[d]|{\displaystyle\cdots\cdots}&X_{n-1}\ar[rr]^{i_{n}} &&X_n\ar[ld]^{q_{n}}_{\scriptscriptstyle+1}\\
&P_0\ar[lu]^{f_{0}}&&P_1\ar[lu]^{f_{1}}&&&&P_n\ar[lu]^{f_{n}}&&
P_{n+1}\ar[lu]_{f_{n+1}}&&P_{n+2}\ar[ll]_{d_{n+2}}^{\scriptscriptstyle+1}&\ar[l]\cdots}
\end{equation*}
where the first $n+1$ triangles are exact, the \emph{cocycle condition}
$$f_{n+1}d_{n+2}=0$$
is satisfied,
and
the restricted Yoneda functor $S_\alpha$ maps
$$\xymatrix@C=13pt{ P_{0}
\ar@{<-}[rr]_-{ q_{0}f_{1} }^{\scriptscriptstyle+1}&&
 P_{1}
\ar@{<-}[r]&
\cdots\cdots
\ar@{<-}[r]&
 P_{n}
\ar@{<-}[rr]_-{q_{n}f_{n+1} }^{\scriptscriptstyle+1}&&
 P_{n+1}\ar@{<-}[rr]_-{d_{n+2} }^{\scriptscriptstyle+1}&&
 P_{n+2} \ar@{<-}[r]&\cdots}$$
to an exact sequence of projective objects in $\Mod{\alpha}{\C C}$.   For $0\leq k\leq n$ we denote
$$d_{k+1}=q_{k}f_{k+1}.$$

Notice that $X_{k}$ is $(k+1)$\nobreakdash-$\C C$-cellular, $0\leq k\leq n$. We will denote the structure morphisms by $f_{k}^{X}$, $0\leq k\leq n+1$, $i_{k}^{X}$, $q_{k}^{X}$, $0\leq k\leq n$,   and $d_{k}^{X}$, $k>0$, if we need to distinguish between different $n$\nobreakdash-truncated Postnikov systems.

A \emph{morphism of $n$\nobreakdash-truncated Postnikov systems}
$$(\psi_{\leq n},\varphi_{*}) \colon(X_{\leq n},P_{*}) \To (Y_{\leq n},Q_{*}) $$
is a diagram
$$\xymatrix@C=5.9pt@R=7pt{0\ar[rrr]\ar[dd]&&&X_0\ar[lld]\ar[rrr]\ar[dd]^{\psi_{0}}&&&X_1
\ar@{}[rrrdd]|{\displaystyle\;\; \cdots\cdots}
\ar[lld]\ar[dd]^{\psi_{1}}&&&X_{n-1}\ar[rrr]\ar[dd]^{\psi_{n-1}}&&&X_{n}\ar[lld]\ar[dd]^{\psi_{n}}\\
&P_0\ar[lu]\ar[dd]^<(.2){\varphi_{0}}&&&P_1\ar[lu]\ar[dd]^<(.2){\varphi_{1}}&&&&&&P_{n}\ar[lu]\ar[dd]^<(.2){\varphi_{n}}&&&P_{n+1}\ar[lu]\ar[dd]^{\varphi_{n+1}}
&&P_{n+2}\ar[ll]\ar[dd]^{\varphi_{n+2}}&\ar[l]\cdots\\
0\ar[rrr]|!{"2,2";"4,2"}{\hole} &&&Y_0\ar[lld]\ar[rrr]|!{"2,5";"4,5"}{\hole}
&&&Y_1\ar[lld]&&&Y_{n-1}\ar[rrr]|!{"2,11";"4,11"}{\hole}&&&Y_n\ar[lld]\\
&Q_0\ar[lu]&&&Q_1\ar[lu]&&&&&&Q_n\ar[lu]&&&Q_{n+1}\ar[lu]&&Q_{n+2}\ar[ll]&\ar[l]\cdots}$$
where all squares commute. Composition is defined in the obvious way.

A pair of morphisms of $n$\nobreakdash-truncated Postnikov systems
$$
(\psi_{\leq n},\varphi_{*}) ,(\bar\psi_{\leq n},\bar\varphi_{*}) \colon(X_{\leq n},P_{*}) \To (Y_{\leq n},Q_{*})
$$
are \emph{homotopic}  $(\psi_{\leq n},\varphi_{*}) \simeq (\bar\psi_{\leq n},\bar\varphi_{*})$ if $\psi_k-\bar \psi_k$ factors through $f_{k+1}\colon Q_{k+1}\r Y_k$ for $0\leq k\leq n$. This condition can be characterized in different ways for $k<n$, see Definition \ref{morpostres} (1--4).

The homotopy natural equivalence relation is additive: two morphisms are homotopic iff their difference $(\psi_{\leq n}-\bar \psi_{\leq n},\varphi_{*}-\bar \varphi_{*})$ is nullhomotopic. We denote by  $\mathbf{Post}_{n}$ the category of $n$\nobreakdash-truncated Postnikov systems and $\mathbf{Post}_{n}^{\simeq}$ its homotopy category. Both categories are additive and the natural projection $\mathbf{Post}_{n}\r \mathbf{Post}_{n}^{\simeq}$ is an additive functor.

The \emph{homology functor}
\begin{align*}
\mathbf{Post}_{n}&\To \Mod{\alpha}{\C{C}},\\
(X_{\leq n},P_{*})&\;\mapsto\; H_{0}S_{\alpha}(P_{*})=\coker S_\alpha(d_{1}),
\end{align*}
factors through the homotopy category,
\begin{align*}
\mathbf{Post}_{n}^{\simeq}&\To \Mod{\alpha}{\C{C}}.
\end{align*}
This factorization is an equivalence for $n=0$.

The \emph{$(n-1)$\nobreakdash-truncation functor}, $n>0$,
$$t_{n-1}\colon \mathbf{Post}_{n}\To \mathbf{Post}_{n-1}$$
is the functor $t_{n-1}(X_{\leq n},P_{*})=(X_{\leq n-1},P_{*})$ defined by forgetting $X_{n}$, $f_{n+1}$, $i_{n}$, and $q_{n}$, but not $d_{n+1}=q_{n}f_{n+1}$. This functor is additive and compatible with the homotopy relation, hence it induces an additive functor
$$t_{n-1}\colon \mathbf{Post}_{n}^{\simeq}\To \mathbf{Post}_{n-1}^{\simeq}.$$
\end{defn}

\begin{lem}\label{iso4}
Given an $n$\nobreakdash-truncated Postnikov system $(X_{\leq n},P_{*})$:
\begin{itemize}
\item $S_{\alpha}(i_{1}q_0^{-1})$ induces a degree $-1$ isomorphism $H_{0}S_{\alpha}(P_{*})\cong \im S_{\alpha}(i_{1})$,
\item $S_{\alpha}(i_{k+1})$ induces a degree $0$ isomorphism $\im S_{\alpha}(i_{k})\cong \im S_{\alpha}(i_{k+1})$ for $0<k< n$,
\item the natural projection $S_\alpha(X_{k})\onto\coker S_{\alpha}(f_{k+1})$ restricts to a degree $0$ isomorphism $\im S_{\alpha}(i_{k})\cong\coker S_{\alpha}(f_{k+1})$, for $0<k\leq n$.
\end{itemize}
In particular, for $0<k\leq n$,
$S_{\alpha}(X_{k})\cong H_{0}S_{\alpha}(P_{*})\oplus \im S_{\alpha}(f_{k+1})$.
\end{lem}

\begin{proof}
Extend $f_{n+1}$ to an exact triangle,
$$\xymatrix@!R@!C=13pt{\ar@{}[d]|{\displaystyle\cdots}&X_{n-1}\ar[rr]^{i_{n}} &&X_n\ar[ld]^{q_{n}}_{\scriptscriptstyle+1}\ar@{-->}[rr]^{i_{n+1}}&&X_{n+1}\ar@{-->}[ld]^{q_{n+1}}_{\scriptscriptstyle+1}\\
&&P_n\ar[lu]^{f_{n}}&&
P_{n+1}.\ar[lu]^{f_{n+1}}&}$$
Consider the following exact couple in $\Mod{\alpha}{\C{C}}$,
 $$\xymatrix{S_\alpha(X_{*})\ar[rr]^{S_\alpha(i_{*})}_{\scriptscriptstyle(+1,0)}&&S_\alpha(X_{*})\ar[ld]^{S_\alpha(q_{*})}_{\scriptscriptstyle(0,+1)}\\
&S_\alpha(P_{*}).\ar[lu]^{S_\alpha(f_{*})}_{\scriptscriptstyle(-1,0)}&}$$
Here for $k>n+1$ we set $X_k=X_{n+1}$, $P_k=0$ and $i_k=\id{X_{n+1}}$.  The $E^2$\nobreakdash-term of the induced spectral sequence is
\begin{align*}
E^2_{0}&=\coker S_\alpha(d_1)=H_0S_\alpha(P_{*}),&
E^2_{n+1}&=\ker S_\alpha(d_{n+1}),
\end{align*}
and $E^2_{k}=0$ otherwise.
The derived exact couple is
$$\xymatrix{\im S_\alpha(i_{*})\ar[rr]_{\scriptscriptstyle(+1,0)}^{i'_*}&&\im S_\alpha(i_{*})\ar[ld]_{\scriptscriptstyle(-1,+1)}^{q'_*} \\
&E^2_{*}.\ar[lu]_{\scriptscriptstyle(0,0)}^{f'_*}&}$$
Since $\im S_\alpha(i_k)$ is concentrated in degrees $k>0$, $q'_*$ contains an isomorphism $\im S_\alpha(i_1)\cong E^2_0=H_{0}S_{\alpha}(P_{*})$ whose inverse is induced by $S_\alpha(i_{1}q_0^{-1})$. By the sparsity of $E^2_{*}$, $i'_*$ contains isomorphisms $\im S_\alpha(i_{k})\cong \im S_\alpha(i_{k+1})$ induced by $S_\alpha(i_{k+1})$ for $0<k\leq n$. This finishes the proof since $\ker S_\alpha(i_{k})=\im S_{\alpha}(f_{k})$ and hence $S_\alpha(i_k)$ induces an isomorphism $\coker S_{\alpha}(f_{k})\cong\im S_{\alpha}(i_{k})$, $0<k\leq n+1$.
\end{proof}

\begin{rem}
Let $(X_{\leq n},P_{*})$ be an $n$\nobreakdash-truncated Postnikov system. The following inclusion defined by Lemma \ref{iso4}, $0<k\leq n+1$, which splits for $0<k\leq n$, has degree $-1$,
$$H_0S_\alpha(P_{*})\mathop{\subset}_{\scriptscriptstyle-1} S_\alpha(X_{k}).$$
Notice that $X_{n+1}$ is not part of the $n$\nobreakdash-truncated Postnikov system, it is simply a mapping cone of $f_{n+1}$, see the proof of Lemma \ref{iso4}.
\end{rem}

\begin{defn}\label{kappa}
Let $(X_{\leq n},P_*)$ be an $n$\nobreakdash-truncated Postnikov system. Extend $f_{n+1}$ to an exact triangle
$$
\xymatrix@!R@!C=13pt{\ar@{}[d]|{\displaystyle\cdots}&X_{n-1}\ar[rr]^{i_{n}} &&X_n\ar[ld]^{q_{n}}_{\scriptscriptstyle+1}\ar@{-->}[rr]^{i_{n+1}}&&X_{n+1}\ar@{-->}[ld]^{q_{n+1}}_{\scriptscriptstyle+1}\\
&&P_n\ar[lu]^{f_{n}}&&
P_{n+1}\ar[lu]^{f_{n+1}}&&P_{n+2}\ar[ll]^{d_{n+2}}_{\scriptscriptstyle+1}\ar@{-->}[lu]_{\bar{f}_{n+2}}&&P_{n+3}\ar[ll]^{d_{n+3}}_{\scriptscriptstyle+1}&\ar[l]\cdots}
$$
By the cocycle condition $f_{n+1}d_{n+2}=0$ there exists $\bar{f}_{n+2}$ with $d_{n+2}=q_{n+1}\bar{f}_{n+2}$. This construction does not yield an $(n+1)$\nobreakdash-truncated Postnikov system  since $\bar{f}_{n+2}d_{n+3}\neq 0$ in general. However, $q_{n+1}\bar{f}_{n+2}d_{n+3}=d_{n+2}d_{n+3}=0$, and then $S_\alpha(\bar{f}_{n+2}d_{n+3})$ factors through $\ker S_{\alpha}(q_{n+1})\cong \coker S_{\alpha}(f_{n+1})\cong H_{0}S_{\alpha}(P_{*})$, see Lemma \ref{iso4}, as
$$
S_\alpha(\bar{f}_{n+2}d_{n+3})\colon S_\alpha(P_{n+3})\mathop{\To}^{\tilde\kappa}_{\scriptscriptstyle +2} H_0S_\alpha(P_{*})\mathop{\subset}_{\scriptscriptstyle-1} S_\alpha(X_{n+1}).
$$
The morphism $\tilde\kappa$ satisfies $\tilde\kappa S_{\alpha}(d_{n+4})=0$ since $\bar{f}_{n+2}d_{n+3}d_{n+4}=0$.

The \emph{obstruction of an $n$\nobreakdash-truncated Postnikov system} $(X_{\leq n},P_*)$ is the element
$$\kappa(X_{\leq n},P_*)\in \ext_{\alpha,\C{C}}^{n+3,-1-n}(H_0S_\alpha(P_*),H_0S_\alpha(P_*))$$
represented by a morphism
$\tilde\kappa$
constructed as in the previous paragraph.
\end{defn}

This obstruction class is natural in the following sense.

\begin{prop}\label{naturality}
Given a morphism of $n$\nobreakdash-truncated Postnikov systems, $$(\psi_{\leq n},\varphi_{*}) \colon(X_{\leq n},P_{*}) \To (Y_{\leq n},Q_{*}),$$
the following equation holds in $\ext_{\alpha,\C{C}}^{n+3,-1-n}(H_0S_\alpha(P_*),H_0S_\alpha(Q_*))$,
$$H_0S_\alpha(\varphi_*)\cdot\kappa(X_{\leq n},P_*)=\kappa(Y_{\leq n},Q_*)\cdot H_0S_\alpha(\varphi_*).$$
\end{prop}

\begin{proof}
Assume we have made choices for the definition of the two obstructions. Take $\psi_{n+1}$ extending $\psi_{n}$ and $\varphi_{n+1}$ to a triangle morphism,
$$\xymatrix@C=6pt@R=7pt{\ar@{}[dd]|{\displaystyle \cdots}&X_{n-1}\ar[rrr]\ar[dd]^<(.65){\psi_{n-1}}&&&X_{n}\ar[lld]\ar@{-->}[rrr]\ar[dd]^<(.65){\psi_{n}}&&&X_{n+1}\ar@{-->}[lld]
\ar@{-->}[dd]^<(.65){\psi_{n+1}}|!{"2,1";"2,10"}{\hole}\\
&&P_n\ar[lu]\ar[dd]^<(.2){\varphi_{n}}&&&P_{n+1}\ar[lu]\ar[dd]^<(.2){\varphi_{n+1}}&&&P_{n+2}\ar[lll]\ar@{-->}[lu]\ar[dd]^{\varphi_{n+2}}
&&P_{n+3}\ar[ll]\ar[dd]^{\varphi_{n+3}}&\ar[l]\cdots\\
&Y_{n-1}\ar[rrr]|!{"2,3";"4,3"}{\hole}
&&&Y_{n}\ar[lld]\ar@{-->}[rrr]|!{"2,6";"4,6"}{\hole}&&&Y_{n+1}\ar@{-->}[lld]\\
&&Q_n\ar[lu]&&&Q_{n+1}\ar[lu]&&&Q_{n+2}\ar@{-->}[lu]\ar[lll]&&Q_{n+3}\ar[ll]&\ar[l]\cdots}$$
The square containing $\psi_{n+1}$ and $\varphi_{n+2}$ need not commute. However,
\begin{align*}
q_{n+1}^{Y}\psi_{n+1}{\bar{f}}_{n+2}^{X}&=\varphi_{n+1}q_{n+1}^{X}{\bar{f}}_{n+2}^{X}=\varphi_{n+1}d_{n+2}^{X}\\
&=d_{n+2}^{Y}\varphi_{n+2}=q_{n+1}^{Y}{\bar{f}}_{n+2}^{Y}\varphi_{n+2},
\end{align*}
hence $S_\alpha(\psi_{n+1}{\bar{f}}_{n+2}^{X}-{\bar{f}}_{n+2}^{Y}\varphi_{n+2})$ factors as
$$
S_\alpha(\psi_{n+1}{\bar{f}}_{n+2}^{X}-{\bar{f}}_{n+2}^{Y}\varphi_{n+2})
\colon S_{\alpha}(P_{n+2})\mathop{\To}^\phi_{\scriptscriptstyle +1} H_0S_{\alpha}(Q_*)\mathop{\subset}_{\scriptscriptstyle -1} S_\alpha(Y_{n+1}).
$$
Moreover, since
\begin{align*}
(\psi_{n+1}{\bar{f}}_{n+2}^{X}-{\bar{f}}_{n+2}^{Y}\varphi_{n+2})d_{n+3}^{X}&=
\psi_{n+1}{\bar{f}}_{n+2}^{X}d_{n+3}^{X}-{\bar{f}}_{n+2}^{Y}d_{n+3}^{Y}\varphi_{n+3}
\end{align*}
we deduce that
$$\phi S_{\alpha}(d_{n+3}^{X})=H_0S_\alpha(\varphi_*)\tilde\kappa^{X}-\tilde\kappa^{Y}S_{\alpha}(\varphi_{n+3}),$$
hence we are done.
\end{proof}

A consequence of Proposition \ref{naturality} is that the construction of $\kappa(X_{\leq n},P_*)$ in Definition \ref{kappa} is independent of choices.

\begin{prop}\label{vanishing}
For an $n$\nobreakdash-truncated Postnikov system $(X_{\leq n},P_*)$,
$\kappa(X_{\leq n},P_*)=0$
if an only if there exists an $(n+1)$\nobreakdash-truncated Postnikov system $(X_{\leq n+1},P_*)$ whose $n$\nobreakdash-truncation is $(X_{\leq n},P_*)$.
\end{prop}

\begin{proof}
If $(X_{\leq n+1},P_*)$ exists we can take $\bar{f}_{n+2}=f_{n+2}$, hence the cocycle condition $f_{n+2}d_{n+3}=0$ implies that $\tilde\kappa=0$, so $\kappa(X_{\leq n},P_*)=0$.

Assume now that $\kappa(X_{\leq n},P_*)=0$. Suppose that we have made the necessary choices for the construction of $\tilde\kappa$. Since $\kappa(X_{\leq n},P_*)=0$ there exists a degree $+1$ morphism
$\zeta\colon S_\alpha(P_{n+2})\r H_0S_\alpha(P_*)$ such that $\tilde\kappa=\zeta d_{n+3}$.
The composite
$$
S_\alpha(P_{n+2})\mathop{\To}^{\zeta}_{\scriptscriptstyle+1} H_0S_\alpha(P_*)\mathop{\subset}_{\scriptscriptstyle -1} S_\alpha(X_{n+1})
$$
is the image by $S_\alpha$ of a unique $\phi\colon P_{n+2}\r X_{n+1}$. The equation $\tilde\kappa=\zeta d_{n+3}$ translates into $\phi d_{n+3}=\bar{f}_{n+2}d_{n+3}$. Hence $i_{n+1}$, $q_{n+1}$ and $f_{n+2}=\bar{f}_{n+2}-\phi$ extend $(X_{\leq n},P_*)$ to  an $(n+1)$\nobreakdash-truncated Postnikov system.
\end{proof}

\begin{defn}\label{theta}
Consider a couple of $n$\nobreakdash-truncated Postnikov systems $(X_{\leq n},P_*)$ and $(Y_{\leq n},Q_*)$, $n>0$, and a morphism between their $(n-1)$\nobreakdash-truncations,
$$(\psi_{\leq n-1},\varphi_*)\colon(X_{\leq n-1},P_*)\To(Y_{\leq n-1},Q_*).$$
Take $\psi_n'$ extending $\psi_{n-1}$ and $\varphi_{n}$ to an exact triangle morphism
$$\xymatrix@C=6pt@R=7pt{\cdots&X_{n-1}\ar[rrr]\ar[dd]_{\psi_{n-1}}&&&X_{n}\ar[lld]\ar@{-->}[dd]^{\psi_{n}'}\\
&&P_{n}\ar[lu]\ar[dd]^<(.2){\varphi_{n}}&&&P_{n+1}\ar[lu]\ar[dd]^{\varphi_{n+1}}
&&P_{n+2}\ar[ll]\ar[dd]^{\varphi_{n+2}}&\ar[l]\cdots\\
\cdots&Y_{n-1}\ar[rrr]|!{"2,3";"4,3"}{\hole}&&&Y_n\ar[lld]\\
&&Q_n\ar[lu]&&&Q_{n+1}\ar[lu]&&Q_{n+2}\ar[ll]&\ar[l]\cdots}$$
The square containing $\psi_{n}'$ and $\varphi_{n+1}$ need not commute, however
$$q_{n}^Y\psi_{n}'f_{n+1}^X=
\varphi_nq_n^Xf_{n+1}^X=\varphi_nd_{n+1}^X=d_{n+1}^Y\varphi_{n+1}=
q_{n}^Yf_{n+1}^Y\varphi_{n+1}.$$
Hence, by Lemma \ref{iso4}, $S_\alpha(\psi_{n}'f_{n+1}^X - f_{n+1}^Y\varphi_{n+1})$ factors through $\ker S_\alpha (q_n^Y)= \im S_\alpha (i_n^{Y})\cong H_0 S_\alpha(Q_*)$,
$$S_\alpha(\psi_{n}'f_{n+1}^X - f_{n+1}^Y\varphi_{n+1})\colon
S_\alpha(P_{n+1})\mathop{\To}^{\tilde\theta}_{\scriptscriptstyle+1}  H_0S_\alpha(Q_*)\mathop{\subset}_{\scriptscriptstyle-1} S_\alpha(Y_{n}).$$
The following equations show that $\tilde\theta S_{\alpha}(d_{n+2}^X)=0$,
\begin{align*}
(\psi_{n}'f_{n+1}^X - f_{n+1}^Y\varphi_{n+1})d_{n+2}^X&=
\psi_{n}'f_{n+1}^Xd_{n+2}^X - f_{n+1}^Y\varphi_{n+1}d_{n+2}^X\\
&=- f_{n+1}^Yd_{n+2}^Y\varphi_{n+2}=0.
\end{align*}
Here we use the cocycle condition for both $n$\nobreakdash-truncated Postnikov systems.

The \emph{obstruction of the morphism $(\psi_{\leq n-1},\varphi_*)$ relative to the initial $n$\nobreakdash-truncated Postnikov systems} is the element
$$\theta_{(X_{\leq n},P_*),(Y_{\leq n},Q_*)}
(\psi_{\leq n-1},\varphi_*)\in \ext_{\alpha,\C{C}}^{n+1,-n}(H_0S_\alpha(P_*),H_0S_\alpha(Q_*))$$
represented by a morphism $\tilde\theta$ constructed as above. We often omit the subscript of $\theta$ so as not to overload the notation. Notice that this obstruction is additive in the morphism,
$$\theta(\psi_{\leq n-1}+\bar\psi_{\leq n-1},\varphi_*+\bar\varphi_*)=\theta(\psi_{\leq n-1},\varphi_*)+\theta(\bar\psi_{\leq n-1},\bar\varphi_*).$$
\end{defn}

The following lemma allows to speak of the obstruction of a homotopy class.

\begin{lem}\label{loboto}
Given two $n$\nobreakdash-truncated Postnikov systems $(X_{\leq n},P_*)$ and $(Y_{\leq n},Q_*)$, $n>0$, and two  homotopic morphisms between their $(n-1)$\nobreakdash-truncations
$$(\psi_{\leq n-1},\varphi_{*})\simeq(\bar\psi_{\leq n-1},\bar\varphi_{*}) \colon(X_{\leq n-1},P_{*}) \To (Y_{\leq n-1},Q_{*}),$$
their obstructions coincide
$\theta(\psi_{\leq n-1},\varphi_*)
=
\theta(\bar\psi_{\leq n-1},\bar\varphi_*)$.
\end{lem}

\begin{proof}
It is enough to check that the obstruction of a nullhomotopic morphism $(\psi_{\leq n-1},\varphi_{*})\simeq 0$ vanishes.
Since it is nullhomotopic $0=i_{n}^{Y}\psi_{n-1}=\psi'_{n}i_{n}^{X}$, so we can factor $\psi_{n}'=\phi q_{n}^{X}$. Moreover, $\varphi_{*}$ is nullhomotopic, so $\varphi_{n+1}=h_{n+1}d_{n+1}^{X}+d_{n+2}^{Y}h_{n+2}$ for certain $h_{n+1}$ and $h_{n+2}$,
$$
\xymatrix@C=8pt@R=8pt{\cdots&X_{n-1}\ar[rrr]\ar[dd]_{\psi_{n-1}}&&&X_{n}\ar[lld]\ar@{-->}[dd]^{\psi_{n}'}\\
&&P_{n}\ar[lu]\ar[dd]_<(.2){\varphi_{n}}
\ar@{-->}@(dr,l)[rrrdd]_<(.6){h_{n+1}}
\ar@{-->}[drr]^\phi
&&&P_{n+1}\ar[lu]\ar[dd]_<(.2){\varphi_{n+1}}\ar@{-->}[rrdd]^<(.4){h_{n+2}}
&&P_{n+2}\ar[ll]\ar[dd]^{\varphi_{n+2}}&\ar[l]\cdots\\
\cdots&Y_{n-1}\ar[rrr]|!{"2,3";"4,3"}{\hole}&&&Y_n\ar[lld]\\
&&Q_n\ar[lu]&&&Q_{n+1}\ar[lu]&&Q_{n+2}\ar[ll]&\ar[l]\cdots}
$$
Using the direct sum decomposition in Lemma \ref{iso4} we obtain
$$\binom{\xi_{1}}{\xi_{2}}=S_{\alpha}(\phi-f_{n+1}^{Y}h_{n+1})\colon S_{\alpha}(P_{n})\To S_{\alpha}(Y_{n})\cong H_{0}S_{\alpha}(Q_{*})\oplus \im S_{\alpha}(f_{n+1}^{Y}).$$
Then $\tilde{\theta}=\xi_{1}S_{\alpha}(d_{n+1}^{X})$ since
\begin{align*}
(\phi-f_{n+1}^{Y}h_{n+1})d_{n+1}^X&
=\phi d_{n+1}^{X}-f_{n+1}^{Y}h_{n+1}d_{n+1}^{X}-f_{n+1}^{Y}d_{n+2}^{Y}h_{n+2}\\
&=\phi q_{n}^{X}f_{n+1}^{X}-f_{n+1}^{Y}(h_{n+1}d_{n+1}^{X}+d_{n+2}^{Y}h_{n+2})\\
&=\psi_{n}'f_{n+1}^X - f_{n+1}^Y\varphi_{n+1}.
\end{align*}
Here we use the cocycle condition $f_{n+1}^{Y}d_{n+2}^{Y}=0$. Therefore $\theta(\psi_{\leq n-1},\varphi_*)=0$.
\end{proof}

As a consequence of Lemma \ref{loboto}, the obstruction of a morphism does not depend on choices.

\begin{prop}\label{dolorprofundo}
With the notation in Definition \ref{theta},
$$\theta_{(X_{\leq n},P_*),(Y_{\leq n},Q_*)}
(\psi_{\leq n-1},\varphi_*)=0$$
if an only if there exists a morphism $\psi_{n}\colon X_n\r Y_n$ extending $(\psi_{\leq n-1},\varphi_*)$ to a morphism $(\psi_{\leq n},\varphi_*)\colon (X_{\leq n},P_*)\r(Y_{\leq n},Q_*)$ of $n$\nobreakdash-truncated Postnikov systems.
\end{prop}

\begin{proof}
If $(\psi_{\leq n},\varphi_*)$ extends the given morphism we can take $\psi_n'=\psi_n$, hence $\tilde\theta=\psi_{n}f_{n+1}^X - f_{n+1}^Y\varphi_{n+1}=0$ and the obstruction vanishes.

Conversely, if the obstruction vanishes take $\xi\colon S_\alpha(P_n)\r H_0S_\alpha(Q_*)$ with $\tilde\theta=\xi S_\alpha(d_{n+1}^X)$. The composite
$$
S_\alpha(P_{n+1})\st{\xi}\To  H_0S_\alpha(Q_*)\mathop{\subset}_{\scriptscriptstyle-1} S_\alpha(Y_{n})
$$
is the image by $S_{\alpha}$ of a unique $\phi\colon P_{n}\st{\scriptscriptstyle-1}\r Y_{n}$, which must satisfy the two following equations
\begin{align*}
q_n^Y\phi&=0 ,\qquad
\phi d_{n+1}^X=\psi_{n}'f_{n+1}^X - f_{n+1}^Y\varphi_{n+1}.
\end{align*}
We can take $\psi_{n}=\psi_{n}'-\phi q_{n}^X$, since
\begin{align*}
\psi_ni_n^X&=(\psi_{n}'-\phi q_{n}^X)i_n^X
=\psi_n'i_n^X=i_n^Y\psi_{n-1},\\
q_n^Y\psi_n&=q_n^Y(\psi_{n}'-\phi d_{n}^X)
=q_n^Y\psi_n'
=\varphi_nq_n^X,\\
\psi_nf_{n+1}^X&=(\psi_{n}'-\phi q_{n}^X)f_{n+1}^X
=\psi_{n}'f_{n+1}^X-\phi d_{n+1}^X\\
&=\psi_{n}'f_{n+1}^X-(\psi_{n}'f_{n+1}^X - f_{n+1}^Y\varphi_{n+1})
=f_{n+1}^Y\varphi_{n+1}.
\end{align*}
\end{proof}

The following result shows that the obstruction $\theta$ in Definition \ref{theta} is a derivation.

\begin{prop}\label{derivation}
Given three $n$\nobreakdash-truncated Postnikov systems
$(X_{\leq n},P_*)$, $(Y_{\leq n},Q_*)$,  and $(Z_{\leq n},R_*)$,
and two composable morphisms between their $(n-1)$\nobreakdash-truncations,
$$\xymatrix@C=50pt{(X_{\leq n-1},P_{*}) \ar[r]^{(\psi_{\leq n-1},\varphi_{*})} & (Y_{\leq n-1},Q_{*})
\ar[r]^{(\bar\psi_{\leq n-1},\bar\varphi_{*})} & (Z_{\leq n-1},R_{*}),}$$
the following equation holds in $\ext_{\alpha,\C{C}}^{n+1,-n}(H_0S_\alpha(P_*),H_0S_\alpha(R_*))$,
\begin{align*}
\theta((\bar\psi_{\leq n-1},\bar\varphi_*)(\psi_{\leq n-1},\varphi_*))
&=
\theta(\bar\psi_{\leq n-1},\bar\varphi_*)\cdot H_0S_\alpha(\varphi_*)
+
H_0S_\alpha(\bar\varphi_*)\cdot \theta(\psi_{\leq n-1},\varphi_*).
\end{align*}
\end{prop}

\begin{proof}
Assume we have chosen $\psi_{n}'$ and $\bar\psi_{n}'$ to define the morphisms $\tilde\theta^{\psi}$ and $\tilde{\theta}^{\bar\psi}$ representing the obstructions of the two given morphisms,
$$\xymatrix@C=6pt@R=7pt{
\cdots&X_{n-1}\ar[rrr]\ar[dd]_{\psi_{n-1}}&&&X_{n}\ar[lld]\ar@{-->}[dd]^{\psi_{n}'}\\
&&P_{n}\ar[lu]\ar[dd]^<(.2){\varphi_{n}}&&&P_{n+1}\ar[lu]\ar[dd]^{\varphi_{n+1}}
&&P_{n+2}\ar[ll]\ar[dd]^{\varphi_{n+2}}&\ar[l]\cdots\\
\cdots&Y_{n-1}\ar[rrr]|!{"2,3";"4,3"}{\hole}
\ar[dd]_{\bar\psi_{n-1}}&&&Y_n\ar[lld]\ar@{-->}[dd]^{\bar\psi_{n}'}\\
&&Q_n\ar[lu]\ar[dd]^<(.2){\bar\varphi_{n}}&&&Q_{n+1}\ar[lu]\ar[dd]^-{\bar\varphi_{n+1}}&&Q_{n+2}\ar[ll]\ar[dd]^-{\bar\varphi_{n+2}}&\ar[l]\cdots\\
\cdots&Z_{n-1}\ar[rrr]|!{"4,3";"6,3"}{\hole}
&&&Z_n\ar[lld]\\
&&R_n\ar[lu]&&&R_{n+1}\ar[lu]&&R_{n+2}\ar[ll]&\ar[l]\cdots
}$$
We can take $\bar\psi_{n+1}'\psi_{n+1}'$ to define the morphism $\tilde\theta^{\bar\psi\psi}$ representing the obstruction of the composition. With this choice, the equation already holds for representatives,
$$\tilde\theta^{\bar\psi\psi}=\tilde{\theta}^{\bar\psi}S_\alpha(\varphi_{n+1})+H_0 S_\alpha(\bar\varphi_*)\tilde\theta^{\psi},$$
since
\begin{align*}
&\hspace{-30pt}
(\bar\psi_{n}'f_{n+1}^Y - f_{n+1}^Z\bar\varphi_{n+1})\varphi_{n+1}
+\bar{\psi}_{n}'(\psi_{n}'f_{n+1}^X - f_{n+1}^Y\varphi_{n+1})
\\
={}&
\bar\psi_{n}'f_{n+1}^Y\varphi_{n+1} - f_{n+1}^Z\bar\varphi_{n+1}\varphi_{n+1}
+\bar{\psi}_{n}'\psi_{n}'f_{n+1}^X - \bar{\psi}_{n}'f_{n+1}^Y\varphi_{n+1}
\\
={}&(\bar{\psi}_{n}'\psi_{n}')f_{n+1}^X  - f_{n+1}^Z(\bar\varphi_{n+1}\varphi_{n+1}).
\end{align*}
\end{proof}

The following proposition shows that the obstruction of a morphism is non-trivial in general.

\begin{prop}\label{derivation2}
For any $n$\nobreakdash-truncated Postnikov system $(X_{\leq n},P_*)$
and any
$$\zeta\in \ext_{\alpha,\C{C}}^{n+1,-n}(H_0S_\alpha(P_*),H_0S_\alpha(P_*))$$
there exists another $n$\nobreakdash-truncated Postnikov system $(Y_{\leq n},Q_*)$
with the same  \mbox{$(n-1)$}-truncation $(X_{\leq n-1},P_*)=(Y_{\leq n-1},Q_*)$ such that
$$\theta_{(X_{\leq n},P_*),(Y_{\leq n},Q_*)}
(\id{(X_{\leq n-1},P_*)})=\zeta.$$
\end{prop}

\begin{proof}
We define the $n$\nobreakdash-truncated Postnikov system $(Y_{\leq n},Q_*)$ as follows,
$X_k=Y_k$, $f_k^X=f_k^Y$, $i_k^X=i_k^Y$, $q_k^X=q_k^Y$, $0\leq k\leq n$,
$P_k=Q_k$, $k\geq 0$, $d_{k}^X=d_{k}^Y$, $k\geq n+2$. It is only left to define
$f_{n+1}^Y$.

Choose a morphism $\tilde\zeta\colon S_\alpha(P_{n+1})\st{\scriptscriptstyle+1}\r H_0S_\alpha(P_*)$ representing $\zeta$. The composite
$$
S_\alpha(P_{n+1})\mathop{\To}^{\tilde\zeta}_{\scriptscriptstyle+1}  H_0S_\alpha(Q_*)\mathop{\subset}_{\scriptscriptstyle -1} S_\alpha(X_{n})
$$
is the image by $S_\alpha$ of a unique $\phi\colon P_{n+1}\r X_n$, which must satisfy $q_n^X\phi=0$ and
$\phi d_{n+2}=0$, since $\tilde\zeta S_\alpha(d_{n+2})=0$. The morphism $f^Y_{n+1}=f_{n+1}^X-\phi$ yields an $n$\nobreakdash-truncated Postnikov system $(Y_{\leq n},Q_*)$ since the cocycle condition holds,
\begin{align*}
f^Y_{n+1}d_{n+2}&=(f_{n+1}^X-\phi)d_{n+2}=f_{n+1}^Xd_{n+2}-\phi d_{n+2}=0-0=0.
\end{align*}
To show that its $(n-1)$\nobreakdash-truncation is $(X_{\leq n-1},P_*)$, it is enough to notice that $d_{n+1}^Y=q_n^Y f_{n+1}^Y=q_n^X (f_{n+1}^X - \phi )=q_n^X f_{n+1}^X -0 = d_{n+1}^X$. In order to compute the obstruction of $\id{(X_{\leq n-1},P_*)}$ we can take $\psi_n'=\id{X_n}$, so $\tilde\theta=\tilde\zeta$ and the obstruction is~$\zeta$.
\end{proof}

\begin{defn}\label{iota}
Given a pair of $n$\nobreakdash-truncated Postnikov systems $(X_{\leq n},P_*)$ and $(Y_{\leq n},Q_*)$, $n>0$, any degree~$0$ morphism $\tilde\zeta\colon S_\alpha(P_n)\rightarrow H_0S_\alpha(Q_*)$ with $\tilde\zeta S_\alpha(d_{n+1}^X)=0$ gives rise to a morphism
$$\bar\imath(\tilde\zeta)\colon (X_{\leq n},P_*)\To (Y_{\leq n},Q_*)$$
whose only non-trivial component is
$g_{\tilde\zeta}q_n^X\colon X_n\r Y_n$,
$$\xymatrix@C=8pt@R=8pt{\cdots&X_{n-1}\ar[rrr]\ar[dd]_{0}&&&X_{n}\ar[lld]\ar[dd]|{g_{\tilde\zeta} q_n^X}\\
&&P_{n}\ar[lu]\ar[dd]_<(.2){0}
\ar@{-->}[rrd]_{g_{\tilde\zeta}}
&&&P_{n+1}\ar[lu]\ar[dd]^{0}
&&P_{n+2}\ar[ll]\ar[dd]^{0}&\ar[l]\cdots\\
\cdots&Y_{n-1}\ar[rrr]|!{"2,3";"4,3"}{\hole}&&&Y_n\ar[lld]\\
&&Q_n\ar[lu]&&&Q_{n+1}\ar[lu]&&Q_{n+2}\ar[ll]&\ar[l]\cdots}$$
Here $g_{\tilde\zeta}\colon P_n\st{\scriptscriptstyle-1}\r Y_n$ is the morphism whose image by $S_\alpha$ is
$$S_\alpha(P_{n})\mathop{\To}^{\tilde\zeta} H_0S_\alpha(Q_*)\mathop{\subset}^{\scriptscriptstyle -1} S_\alpha(Y_{n}).$$
This construction defines a natural homomorphism
$$\bar\imath\colon \ker\hom^{0}_{\alpha,\C{C}}(S_\alpha(d_{n+1}^X), H_0S_\alpha(Q_*))\To
\mathbf{Post}_n((X_{\leq n},P_*), (Y_{\leq n},Q_*)).$$
\end{defn}

\begin{prop}\label{iprimera}
The natural homomorphism $\bar\imath$ factors as
$$\imath\colon\ext^{n,-n}_{\alpha,\C{C}}(H_0S_\alpha(P_*), H_0S_\alpha(Q_*))\To\mathbf{Post}_n((X_{\leq n},P_*), (Y_{\leq n},Q_*)).$$
\end{prop}

\begin{proof}
It is enough to notice that if $\tilde\zeta$ factors through $S_\alpha(d_n^X)$ then $\bar\imath(\tilde\zeta)=0$. This follows from $d_n^Xq_n^X=q_{n-1}^Xi_n^Xq_n^X=q_{n-1}^X0=0$.
\end{proof}

The kernel of $\imath=\imath_{(X_{\leq n},P_*), (Y_{\leq n},Q_*)}$ and of its composition with the natural projection onto the homotopy category $\mathbf{Post}_n^{\simeq}$ can be computed by means of spectral sequences associated to the Postnikov system $(X_{\leq n},P_*)$. We omit the details to avoid further technicalities, compare \cite[page 340 and VI.5.16]{B89}.

\begin{prop}\label{isegunda}
Given a morphism of $n$\nobreakdash-truncated Postnikov systems $$(\psi_{\leq n},\varphi_*)\colon (X_{\leq n},P_*)\To (Y_{\leq n},Q_*),$$ its $(n-1)$\nobreakdash-truncation $(\psi_{\leq n-1},\varphi_*)$ is nullhomotopic if and only if $(\psi_{\leq n},\varphi_*)$ is homotopic to a morphism in the image of $\imath$ in Proposition \ref{iprimera}.
\end{prop}

\begin{proof}
The truncation of a morphism in the image of $\imath$ is trivial. Conversely, if $(\psi_{\leq n-1},\varphi_*)$ is nullhomotopic then $0=i_{n}^{Y}\psi_{n-1}=\psi_{n}i_{n}^{X}$, so we can factor $\psi_{n}=\phi q_{n}^{X}$. Moreover, $\varphi_{*}$ is nullhomotopic, so $\varphi_{n+1}=h_{n+1}d_{n+1}^{X}+d_{n+2}^{Y}h_{n+2}$ for certain $h_{n+1}$ and $h_{n+2}$,
$$
\xymatrix@C=8pt@R=8pt{\cdots&X_{n-1}\ar[rrr]\ar[dd]_{\psi_{n-1}}&&&X_{n}\ar[lld]\ar[dd]^{\psi_{n}}\\
&&P_{n}\ar[lu]\ar[dd]_<(.2){\varphi_{n}}
\ar@{-->}@(dr,l)[rrrdd]_<(.6){h_{n+1}}
\ar@{-->}[drr]^\phi
&&&P_{n+1}\ar[lu]\ar[dd]_<(.2){\varphi_{n+1}}\ar@{-->}[rrdd]^<(.4){h_{n+2}}
&&P_{n+2}\ar[ll]\ar[dd]^{\varphi_{n+2}}&\ar[l]\cdots\\
\cdots&Y_{n-1}\ar[rrr]|!{"2,3";"4,3"}{\hole}&&&Y_n\ar[lld]\\
&&Q_n\ar[lu]&&&Q_{n+1}\ar[lu]&&Q_{n+2}\ar[ll]&\ar[l]\cdots}
$$
If we denote $\gamma=\phi-f_{n+1}^Yh_{n+1}$ we have that $\gamma d_{n+1}^X=0$, since
\begin{align*}
\phi d_{n+1}^X&=\phi q_n^Xf_{n+1}^X=\psi_{n}f_{n+1}^X=f_{n+1}^Y\varphi_{n+1}\\
&=f_{n+1}^Y(h_{n+1}d_{n+1}^X+d_{n+2}^Yh_{n+2})=f_{n+1}^Yh_{n+1}d_{n+1}^X.
\end{align*}
Here we use the cocycle condition $f_{n+1}^Yd_{n+2}^Y=0$.

Using the direct sum decomposition in Lemma \ref{iso4},
$$\binom{\xi_{1}}{\xi_{2}}=S_{\alpha}(\gamma)\colon S_{\alpha}(P_{n})\To S_{\alpha}(Y_{n})\cong H_{0}S_{\alpha}(Q_{*})\oplus \im S_{\alpha}(f_{n+1}^{Y}).$$
Since $\gamma d_{n+1}^X=0$ we have
 $\xi_kS_\alpha(d_{n+1}^X)=0$, $k=1,2$. Let us check that $\bar\imath(\xi_1)$ is homotopic to $(\psi_{\leq n},\varphi_*)$. Notice that, since $(\psi_{\leq n-1},\varphi_*)$ is nullhomotopic, we only have to check that $\psi_{n}-g_{\xi_{1}}q_{n}^{X}=(\phi-g_{\xi_{1}})q_{n}^{X}$ factors through $f_{n+1}^Y\colon Q_{n+1}\r Y_n$ where $g_{\xi_{1}}$ is the morphism whose image by $S_\alpha$ is $\xi_1$. This is obvious since by construction the image of $S_{\alpha}(\phi-g_{\xi_{1}})=S_{\alpha}(\phi)-\xi_{1}$ lies on $\im S_{\alpha}(f_{n+1}^{Y})$ in the previous direct sum decomposition.
\end{proof}

\subsection{The obstruction of a module}\label{kappamodule}

In this short section we analyze the most basic of the obstructions in Section \ref{sec-truncated-post}.

\begin{defn}
The \emph{obstruction of an $\alpha$\nobreakdash-continuous $\C{C}$\nobreakdash-module} $M$ is  the obstruction of a $0$\nobreakdash-truncated Postnikov system $(X_{ 0},P_*)$ with homology  $H_0S_\alpha(P_*)=M$,
$$\kappa(M)=\kappa(X_{ 0},P_*)\in \ext_{\alpha,\C{C}}^{3,-1}(M,M).$$
\end{defn}

The following characterization of this obstruction extends \cite[Theorem 3.7]{rmtc}.

\begin{prop}\label{pequeno}
Given an $\alpha$\nobreakdash-continuous $\C{C}$\nobreakdash-module $M$, $\kappa(M)=0$
if an only if $M$ is a retract of a restricted representable functor $S_{\alpha}(X)$.
\end{prop}

\begin{proof}
If $\kappa(M)=0$ we can extend $(X_{0},P_{*})$ to a $1$\nobreakdash-truncated Postnikov system $(X_{\leq 1},P_{*})$ by Proposition \ref{vanishing}, and Lemma \ref{iso4} shows that $M$ is a direct summand of $S_{\alpha}(X_{1})$. Conversely, we always have $\kappa (S_{\alpha}(X))=0$ from the existence of Postnikov resolutions, see Corollary \ref{postexis} and Proposition \ref{vanishing}. Moreover, if
$$
\xymatrix{M \ar@<.1cm>[r]^-i & S_{\alpha}(X) \ar@<.1cm>[l]^-r & ri=\id{M}}
$$
is a retraction, then, by Proposition \ref{naturality},
\begin{align*}
\kappa(M)&=\id{M}\cdot\kappa(M)=ri\cdot\kappa(M)
=r\cdot\kappa(S_{\alpha}(X))\cdot i=0.
\end{align*}
\end{proof}

\begin{cor}
If $S_{\alpha}$ is full, then an $\alpha$\nobreakdash-continuous $\C{C}$\nobreakdash-module $M$ is isomorphic to a restricted representable functor, $M\cong S_{\alpha}(X)$, if and only if $\kappa(M)=0$.
\end{cor}

\begin{proof}
If $M$ is a restricted representable functor, then $\kappa(M)=0$ by Proposition \ref{pequeno}. Conversely, if $\kappa(M)=0$, then we have a retraction
$$
\xymatrix{M \ar@<.1cm>[r]^-i & S_{\alpha}(X) \ar@<.1cm>[l]^-r & ri=\id{M}.}
$$
Since $S_{\alpha}$ is full $ir\colon S_{\alpha}(X)\r S_{\alpha}(X)$ is the image by $S_\alpha$ of some $f\colon X\r X$. One can check, as in the proof of Theorem \ref{uneq}, that $M$ is isomorphic to the image by $S_{\alpha}$ of
$$\Hocolim(X\st{f}\To X\st{f}\To X\st{f}\To \cdots).$$
\end{proof}

\subsection{Connection with the Adams spectral sequence}\label{SS}

Given a pair of objects $X$ and $Y$ in $\C T$, the \emph{Adams spectral sequence} is a conditionally convergent cohomological spectral sequence abutting to $\C T^{*}(X,Y)$ with $E_{2}$\nobreakdash-term $$E_{2}^{p,q}=\ext_{\alpha,\C C}^{p,q}(S_{\alpha}(X),S_{\alpha}(Y)),$$
cf.~\cite[Section~4]{Ch98}. It is defined by the exact couple
$$
\xymatrix{\C T(W_{*},Y)\ar@{<-}[rr]^{\C T(j_{*},Y)}_{\scriptscriptstyle(-1,0)}&&\C T(W_{*},Y)\ar@{<-}[ld]^{\C T(r_{*},Y)}_{\scriptscriptstyle(0,-1)}\\
&\C T(P_{*},Y)\ar@{<-}[lu]^{\C T(g_{*},Y)}_{\scriptscriptstyle(+1,0)}&}
$$
associated to an Adams resolution $(X,W_*,P_*)$ of $X$. Here we set $W_{-1}=X$. The induced decreasing filtration of $\C T^{*}(X,Y)$ is the filtration by powers of the ideal $\C I$ of $\C C$\nobreakdash-phantom maps. The spectral sequence strongly converges
if and only if $X$ is $n$-$\C C$-cellular for some $n$ \cite[Proposition 4.5]{Ch98}, or equivalently if $\C I^{n}(X,-)$ vanishes for some $n$. This is one of the reasons why the powers of the phantom ideal have attracted attention in the literature, as recalled in Remark \ref{classical}.

The following result relates the obstruction to the lifting of morphisms along $t_{n-1}$ in Definition \ref{theta} with the differentials of the Adams spectral sequence.

\begin{prop}
Let $(X,X_{*},P_{*})$ and $(Y,Y_{*},Q_{*})$ be Postnikov resolutions and let
$(\psi_{\leq n-1},\varphi_{*})\colon (X_{\leq n-1},P_{*})\r (Y_{\leq n-1},Q_{*})$ be a morphism of  $(n-1)$\nobreakdash-truncated Postnikov systems, $n>0$. The morphism of $\alpha$-continuous $\C C$-modules
$$H_{0}S_{\alpha}(\varphi_{*})\in \hom_{\alpha,\C C}(S_{\alpha}(X),S_{\alpha}(Y))=E_2^{0,0},$$
which lies in the second page of the previous Adams spectral sequence, lives actually in $E^{0,0}_{n+1}\subset E^{0,0}_{2}$ and
$d_{n+1}(H_{0}S_{\alpha}(\varphi_{*}))\in E_{n+1}^{n+1,-n}$ is represented by
$$\theta_{(X_{\leq n},P_{*}),(Y_{\leq n},Q_{*})}(\psi_{\leq n-1},\varphi_*)\in \ext_{\alpha,\C C}^{n+1,-n}(S_{\alpha}(X),S_{\alpha}(Y))=E_2^{n+1,-n}.$$
\end{prop}

\begin{proof}
Take an Adams resolution $(X,W_{*},P_{*})$ adapted to $(X,X_{*},P_{*})$ in the sense of Lemma \ref{postnikovaadams}. The morphisms $\phi_n$ and $\id{P_n}$, $n\geq 0$, define a morphism between the previous exact couple and the exact couple
$$
\xymatrix{\C T(X_{*},Y)\ar@{<-}[rr]^{\C T(i_{*},Y)}_{\scriptscriptstyle(-1,0)}&&\C T(X_{*},Y)\ar@{<-}[ld]^{\C T(q_{*},Y)}_{\scriptscriptstyle(0,-1)}\\
&\C T(P_{*},Y)\ar@{<-}[lu]^{\C T(f_{*},Y)}_{\scriptscriptstyle(+1,0)}&}
$$
associated to the Postnikov system $(X_{*},P_{*})$. This morphism is the identity on $E^1$\nobreakdash-terms, and hence on $E^k$\nobreakdash-terms for all $k\geq 1$. We can therefore compute the differentials of the Adams spectral sequence by using this second exact couple.

Let $n=1$. Since $H_{0}S_{\alpha}(\varphi_{*})$ is represented by $p_0^Y(q_0^Y)^{-1}\varphi_0q_0^X$, using the second exact couple it is clear that $d_2(H_{0}S_{\alpha}(\varphi_{*}))$ is represented by $p_1^Y\psi_1'f_2^X$.
$$
\xymatrix@!C=0pt@!R=0pt{
X &\\
&\ar[lu]0\ar[rrr]
^{i_0^X}\ar@{=}[ddd]&&&
\ar@(lu,r)[llllu]
^<(.6){p_0^X}
X_0\ar[ddd]^{\psi_{0}}\ar[lld]_{\scriptscriptstyle +1}^{q_0^X}
\ar[rrr]
^{i_1^X}&&&
\ar@(lu,r)@{-}[lllllllu]
_<(.1){p_1^X}
X_1\ar@{-->}[ddd]^{\psi'_{1}}\ar[lld]_{\scriptscriptstyle +1}^{q_1^X}
&&&
\cdots &\\
&&P_0\ar[lu]\ar[ddd]^<(0.1){\varphi_{0}}&&&P_1\ar[lu]\ar[ddd]^<(.1){\varphi_{1}}&&&P_2\ar[lu]_<(.5){f_2^X}\ar[ddd]^<(.1){\varphi_{2}}&&&&\\
Y&&&&\\
&\ar[lu]0\ar[rrr]|!{"2,3";"5,3"}{\hole}^{i_0^Y} &&&
\ar@(lu,r)[llllu]|!{"2,3";"5,3"}{\hole}|!{"2,2";"5,2"}{\hole}^<(.6){p_0^Y}
Y_0\ar[lld]_{\scriptscriptstyle +1}^{q_0^Y}
\ar[rrr]|!{"2,6";"5,6"}{\hole}^{i_1^Y}
&&&
\ar@(lu,r)@{-}[lllllllu]|!{"2,6";"5,6"}{\hole}|!{"2,5";"5,5"}{\hole}|!{"2,3";"5,3"}{\hole}|!{"2,2";"5,2"}{\hole}_<(.1){p_1^Y}
Y_1\ar[lld]_{\scriptscriptstyle +1}^{q_1^Y}
&&& \dots &\\
&&Q_0\ar[lu]^{f_0^Y}&&&Q_1\ar[lu]^{f_1^Y}&&&Q_2\ar[lu]^{f_2^Y}&&&}
$$
Moreover, $S_\alpha(p_1^Y\psi_1'f_2^X)=\tilde\theta$ by Lemma \ref{unsplit} since
$$p_1^Y(\psi_1'f_2^X-f_2^Y\varphi_2)=
p_1^Y\psi_1'f_2^X-p_1^Yf_2^Y\varphi_2=p_1^Y\psi_1'f_2^X-0\varphi_2=p_1^Y\psi_1'f_2^X.$$

If $n>1$ then $\psi_1'=\psi_1$ and $p_1^Y\psi_1'f_2^X=p_1^Yf_2^Y\varphi_2=0\varphi_2=0$ by Lemma \ref{unsplit}. In this way, by induction  $d_{k}(H_{0}S_{\alpha}(\varphi_{*}))=0$ for $1<k\leq n$ and
$d_{n+1}(H_{0}S_{\alpha}(\varphi_{*}))$ is represented by $p_n^Y\psi_n'f_{n+1}^X$. Moreover, $S_\alpha(p_n^Y\psi_n'f_{n+1}^X)=\tilde\theta$ by Lemma \ref{unsplit} since
\begin{align*}
p_n^Y(\psi_n'f_{n+1}^X-f_2^Y\varphi_{n+1})&=
p_{n}^Y\psi_{n}'f_{n+1}^X-p_{n}^Yf_{n+1}^Y\varphi_{n+1}\\
&=p_{n}^Y\psi_{n}'f_{n+1}^X-0\varphi_{n+1}=p_{n}^Y\psi_{n}'f_{n+1}^X.
\end{align*}
\end{proof}

\section{The first obstruction of an extension of representables}\label{sec-first-obs}\numberwithin{equation}{section}

A triangulated category is said to be \emph{algebraic} if it is a full triangulated subcategory of the homotopy category $K(\C A)$ of some additive category $\C A$, \mbox{cf.~\cite[\S7.5]{chicago}}. Recall the standing assumptions from Section \ref{ryf}: $\alpha$ is a regular cardinal, $\C T$ is a well generated triangulated category, and $\C{C}\subset\C{T}^{\alpha}$ is an essentially small full subcategory, closed under (de)suspensions and coproducts of less than $\alpha$ objects, that generates $\C{T}$. We do homological algebra in the abelian category $\Mod{\alpha}{\C C}$ of $\alpha$\nobreakdash-continuous (right) $\C C$\nobreakdash-modules, i.e.~functors $\C C^{\op}\r\Ab$ preserving products of less than $\alpha$ objects. This category is regarded as a graded abelian category with graded morphism objects \eqref{gradom} defined by using the suspension functor in $\C T$. Given $X$ in $\C T$, $S_{\alpha}(X)=\C{T}(-,X)_{|_{\C{C}}}$ is an example of $\alpha$\nobreakdash-continuous $\C C$\nobreakdash-module.

\begin{thm}\label{calculillo}
Let $\C T$ be an algebraic triangulated category. Suppose $F$ is an $\alpha$\nobreakdash-continuous $\C C$\nobreakdash-module fitting into a short exact sequence
\begin{equation*}\label{laext}
S_{\alpha}(Y)\st{a}\hookrightarrow F\st{b}\onto S_{\alpha}(X)
\end{equation*}
classified by
$$e_{F}\in \ext^{1,0}_{\C C}(S_{\alpha}(X),S_{\alpha}(Y)).$$
Then
the obstruction of $F$ is
$$\kappa(F)=a\cdot d_{2}(e_{F})\cdot b\in \ext^{3,-1}_{\C C}(F,F),$$
where $d_2$ is the second differential of the Adams spectral sequence in Section \ref{SS} abutting to $\C{T}^{*}(X,Y)$.
\end{thm}

This result is a paradigmatic example of a statement which makes sense for any triangulated category but which requires the use of models in its proof. The proof uses maps and homotopies in the category of complexes in $\C A$, actually homotopy classes of homotopies suffice, but we will not get into such technicalities. Nevertheless, this suggests that it should be enough to assume that $\C T$ is the homotopy category of a triangulated track category \cite{BM1,BM2}. This includes topological triangulated categories, i.e.~full triangulated subcategories of stable model categories. The proof in the non-additive setting is however more complicated. This is why we restrict to algebraic triangulated categories here. The proof is at the end of this section.

\begin{defn}
Let $C(\C A)$ be the category of chain complexes in an additive category $\C A$. Differentials of chain complexes in $\C A$ are denoted by $\partial$ and have degree $-1$. We add a superscript $\partial^A$ if we need to specify the complex $A$. A \emph{type $0$ standard exact triangle  starting at $A$} is a diagram in $C(\C A)$
$$\xymatrix@!=5pt{A\ar[rr]^{f}&&B\ar[ld]^{i}\\
&C_{f}\ar[lu]^{q}_{\scriptscriptstyle+1}&}$$
such that $C_{f}$ is the mapping cone of $f$,
$$(C_{f})_{n}=A_{n-1}\oplus B_{n},\qquad
\partial_{n}^{C_{f}}=\left(
\begin{array}{cc}
-\partial^{A}_{n-1}&0\\
f_{n-1}&\partial^{B}_{n}
\end{array}
\right),
$$
and $i$ and $q$ are given by
$$\xymatrix@C=40pt{
B_{n}\ar[r]^-{i_{n}=\binom{0}{1}}&(C_{f})_{n}=A_{n-1}\oplus B_{n}\ar[r]^-{q_{n}=(1,0)} &A_{n-1}.
}$$
The \emph{type $1$ standard exact triangle starting at $\Sigma A$ } and the \emph{type $2$ standard exact triangle starting at $A$} are
$$\xymatrix@!=5pt{\Sigma A\ar[rr]^{f}_{\scriptscriptstyle+1}&&B\ar[ld]^{i}\\
&C_{f}\ar[lu]^{q}&}\qquad\qquad\qquad
\xymatrix@!=5pt{A\ar[rr]^{f}&&B\ar[ld]^{i}_{\scriptscriptstyle+1}\\
&\Sigma^{-1} C_{f}\ar[lu]^{q}&}$$
respectively. Notice that, in all cases, $qi=0$ in $C(\C A)$.
\end{defn}

\begin{rem}
Recall that a chain map $\binom{g}{h}\colon D\r C_{f}$ is the same as a chain map $g\colon D\r \Sigma A$, given by morphisms $g_{n}\colon D_{n}\r A_{n-1}$ with $\partial^{A}_{n-1}g_{n}+g_{n-1}\partial_{n}^{D}=0$, together with a nullhomotopy $h\colon (\Sigma f)g\rr 0$, i.e.~a sequence of morphisms $h_{n}\colon D_{n}\r B_{n}$ with $f_{n-1}g_{n}+\partial_{n}^{B}h_{n}=h_{n-1}\partial_{n}^{D}$. Similarly, a chain map $(h,g)\colon C_{f}\r D$ is simply a map $g\colon B\r D$ together with a nullhomotopy $h\colon gf\rr 0$.
\end{rem}

Suppose for the rest of this section that $\C T$ is algebraic, and fix an embedding $\C T\subset K(\C A)$ which allows us to work with complexes in $\C A$. The following lemma shows how to compute $\kappa(F)$ by means of chain homotopies.

\begin{lem}\label{capa}
Let $F$ be an $\alpha$\nobreakdash-continuous $\C C$\nobreakdash-module and
$$\cdots\r R_{m}\st{d_m}\To R_{m-1}\r\cdots\r R_0$$
a sequence of morphisms in $C(\C A)$ whose homotopy classes lie in $\C T$ and map by $S_\alpha$ to a resolution of  $F$ in $\Mod{\alpha}{\C C}$. Let $h_{m}\colon d_{m}^{}d_{m+1}^{}\rr 0$ be nullhomotopies, $m=1,2$,
$$\xymatrix@l@!=40pt{
R_{0}&
R_{1}\ar[l]_{\scriptscriptstyle+1}^{d_{1}}&
R_{2}\ar[l]_{\scriptscriptstyle+1}^{d_{2}}
\ar@/^30pt/[ll]_{\scriptscriptstyle+2}="b"^{0}&
R_{3}\ar[l]_{\scriptscriptstyle+1}^{d_{3}}
\ar@/_30pt/[ll]^{\scriptscriptstyle+2}="a"_{0}
\ar@{=>}"1,3";"a"^{h_{2}}
\ar@{=>}"1,2";"b"^{h_{1}}
}$$
The degree $+2$ chain morphism $R_{3}\r R_{0}$ defined by the morphisms
$$h_{1,n-1}d_{3,n}^{}-d_{1,n-1}^{}h_{2,n}\colon R_{3,n}\To R_{0,n-2},\quad n\in\mathbb Z,$$
represents $\kappa(F)$.
\end{lem}

\begin{proof}

The nullhomotopies consist of morphisms $h_{m,n}\colon R_{m+1,n}\r R_{m-1,n-1}$ in $\C A$, $n\in\mathbb Z$, with
\begin{equation*}
d_{m,n-1}^{}d_{m+1,n}^{}= \partial_{n-1}^{R_{m-1}}h_{m,n}+h_{m,n-1} \partial_{n}^{R_{m+1}}.
\end{equation*}
Take $Z_{0}=\Sigma R_{0}$, $f_{1}^{}=d_{1}^{}$, and extend this morphism to a standard exact triangle of type $0$ starting at $R_{1}$. For the definition of $\kappa(F)$ we can take  $f_{2}'=\binom{d_{2}^{}}{h_{1}}\colon R_{2}\r Z_{1}=C_{d_{1}^{}}$ as in the following diagram
$$
\xymatrix@!R@!C=13pt{\Sigma R_{0}\ar@{-->}[rr]^{i_{ 1}}&&Z_{ 1}\ar@{-->}[ld]^{q_{ 1}}_{\scriptscriptstyle+1}\\
&
R_{ 1}\ar[lu]^{d_{ 1}}&&R_{ 2}\ar[ll]^{d_{ 2}}_{\scriptscriptstyle+1}\ar@{-->}[lu]_{f_{ 2}'}&&R_{ 3}\ar[ll]_{d_{ 3}}^{\scriptscriptstyle+1}&\ar[l]\cdots}
$$
Then $f_2'd_3$ is given by the following morphisms in $\C A$, $n\in\mathbb Z$:
\begin{align*}
\binom{d_{2,n-1}^{}}{h_{1,n-1}}d_{3,n}^{}&=
\binom{d_{2,n-1}^{}d_{3,n}^{}}{h_{1,n-1}d_{3,n}^{}}=\binom{\partial_{n-1}^{R_{1}}h_{2,n}+h_{2,n-1}\partial_{n}^{R_{3}}}{h_{1,n-1}d_{3,n}^{}}.
\end{align*}
We can deform this representative of the composite $f_2'd_3$ in $\C T$ by using the morphisms
$\binom{h_{2,n}}{0}\colon R_{3,n}\r R_{1,n-1}\oplus R_{0,n}$, $n\in\mathbb Z$, obtaining a chain morphism in the same homotopy class defined by the following morphisms in $\C A$, $n\in\mathbb Z$:
\begin{align*}
&\binom{d_{2,n-1}^{}}{h_{2,n-1}}d_{3,n}^{}-\binom{h_{2,n-1}}{0}\partial^{R_{3}}_{n}-
\partial_{n}^{C_{d_{1}^{}}}
\binom{h_{2,n}}{0}\\
=&
\binom{\partial_{n-1}^{R_{1}}h_{2,n}+h_{2,n-1}\partial_{n}^{R_{3}}}{h_{1,n-1}d_{3,n}^{}}
-\binom{h_{2,n-1}}{0}\partial^{R_{3}}_{n}-
\left(
\begin{smallmatrix}
-\partial_{n-1}^{R_{1}}&0\\
d_{1,n-1}^{}&\partial_{n}^{R_{0}}
\end{smallmatrix}
\right)
\binom{h_{2,n}}{0}\\
=&\binom{0}{h_{1,n-1}d_{3,n}^{}-d_{1,n-1}^{}h_{2,n}},
\end{align*}
hence we are done.
\end{proof}

\begin{lem}\label{dedos}
Let $(X,X_{*},P_{*})$ be a Postnikov resolution whose underlying Postnikov system consists of type $2$ standard triangles starting at $X_{m-1}$, $m\geq 0$, where $X_{-1}=0$,
\begin{equation*}
\xymatrix@!R@!C=8pt{0\ar[rr]^{i_{0}}&&X_0\ar[ld]^{\scriptscriptstyle +1}_{q_{0}}\ar[rr]^{i_{1}}&&X_1\ar[ld]^{\scriptscriptstyle +1}_{q_{1}}\ar[rr]^{i_{2}}&&X_2\ar[ld]^{\scriptscriptstyle +1}_{q_{2}}\ar[rr]^{i_{3}}&&X_3\ar[ld]^{\scriptscriptstyle +1}_{q_{3}}\ar@{}[rd]|{\displaystyle\cdots}\\
&P_0\ar[lu]^{f_{0}}&&P_1\ar[lu]^{f_{1}}&&P_2\ar[lu]^{f_{2}}&&P_3\ar[lu]^{f_{3}}&&}
\end{equation*}
Given $e\in\ext_{\alpha,\C C}^{s,t}(S_\alpha(X),S_\alpha(Y))=E^{s,t}_2$ represented by a chain map $\tilde e\colon P_s\r Y$ of degree $s+t$, if $l\colon \tilde ed_{s+1}^X\rr 0$ is a nullhomotopy, then the image of $e$ by the Adams spectral sequence's second differential $d_2(e)\in \ext_{\alpha,\C C}^{s+2,t-1}(S_\alpha(X),S_\alpha(Y))=E^{s+2,t-1}_2$ is represented by the chain map $P_{s+2}\r Y$ of degree $s+t+1$ defined by the following morphisms, $n\in\mathbb Z$:
$$-l_{n-1}d_{s+2,n}^X\colon P_{s+2,n}\To Y_{n-s-t-1}.$$
\end{lem}

\begin{proof}
The nullhomotopy $l$ is given by morphisms $l_{n}\colon P_{s+1,n}\r Y_{n-s-t}$ in $\C A$ satisfying
$\tilde e_{n-1}d_{s+1,n}^{X}=\partial_{n-s-t}^{Y}l_{n}+l_{n-1}\partial_{n}^{P_{s+1}}$, $n\in\mathbb Z$. This nullhomotopy and $\tilde eq_s^X$ define a degree $s+t+1$ morphism
$(l,\tilde eq_s^X)\colon C_{f_{s+1}^X}\r Y$. Since the exact triangles of $(X_*,P_*)$ are standard of type $2$, $P_{s+1}$ is the desuspension of the mapping cone of $i_{s+1}^X$, and $C_{f_{s+1}^X}$ is given by
$$
(C_{f_{s+1}^X})_n=X_{s,n-1}\oplus X_{s+1,n}\oplus X_{s,n},\qquad
\partial^{C_{f_{s+1}^X}}_n=\left(
\begin{array}{cc|c}
-\partial^{X_s}_{n-1}&0&0\\
i_{s+1,n}&\partial^{X_{s+1}}_n&0\\\hline&&\\[-10pt]
 1&0&\partial^{X_s}_n
\end{array}
\right).
$$
The inclusion of the middle direct summands
$$\left(
\begin{array}{c}
0\\1\\0
\end{array}
\right)\colon X_{s+1,n}\To (C_{f_{s+1}^{X}})_{n}= X_{s,n-1}\oplus X_{s+1,n}\oplus X_{s,n}$$
yield a homotopy equivalence $X_{s+1}\st{\sim}\r C_{f_{s+1}^X}$ such that the  triangle
$$\xymatrix{X_s\ar[r]^{i_{s+1}}\ar[rd]_-{\begin{array}{c}
\scriptstyle\text{inclusion into}\\[-4pt]
\scriptstyle\text{mapping cone}
                                        \end{array}}
&X_{s+1}\ar[d]^\sim\\
&C_{f_{s+1}^X}}$$
anticommutes up to the homotopy given by the morphisms
$$\left(
\begin{array}{c}
1\\0\\0
\end{array}
\right)\colon X_{s,n}\To (C_{f_{s+1}^{X}})_{n+1}= X_{s,n}\oplus X_{s+1,n+1}\oplus X_{s,n+1}.$$
Hence $-d_2(e)$ is represented by
$$\xymatrix{P_{s+2}\ar[r]^{f_{s+2}^{X}}&X_{s+1}\ar[r]^{\sim}& C_{f_{s+1}^{X}}\ar[rr]^{(l,\tilde eq_{s}^{X})}_{\scriptscriptstyle p+s+1}&&Y.}$$
This composite is defined by the morphisms $l_{n-1}d_{s+2,n}^X$, $n\in\mathbb Z$, hence we are done.
\end{proof}

\begin{rem}
It is always possible to represent a Postnikov system by type $2$ standard triangles as in the statement of Lemma \ref{dedos}. Moreover, in the conditions of that statement, if $\tilde e\colon P_s\r Y$ represents an element in $\ext_{\alpha,\C C}^{s,t}(S_\alpha(X),S_\alpha(Y))$ there must exist a nullhomotopy $l\colon \tilde ed_{s+1}^X\rr 0$ since $\tilde ed_{s+1}^X= 0$ in $\C T$.
\end{rem}

\begin{proof}[Proof of Theorem \ref{calculillo}]
Take a Postnikov resolution $(X,X_{*},P_{*})$ whose underlying Postnikov system $(X_{*},P_{*})$
consists of type $2$ standard triangles starting at $X_{m-1}$, $m\geq 0$, and an Adams resolution $(Y,W_{*},Q_{*})$ consisting of type $2$ standard triangles starting at $Y$ and $W_{m}$, $m\geq 0$.
By elementary homological algebra, there are degree $+1$ chain maps $s_{m}\colon P_{m}\r Q_{m-1}$, $m>0$, such that the morphisms
$$d_{m}^{Z}= \left(
\begin{array}{cc}
d_{m}^{Y}&s_{m}\\
0&d_{m}^{X}
\end{array}
\right)\colon  Q_{m}\oplus P_{m}\To Q_{m-1}\oplus P_{m-1},$$
map to a projective resolution of $F$ in $\Mod{\alpha}{\C C}$.
The element $e_{F}$ is represented by $-S_{\alpha}(s_{1})$.
Since these matrices define differentials in $\C T$, $d_{m}^{X}d_{m+1}^{X}=0$, $d_{m}^{Y}d_{m+1}^{Y}=0$, and
$d_{m}^{Y}s_{m+1}+s_{m}d_{m+1}^{X}=0$. The first two equations also hold at the level of chain maps by the properties of standard triangles.
For the third equation, we choose an arbitrary nullhomotopy $k_{m}\colon d_{m}^{Y}s_{m+1}+s_{m}d_{m+1}^{X}\rr 0$, defined by morphisms $k_{m,n}\colon P_{m+1,n}\r Q_{m-1,n-1}$, $n\in\mathbb Z$, satisfying
\begin{align}\label{eq2}
d_{m,n-1}^{Y}s_{m+1,n}+s_{m,n-1}d_{m+1,n}^{X}&=\partial^{Q_{m-1}}_{n-1}k_{m,n}+k_{m,n-1}\partial^{P_{m+1}}_{n}.
\end{align}
We take $h_{m}\colon d_{m}^{Z}d_{m+1}^{Z}\rr 0$, $m=1,2$, to be defined by
$$h_{m,n}=\left(
\begin{array}{cc}
0&k_{m,n}\\
0&0
\end{array}
\right),\quad n\in\mathbb Z.$$
By Lemma \ref{capa} the following morphisms define a chain morphism representing $\kappa(F)$,
$$\left(\begin{array}{cc}
0&k_{1,n-1}d_{3,n}^{X}-
d_{1,n-1}^{Y} k_{2,n}\\0&0
\end{array}\right)\colon Q_{3,n} \oplus P_{3,n}\To Q_{0,n-2} \oplus P_{0,n-2}.
$$
This shows that if $x\in \ext^{3,-1}_{\alpha,\C C}(S_{\alpha}(X),S_{\alpha}(Y))$ is the element represented by the chain map defined by the following morphisms, $n\in\mathbb Z$,
$$g_{0,n-2}(k_{1,n-1}d_{3,n}^{X}-
d_{1,n-1}^{Y} k_{2,n})
=g_{0,n-2}k_{1,n-1}d_{3,n}^{X}-
0 k_{2,n}=g_{0,n-2}k_{1,n-1}d_{3,n}^{X},$$
then
$$\kappa(F)=a\cdot x\cdot b.$$
We now identify this $x$ with $d_{2}(e_{F})$.

Take
$$\tilde e\colon P_{1}\mathop{\To}^{s_{1}}_{\scriptscriptstyle+1}Q_{0}\st{g_{0}}\To Y,\qquad
l_{n}=g_{0,n-1}k_{1,n},\quad n\in\mathbb Z.
$$
We must check that $l$ defined in this way is a homotopy. Indeed, since $g_{0}$ is a chain map
\begin{align*}
\partial_{n-1}^{Y}l_{n}+l_{n-1}\partial_{n}^{P_{2}}&=
\partial_{n-1}^{Y}g_{0,n-1}k_{1,n}+g_{0,n-2}k_{1,n-1}\partial_{n}^{P_{2}}\\
&=
g_{0,n-2}\partial_{n-1}^{Q_{0}}k_{1,n}+g_{0,n-2}k_{1,n-1}\partial_{n}^{P_{2}}\\
&=
g_{0,n-2}(\partial_{n-1}^{Q_{0}}k_{1,n}+k_{1,n-1}\partial_{n}^{P_{2}})\\
&=
g_{0,n-2}(d_{1,n-1}^{Y}s_{2,n}+s_{1,n-1}d_{2,n}^{X})\text{ by \eqref{eq2}}\\
&=0s_{2,n}+g_{0,n-2}s_{1,n-1}d_{2,n}^{X}\\
&=\tilde e_{n-1}d_{2,n}^{X}.
\end{align*}
Hence $d_{2}(e_{F})=x$ by Lemma \ref{dedos}.
\end{proof}

\section{A characterization of $\alpha$-compact objects}\label{sec-alphacpt}

The following theorem is used in Sections \ref{sec-for-rings} and \ref{sec-aleph_1} to prove that some categories satisfy \aro{\aleph_1} under the continuum hypothesis.

\begin{thm}\label{Improved_ThmC}
Let $\alpha$ be a regular cardinal. Suppose that $\beta$ is a cardinal satisfying one of the following hypotheses:
\begin{enumerate}
\item $\beta=(\gamma^{<\delta})^{+}$ for some $\gamma \geq \card \C T^{\alpha}$ and some regular cardinal $\delta\geq\alpha$.
\item $\beta> \card \C T^{\alpha}$ is inaccessible.
\end{enumerate}
Then $\C T^{\beta}$ is the full subcategory of objects $Z$ such that $\card \C T(Y,Z)<\beta$ for any $Y$ in~$\C T^{\alpha}$.
\end{thm}

One can easily produce  cardinals like $\beta$ in (1), however the existence of cardinals as in (2) depends on large cardinal principles. Theorem \ref{Improved_ThmC} recovers Krause's \cite[Theorem
C]{Kr02} by taking $\beta=(\gamma^{<\delta})^{+}$ as in (1) for $\gamma=\card \C T^{\alpha}$ and $\delta=\alpha^{+}$, i.e.~$\beta=((\card \C T^{\alpha})^{\alpha})^{+}$. Notice that the smallest cardinal $\beta$ we can take as in (1) is $\beta=((\card \C T^{\alpha})^{<\alpha})^{+}$, which is smaller than Krause's choice.

\begin{lem}\label{lem_Thm_C}
Let $\beta$ be as in the statement of Theorem \ref{Improved_ThmC}.
Given a set $I$ of $\card I<\beta$ and objects $X_{i},Y$ in $\C T^\alpha$, $i\in I$, then $\card\C T(Y,\amalg_{i\in
I}X_i)<\beta$.
\end{lem}

\begin{proof}
Notice that $\beta>\card \C T^{\alpha}$ in both cases.
Since $Y$ is $\alpha$\nobreakdash-small,
$$\C T(Y,\coprod_{i\in
I}X_i)=\colim_{\begin{array}{c}\\[-15pt]
\scriptstyle J\subset I\\[-4pt]
\scriptstyle \card J<\alpha
\end{array}}
\C T^{\alpha}(Y,\coprod_{i\in
I}X_i).$$
The cardinal of this set is bounded above by $(\card I)^{<\alpha}\cdot\card \C T^{\alpha}$, so it is enough to check that $(\card I)^{<\alpha}<\beta$. If $\beta$ satisfies condition (2), the result follows from the strong limit property. Otherwise,
$(\card I)^{<\alpha}\leq (\gamma^{<\delta})^{<\alpha}=\gamma^{<\delta}<\beta$ by \cite[Lemma 2.10]{AR94}. \end{proof}

The following lemma is obvious.

\begin{lem}\label{lem_Thm_C0}
Let $\mathcal S$ be a class of objects in $\C T$ closed under (de)suspensions, $\Sigma \mathcal S=\mathcal S$, and $\beta$ an infinite cardinal. The full subcategory of objects $Z$ such that $\C T(Y,Z)<\beta$ for all $Y\in \mathcal S$ is triangulated.
\end{lem}

We are now ready to prove Theorem \ref{Improved_ThmC}.

\begin{proof}[Proof of Theorem \ref{Improved_ThmC}]
Denote by $\C S$ the full subcategory of $\C T$ spanned by the objects $Z$ such that $\C T(Y,Z)<\beta$ for any $Y$ in $\C T^{\alpha}$. This subcategory is triangulated by Lemma \ref{lem_Thm_C0}. We claim that $Z$ is in $\C S$ if and only if there is an morphism $g_{0}\colon P_{0}=\amalg_{i\in I} X_{i}\r Z$ with $X_{i}$ in $\C T^{\alpha}$ and $\card I<\beta$, such that $S_{\alpha}(g_{0})$ is an epimorphism. If such a morphism exists, then for any $Y$ in $\C T^{\alpha}$, $\card \C T(Y,Z)\leq\card \C T(Y,\amalg_{i\in I}X_{i})<\beta$ by Lemma \ref{lem_Thm_C}, so $Z$ is in $\C S$. Conversely, if $Z$ is in $\C S$, consider the evaluation morphism $$g \colon P =\hspace{-10pt}\coprod_{\begin{array}{c}\\[-15pt]
\scriptstyle Y\in\mathcal S\\[-4pt]
\scriptstyle \C T(Y,Z)
\end{array}}\hspace{-10pt}Y\To Z,$$
where $\mathcal  S$ is a set of representatives of isomorphism classes of objects in $\C T^\alpha$.
The coproduct is indexed by a set of cardinality $\leq \sum_{Y\in \mathcal S}\card \C T(Y,Z)<\beta$, since $\card \C T^{\alpha}<\beta$ and $\beta$ is regular, and $S_{\alpha}(g )$ is clearly an epimorphism.

We now prove that $\C S=\C T^{\alpha}$. Given an object $Z$ in $\C S$, we can construct, as in Remark \ref{rem-ar}, an Adams resolution $(Z,W_{*},P_{*})$ where each $P_{n}$ is a direct sum of $<\beta$ objects in $\C T^{\alpha}\subset \C T^{\beta}$, so each $P_{n}$ is in $\C T^{\beta}$. Let $(Z,Z_{*},P_{*})$ be an associated Postnikov resolution, as in Lemma \ref{postnikovaadams}. It can be seen by induction that each $Z_{n}$ is in $\C T^{\beta}$ since we have exact triangles $P_n\to Z_{n-1}\to Z_n\to\Sigma P_n$. Hence $Z=\Hocolim_{n}Z_{n}$ is also in $\C T^{\beta}$ because, since $\beta>\aleph_{0}$, $\C T^{\beta}$ has countable coproducts. This proves $\C S\subset \C T^{\beta}$.

Since $\C T^{\beta}$ is the smallest $\beta$\nobreakdash-localizing subcategory containing a set of $\alpha$\nobreakdash-compact generators, in order to show $\C T^{\beta}\subset \C S$ it is enough to see that $\C S$ is $\beta$\nobreakdash-localizing, i.e.~closed under coproduct of $<\beta$ objects. Let $\{Z_{j}\}_{j\in J}$ be a set of objects in $\C S$ with $\card J<\beta$. By the first part of the proof there are morphisms $g_{j}\colon P_{j}\r Z_{j}$ such that $P_{j}$ is a coproduct of $<\beta$ objects in $\C T^{\alpha}$ and $S_{\alpha}(g_{j})$ is an epimorphism for all $i\in J$. Hence, the source of $\amalg_{j\in J} g_{j}\colon \amalg_{j\in J}P_{j}\r \amalg_{j\in J}Z_{j}$ is also a coproduct of $<\beta$ objects in $\C T^{\alpha}$. Moreover, $S_{\alpha}(g_{j})$ is an epimorphism by \cite[Theorem A]{Kr01}, therefore $\amalg_{j\in J}Z_{j}$ is in $\C S$.
\end{proof}

\begin{prop}\label{bound_to_card_T^beta}
Let $\C T$ be an $\alpha$\nobreakdash-compactly generated triangulated category and $\kappa>\alpha$ be regular cardinal such that
 either $\kappa$ is strongly inaccessible or $2^\lambda=\lambda^+$ for every
$\lambda<\kappa$. If $\card \C T^\alpha\leq \kappa$,
 then $\card \C T^\kappa=\kappa$.
\end{prop}

\begin{proof}
Our assumptions on $\kappa$ and \cite[Theorem 5.20]{Je03} show that $\kappa^{<\kappa}=\kappa$. Taking $\beta=\kappa^{+}$ in Theorem \ref{Improved_ThmC}, we deduce that the size of morphism sets in $\C T^{\kappa^{+}}$ is $<\kappa^{+}$, i.e.~$\leq \kappa$. Hence the same is true for $\C T^{\kappa}\subset\C T^{\kappa^{+}}$.

By the proof of \cite[Lemma 3.2.4 and Proposition
3.2.5]{triang}, the set of objects $S_{\kappa}$ of $\C T^\kappa$ can be constructed as a continuous increasing union $S_{\kappa}=\bigcup_{\mu<\kappa}S_{\mu}$ starting with the set $S_{0}$ of objects of $\C T^\alpha$. The set $S_{\mu+1}$ is defined from $S_{\mu}$ by adding coproducts of $<\kappa$ objects in $S_{\mu}$ and  mapping cones of all possible morphisms between such coproducts. Assume that $\card S_{\mu}\leq \kappa$.  Adding coproducts of $<\kappa$ objects increases the cardinal at most to $(\card S_{\mu})^{<\kappa}\leq \kappa^{<\kappa}=\kappa$. Adding mapping cones neither increases the cardinal of $S_{\mu}$ since the size of morphism sets in $\C T^{\kappa}$ is~$\leq \kappa$.
\end{proof}

\begin{cor}\label{saltito}
Let $\C T$ be an $\aleph_0$\nobreakdash-compactly generated triangulated category. Assuming the continuum hypothesis, if $\card\C{T}^{\aleph_0}\leq\aleph_1$, then $\card\C{T}^{\aleph_1}=\aleph_1$.
\end{cor}

\providecommand{\bysame}{\leavevmode\hbox to3em{\hrulefill}\thinspace}
\providecommand{\MRhref}[2]{
  \href{http://www.ams.org/mathscinet-getitem?mr=#1}{#2}
}
\providecommand{\href}[2]{#2}

\end{document}